\documentclass[12pt]{amsart}

\usepackage{amsmath}
\usepackage{amssymb}
\usepackage{amsthm}
\usepackage{tabularx}
\usepackage{graphicx}
\usepackage{enumerate}
\numberwithin{equation}{section}

\newtheorem{thm}{Theorem}[section]
\newtheorem{prop}[thm]{Proposition}
\newtheorem{cor}[thm]{Corollary}
\newtheorem{lem}[thm]{Lemma}

\newtheorem{obs}[thm]{Observation} 
\theoremstyle{definition}
\newtheorem{defn}{Definition}[section]
\newtheorem{prob}[thm]{Problem}

\newcommand{\ssm}{\smallsetminus}

\newcommand{\tr}{{\negthickspace \top \negthickspace}}

\newcommand{\spn}{\mathrm{span}}
\newcommand{\type}{\mathrm{type}}
\newcommand{\defeq}{:=} 
\newcommand{\per}{\mathrm{per}}
\newcommand{\qdet}{\mathrm{det}_q}
\newcommand{\qper}{\mathrm{per}_q}
\newcommand{\pavoiding}{$3412$-avoiding, $4231$-avoiding }
\newcommand{\avoidsp}{avoids the patterns $3412$ and $4231${}}
\newcommand{\avoidp}{avoid the patterns $3412$ and $4231${}}
\newcommand{\avoidingp}{avoiding the patterns $3412$ and $4231${}}

\newcommand{\sn}{\mathfrak{S}_n}
\newcommand{\slambda}{\mathfrak{S}_\lambda}
\newcommand{\slambdamin}{\mathfrak{S}_\lambda^{-}}
\newcommand{\mfs}[1]{\mathfrak{S}_{#1}}

\newcommand{\hnq}{H_n(q)}

\newcommand{\zx}{\mathbb{Z}[x]}

\newcommand{\zqq}{\mathbb{Z}[\qp12, \qm12]}

\newcommand{\zsn}{\mathbb{Z}\sn}
\newcommand{\anq}{\mathcal{A}_n(q)}

\newcommand{\quv}{q_{u,v}}

\newcommand{\qvw}{q_{v,w}}

\newcommand{\qev}{q_{e,v}}
\newcommand{\qew}{q_{e,w}}

\newcommand{\qiuv}{q_{u,v}^{-1}}

\newcommand{\qiev}{q_{e,v}^{-1}}
\newcommand{\qiew}{q_{e,w}^{-1}}
\newcommand{\qp}[2]{q^{\frac{#1}{#2}}}
\newcommand{\qm}[2]{q^{\negthinspace\Bar\,\frac{#1}{#2}}}
\newcommand{\qdiff}{\qp12 - \qm12}

\newcommand{\imm}[1]{\mathrm{Imm}_{#1}}

\newcommand{\permmon}[2]{#1_{1,#2_1} \cdots {#1}_{n,#2_n}}
\newcommand{\doublepermmon}[3]{{#1}_{#2_1,#3_1} \cdots {#1}_{#2_n,#3_n}}

\newcommand{\sh}{\mathrm{sh}}

\newcommand{\ctype}[1]{\mathrm{cyc}(#1)}

\newcommand{\mathscr}{\EuScript} 

\newcommand{\inv}{\textsc{inv}}
\newcommand{\rinv}{\textsc{rinv}}

\newcommand{\STAT}{\textsc{stat}}

\newcommand{\asc}{\text{asc}}

\newcommand{\bp}{\begin{prob}}
\newcommand{\ep}{\end{prob}}

\newcommand{\precdot}{\prec \negthickspace \negthinspace \cdot\ }

\newcommand{\ntnsp}{\negthinspace}
\newcommand{\ntksp}{\negthickspace}
\newcommand{\nTksp}{\negthickspace\negthickspace}
\newcommand{\nTtksp}{\negthickspace\negthickspace\negthickspace}

\newcommand{\ssec}[1]{\subsection{#1}{$\negthinspace$}}

\newcommand{\sumsb}[1]{\sum_{\substack{#1}}}  
\newcommand{\prodsb}[1]{\prod_{\substack{#1}}}  

\newcounter{countcases}

\renewcommand{\baselinestretch}{1.0}

\newcommand{\PBS}[1]{\let\temp=\\#1\let\\=\temp}  

\newcommand{\zxn}{\mathbb{Z}[x_{1,1},\dotsc,x_{n,n}]}

\long\def\symbolfootnote[#1]#2{\begingroup%
\def\thefootnote{\fnsymbol{footnote}}\footnote[#1]{#2}\endgroup}

\def\hhhsp{\def\baselinestretch{0.125}\large\normalsize}
\def\hhhpsp{\def\baselinestretch{0.13}\large\normalsize}
\def\ssp{\def\baselinestretch{1.0}\large\normalsize}

\numberwithin{figure}{section}

\def\employer{\@employer}
\def\@employer#1{\noindent \par{\sc #1}}

\begin{document}

\title{Evaluations of Hecke algebra traces at Kazhdan-Lusztig basis elements}  
\date{\today}
\author{Samuel Clearman, Matthew Hyatt, Brittany Shelton, Mark Skandera}



\begin{abstract}
\vspace{.15 in}
For irreducible characters 
$\{ \chi_q^\lambda \,|\, \lambda \vdash n \}$,
induced sign characters
$\{ \epsilon_q^\lambda \,|\, \lambda \vdash n \}$,
and induced trivial characters
$\{ \eta_q^\lambda \,|\, \lambda \vdash n \}$
of the Hecke algebra $\hnq$, 
and Kazhdan-Lusztig basis elements
$C'_w(q)$ with $w$ \avoidingp,
we combinatorially interpret the polynomials
$\smash{\chi_q^\lambda(\qp{\ell(w)}2 C'_w(q))}$,
$\smash{\epsilon_q^\lambda(\qp{\ell(w)}2 C'_w(q))}$,
and
$\smash{\eta_q^\lambda(\qp{\ell(w)}2 C'_w(q))}$.
This leads to a new algebraic interpretation of 
chromatic quasisymmetric functions of Shareshian and Wachs,
and a new combinatorial interpretation of special cases of results of 
Haiman.
We prove similar results 
for other $\hnq$-traces,
and confirm
a formula conjectured by Haiman.
\end{abstract}

\maketitle

\section{Introduction}\label{s:intro}

The {\em symmetric group algebra} $\zsn$
and 
the 
{\em (Iwahori-) Hecke algebra} $\hnq$
have similar presentations as algebras over $\mathbb Z$ and $\zqq$
respectively, with 
multiplicative identity elements $e$ and $T_e$,
generators 
$s_1,\dotsc, s_{n-1}$ and 
$T_{s_1},\dotsc, T_{s_{n-1}}$, 
and relations
\begin{equation*}
\begin{alignedat}{3}
s_i^2 &= e &\qquad
T_{s_i}^2 &= (q-1) T_{s_i} + qT_e &\qquad
&\text{for $i = 1, \dotsc, n-1$},\\
s_is_js_i &= s_js_is_j &\qquad
T_{s_i}T_{s_j}T_{s_i} &= T_{s_j}T_{s_i}T_{s_j} &\qquad
&\text{for $|i - j| = 1$},\\
s_is_j &= s_js_i &\qquad
T_{s_i}T_{s_j} &= T_{s_j}T_{s_i} &\qquad
&\text{for $|i - j| \geq 2$}.
\end{alignedat}
\end{equation*}
Analogous to the {\em natural basis} $\sn$ of $\zsn$ 
is the {\em natural basis} $\{ T_w \,|\, w \in \sn \}$
of $\hnq$,
where we define
$T_w = T_{s_{i_1}} \cdots T_{s_{i_\ell}}$ whenever $s_{i_1} \cdots s_{i_\ell}$
is a reduced (short as possible) expression for $w$ in $\sn$.  
We call $\ell$ the {\em length}
of $w$ and write $\ell = \ell(w)$.  It is known that $\ell(w)$ is equal
to $\inv(w)$, the number of {\em inversions} 
in the one-line notation $w_1 \cdots w_n$ of $w$, i.e., the number of pairs
$(i,j)$ with $i < j$ and $w_i > w_j$.
The specialization of $\hnq$ at $\qp12 = 1$ is isomorphic to $\zsn$.

In addition to the natural bases of $\zsn$ and $\hnq$, we have 
the (signless) {\em Kazhdan-Lusztig bases}~\cite{KLRepCH} 
$\{ C'_w(1) \,|\, w \in \sn \}$,
$\{ C'_w(q) \,|\, w \in \sn \}$, defined in terms of certain 
{\em Kazhdan-Lusztig polynomials} $\{ P_{v,w}(q) \,|\, v,w \in \sn \}$ 
in $\mathbb N[q]$ by
\begin{equation}\label{eq:KLbasis}
C'_w(1) = \sum_{v \leq w} P_{v,w}(1) v, 
\qquad
C'_w(q) = \qiew \sum_{v \leq w} P_{v,w}(q) T_v,
\end{equation}
where $\leq$ denotes the Bruhat order on $\sn$
and we define 
$\qvw = \qp{\ell(w) - \ell(v)}2$.
We have the identity $P_{v,w}(q) = 1$ when $w$ \avoidsp,
i.e., when no subword $w_{i_1} w_{i_2} w_{i_3} w_{i_4}$ of $w_1 \cdots w_n$
consists of letters which appear in the same relative order as 
$3412$ or $4231$~\cite{LakSan}.
These particular permutations are of interest to algebraic geometers
because they correspond to smooth Schubert varieties.
(See \cite[Ch.\,4, Ch.\,13]{BilleyLak}.)


Representations of $\zsn$ and $\hnq$ are often studied in terms of 
linear maps called 
{\em characters}.
(See \cite[Ch.\,1]{Sag} for definitions.)
The 
span
of the $\sn$-characters is called the space of
{\em $\sn$-class functions}, and has
dimension equal to the number of integer partitions of $n$.  
Three well-studied bases 
are the irreducible characters $\{ \chi^\lambda \,|\, \lambda \vdash n \}$,
induced sign characters $\{ \epsilon^\lambda \,|\, \lambda \vdash n \}$, and
induced trivial characters $\{ \eta^\lambda \,|\, \lambda \vdash n \}$,
where $\lambda \vdash n$ denotes that $\lambda$ is a partition of $n$.
Letting $\slambda$ denote the Young subgroup of $\sn$ of type $\lambda$,
we have
\begin{equation*}
\epsilon^\lambda \defeq \mathrm{sgn}\ntnsp\uparrow_{\slambda}^{\sn}, \qquad
\eta^\lambda \defeq \mathrm{triv}\ntnsp\uparrow_{\slambda}^{\sn}\ntnsp.
\end{equation*}
The 
span 
of the $\hnq$-characters, called the space of
{\em $\hnq$-traces}, has the same dimension and analogous character bases
$\{ \chi_q^\lambda \,|\, \lambda \vdash n \}$,
$\{ \epsilon_q^\lambda \,|\, \lambda \vdash n \}$, 
$\{ \eta_q^\lambda \,|\, \lambda \vdash n \}$,
specializing at $\smash{\qp12} = 1$ to the $\sn$-character bases.
Each of the two spaces has a fourth basis consisting of 
{\em monomial} class functions $\{ \phi^\lambda \,|\, \lambda \vdash n \}$ 
or traces $\{ \phi_q^\lambda \,|\, \lambda \vdash n \}$,
and a fifth basis consisting of
{\em power sum} class functions $\{ \psi^\lambda \,|\, \lambda \vdash n \}$
or traces $\{ \psi_q^\lambda \,|\, \lambda \vdash n \}$.
These are defined via the inverse Kostka numbers 
$\{ K_{\lambda,\mu}^{-1} \,|\, \lambda, \mu \vdash n\}$ 
and the numbers
$\{ L_{\lambda,\mu} \,|\, \lambda, \mu \vdash n\}$ 
of row-constant Young tableaux of shape $\lambda$ and content $\mu$ by
\begin{equation}\label{eq:moncfdef}
\phi^\lambda \defeq \sum_\mu K_{\lambda,\mu}^{-1} \chi^\mu,
\quad \phi_q^\lambda \defeq \sum_\mu K_{\lambda,\mu}^{-1} \chi_q^\mu,
\qquad \psi^\lambda \defeq \sum_\mu L_{\lambda,\mu} \phi^\mu,
\quad \psi_q^\lambda \defeq \sum_\mu L_{\lambda,\mu} \phi_q^\mu.
\end{equation}
Each of these functions is not a character, but is a 
difference of two characters.
In each space, 
the five bases are related to one another by the same transition matrices
which relate the
Schur $\{ s_\lambda \,|\, \lambda \vdash n \}$, 
elementary $\{ e_\lambda \,|\, \lambda \vdash n \}$,  
complete homogeneous $\{ h_\lambda \,|\, \lambda \vdash n \}$, 
monomial $\{ m_\lambda \,|\, \lambda \vdash n \}$, 
and power sum $\{ p_\lambda \,|\, \lambda \vdash n \}$ 
bases of the space $\Lambda_n$
of homogeneous degree $n$ symmetric functions.
It follows from the theory of symmetric functions
that basis elements in each space are integer linear combinations
of irreducible characters in that space. 
(For more information on 
the transition matrices and symmetric functions, 
see \cite[Sec.\,2]{RemTrans}, \cite[Ch.\,7]{StanEC2}, respectively.)
A correspondence between these class functions 
and symmetric functions
is given by the Frobenius 
{\em characteristic map} 
\begin{equation*}
\mathrm{ch}(\theta) \defeq \frac1{n!} \sum_{w \in \sn} \theta(w) p_{\mathrm{ctype}(w)}, 
\end{equation*}
where $\mathrm{ctype}(w)$ is the cycle type of $w$.
In particular, we have 
$\mathrm{ch} (\chi^\lambda) = s_\lambda$,
$\mathrm{ch} (\epsilon^\lambda) = e_\lambda$,
$\mathrm{ch} (\eta^\lambda) = h_\lambda$,
$\mathrm{ch} (\phi^\lambda) = m_\lambda$,
$\mathrm{ch} (\psi^\lambda) = p_\lambda$,
which explains our names for the fourth and fifth bases.
We naturally use
the Greek ancestors $\epsilon$, $\eta$ of $e$, $h$
in our class function notation, but we prefer not to use
the ancestors
$\sigma$, $\mu$, $\pi$ of $s$, $m$, $p$.
Instead we follow the standard practices of using $\chi$ for 
irreducible characters, while reserving $\mu$ for integer
partitions, and $\pi$ for path families in planar networks.
We follow \cite{HaimanHecke}, \cite{StemConj} in using
$\phi$,
and 
we use $\psi$ 
because {\em psi} begins with $p$.

 
For any $\sn$-class function $\theta$ belonging to the bases above, 
and for any element $z$ of the natural or Kazhdan-Lusztig basis of $\zsn$,
we have $\theta(z) \in \mathbb Z$.
This follows from the linearity of $\theta$ and the fact that 
$\chi^\lambda(w)$ can be expressed as the trace of an integer matrix.
(See, e.g., \cite[Sec.\,2.3]{Sag}.)
On the other hand, we do not in general have an elementary formula
for the integer $\theta(z)$.
This incomplete understanding of the $\sn$-class functions is unfortunate,
since the functions encode much important information about $\sn$.
For some class functions and basis elements,
we may associate sets $R$, $S$ to the pair $(\theta, z)$
to 
combinatorially interpret
the integer 
$\theta(z)$ as
$(-1)^{|S|}|R|$,
or simply as $|R|$ if $\theta(z) \in \mathbb N$.
We summarize 
results and open problems in the following table.

\hhhsp
\begin{center}
\newcolumntype{R}{>{$}c<{$}}
\begin{tabularx}{162.0mm}{|R|R|R|R|R|}%
\hline
& & & &\\
& & & &\\
\theta
& \theta(w) \in \mathbb N \mathrm{?} 
& \begin{matrix} \mbox{interpretation of} \\ 
\mbox{$\theta(w)^{\phantom{\hat T}}\nTksp$ as $(-1)^{|S|}|R|$?} \end{matrix}
& \theta(C'_w(1)) \in \mathbb N \mathrm{?}  
& \begin{matrix} \mbox{interpretation of} \\ 
\mbox{$\theta(C'_w(1))$ as $|R|^{\phantom{\hat T}}\nTksp$} \\ 
\mbox{for $w^{\phantom{\hat T}}\nTksp$ avoiding} \\
\mbox{$3412^{\phantom{\hat T}}\nTksp$ and $4231$?}\end{matrix} \\
& & & &\\
& & & &\\
\hline 
& & & &\\
& & & &\\
\eta^\lambda & \mbox{yes} & \mbox{yes} & \mbox{yes} & \mbox{yes} \\
& & & &\\
& & & &\\
\epsilon^\lambda & \mbox{no} & \mbox{yes} & \mbox{yes} & \mbox{yes} \\
& & & &\\
& & & &\\
\chi^\lambda & \mbox{no} & \mbox{open} & \mbox{yes} & \mbox{yes}\\
& & & &\\
& & & &\\
\psi^\lambda & \mbox{yes}  & \mbox{yes} & \mbox{yes} & \mbox{yes} \\
& & & &\\
& & & &\\
\phi^\lambda & \mbox{no} & \mbox{yes} & \mbox{conj.\ by Stembridge, Haiman} & \mbox{open} \\
& & & &\\
& & & &\\
\hline
\end{tabularx}
\end{center}
\ssp


For the above 
combinatorial interpretations of $\theta(w)$, 
see 
\cite{RemTrans}.
The number $\chi^\lambda(w)$ may be computed by the well-known algorithm
of Murnaghan and Nakayama (See, e.g., \cite[Ch.\,7]{StanEC2}.)  
but has no conjectured expression 
of the type stated above. 
Interpretations of $\theta(C'_w(1))$ are not known for general $w \in \sn$,
but nonnegativity follows from work of 
Haiman~\cite{HaimanHecke} and Stembridge~\cite{StemImm}.
Interpretations of $\eta^\lambda(C'_w(1))$, 
$\epsilon^\lambda(C'_w(1))$, 
$\chi^\lambda(C'_w(1))$,
$\psi^\lambda(C'_w(1))$,
for $w$ avoiding the 
patterns $3412$, $4231$
follow via straightforward arguments from 
results of the fourth author 
{\cite[Thms.\ 4.3, 5.4]{SkanNNDCB}} and 
others,
notably 
Gasharov~\cite{GashInc},
Karlin-MacGregor~\cite{KMG}, 
Lindstr\"om~\cite{LinVRep}, 
Littlewood~\cite{LittlewoodTGC}, 
Merris-Watkins~\cite{MerWatIneq},
Stanley-Stembridge~\cite{StanStemIJT}, \cite{StemImm}.
These will be discussed in Section~\ref{s:pathposet}.
There is no conjectured combinatorial interpretation of 
$\phi^\lambda(C'_w(1))$, even for $w$ 
\avoidingp,
but
interpretations have been given 
for particular partitions $\lambda$
by Stembridge~\cite{StemConj},
several of the authors~\cite{CSSkanPathTab},
and Wolfgang~\cite{WolfgangThesis}.
The problem of 
interpreting $\theta(C'_w(1))$
when $w$ does not \avoidp\ is open.

Our understanding of $\hnq$-traces is even less complete. 
We know that irreducible $\hnq$-characters 
$\{ \chi_q^\lambda \,|\, \lambda \vdash n \}$
satisfy $\chi_q^\lambda(T_w) \in \mathbb Z[q]$ for all $w \in \sn$, since
$\chi^\lambda_q(T_w)$ can be expressed as the trace of a $\mathbb Z[q]$ 
matrix~\cite{KLRepCH}.
Thus for any element $\theta_q$ of the mentioned $\hnq$-trace bases 
and any element $z \in \spn_{\mathbb Z[q]} \{ T_w \,|\, w \in \sn \}$,
we have $\theta_q(z) \in \mathbb Z[q]$ as well.
For instance, elements of a modified Kazhdan-Lusztig basis 
$\{ \qew C'_w(q) \,|\, w \in \sn \}$ belong to this span.
On the other hand, 
we do not have a general elementary formula
for the polynomial $\theta_q(z)$.
This is unfortunate, since $\hnq$-characters are important in the study of 
$\hnq$ and quantum groups. 
In some cases, we may associate sequences 
$(S_k)_{k \geq 0}$, $(R_k)_{\geq 0}$ of sets to the pair $(\theta_q, z)$
to combinatorially interpret
$\theta_q(z)$ as
$\sum_k (-1)^{|S_k|}|R_k| q^k$,
or simply as $\sum_k |R_k| q^k$ if $\theta_q(z) \in \mathbb N[q]$.
We summarize 
results and open problems in the following table.

\hhhpsp
\begin{center}
\newcolumntype{R}{>{$}c<{$}}
\begin{tabularx}{154.7mm}{|R|R|R|R|R|}%
\hline
& & & &\\
& & & &\\
\theta_q \ntnsp
& \theta_q(T_w) \in \mathbb N[q] \mathrm{?}\ntksp 
& \begin{matrix}\mbox{interpretation of}\\ 
\mbox{$\theta_q(T_w)^{\phantom{\hat T}}\nTksp$ as}\\ 
\mbox{$\sum_k^{\phantom{\hat T}} (-1)^{|S_k|}|R_k|q^k$?}\end{matrix}\ntksp
& \theta_q(\qew C'_w(q)) \in \mathbb N[q] \mathrm{?}  
& \begin{matrix} \mbox{interpretation of}\\ 
\mbox{$\theta_q(\qew C'_w(q))^{\phantom{\hat T}}\nTksp$ as}\\ 
\mbox{$\sum_k^{\phantom{\hat T}}\ntksp|R_k|q^k$ for}\\ 
\mbox{$w^{\phantom{\hat T}}\nTksp$ avoiding}\\ 
\mbox{$3412^{\phantom{\hat T}}\nTksp$ and $4231$?}\end{matrix} \\
& & & &\\
& & & &\\
\hline 
& & & &\\
& & & &\\
\eta_q^\lambda & \mbox{no} & \mbox{open} & \mbox{yes} & \mbox{stated in Section 5}\\
& & & &\\
& & & &\\
\epsilon_q^\lambda & \mbox{no} & \mbox{open} & \mbox{yes} & \mbox{stated in Section 6}\\
& & & &\\
& & & &\\
\chi_q^\lambda \ntksp & \mbox{no} & \mbox{open} & \mbox{yes} & \mbox{stated in Section 8}\\
& & & &\\
& & & &\\
\psi_q^\lambda \ntksp & \mbox{no} & \mbox{open} & \mbox{conj.\ by Haiman} & \mbox{stated in Section 9}\\
& & & &\\
& & & &\\
\phi_q^\lambda \ntksp & \mbox{no} & \mbox{open} & \mbox{conj.\ by Haiman} & \mbox{open} \\
& & & &\\
& & & &\\
\hline
\end{tabularx}
\end{center}


The polynomial $\chi_q^\lambda(T_w)$, 
and therefore all polynomials $\theta_q(T_w)$ above,
may be computed via a $q$-extension of the Murnaghan-Nakayama algorithm,
developed 
in~\cite{KerVerCFR}, \cite{KingWRepTr}, \cite{RamFrob}, \cite{VanderAlg}.
However,
none of these has a conjectured expression of the type stated above.
Interpretations of $\theta_q(\qew C'_w(q))$ 
are not known for general $w \in \sn$,
but results concerning containment in $\mathbb N[q]$ 
follow principally from work of Haiman~\cite{HaimanHecke}.
In all cases above, coefficients of $\theta_q(\qew C'_w(q))$ are symmetric
about $\qew$.
They are also unimodal 
for $\theta_q \in \{ \eta_q^\lambda, \epsilon_q^\lambda, \chi_q^\lambda \}$
and conjectured to be so for $\theta_q \in \{ \psi_q^\lambda, \phi_q^\lambda \}$
\cite[Lem.\,1.1, Conj.\,2.1]{HaimanHecke}.
For $w$ \avoidingp,
formulas for 
$\eta_q^\lambda(\qew C'_w(q))$, 
$\epsilon_q^\lambda(\qew C'_w(q))$, 
$\chi_q^\lambda(\qew C'_w(q))$, and 
$\psi_q^\lambda(\qew C'_w(q))$ 
follow from work of 
Athanasiadis~\cite{AthanPSE},
Gasharov~\cite{GashInc},
Shareshian-Wachs~\cite{SWachsChromQF}, and 
the authors.
There is no conjectured combinatorial interpretation of 
$\phi_q^\lambda(\qew C'_w(q))$, even for $w$ avoiding the 
patterns $3412$, $4231$,
although for particular partitions $\lambda$
interpretations are given by 
the authors in Section~\ref{s:hnqinterpm}.
The problem of combinatorially interpreting $\theta_q(\qew C'_w(q))$
when $w$ does not \avoidp\ is open.


Another way to understand the evaluations $\theta(w)$ 
is to define 
a generating function $\imm{\theta}(x)$ 
for $\{ \theta(w) \,|\, w \in \sn \}$
in the polynomial ring $\zxn$. 
Similarly, we may define
a generating function $\imm{\theta_q}(x)$ 
for $\{ \theta_q(T_w) \,|\, w \in \sn \}$
in a certain noncommutative ring $\anq$. 
In some cases these generating functions have simple forms.
We summarize known results in the following tables.

\hhhsp
\begin{center}
\newcolumntype{R}{>{$}c<{$}}
\begin{tabularx}{68.8mm}{|R|R|}%
\hline
&\\
&\\
&\\
\theta
& \mbox{simple expression for $\imm{\theta}(x)$?} \\
&\\
&\\
\hline 
&\\
&\\
&\\
\eta^\lambda & \mbox{yes} \\
&\\
&\\
&\\
\epsilon^\lambda & \mbox{yes} \\
&\\
&\\
&\\
\chi^\lambda & \mbox{open} \\
&\\
&\\
&\\
\psi^\lambda & \mbox{yes} \\
&\\
&\\
&\\
\phi^\lambda & \mbox{open} \\
&\\
&\\
&\\
\hline
\end{tabularx}
\qquad \qquad
\begin{tabularx}{70.185mm}{|R|R|}%
\hline
&\\
&\\
\theta_q
& \mbox{simple expression for $\imm{\theta_q}(x)$?} \\
&\\
&\\
\hline 
&\\
&\\
\eta_q^\lambda & \mbox{yes} \\
&\\
&\\
\epsilon_q^\lambda & \mbox{yes} \\
&\\
&\\
\chi_q^\lambda & \mbox{open} \\
&\\
&\\
\psi_q^\lambda & \mbox{open} \\
&\\
&\\
\phi_q^\lambda & \mbox{open} \\
&\\
&\\
\hline
\end{tabularx}
\end{center}
\vspace{3mm}
\ssp

Simple expressions for 
$\imm{\eta^\lambda}(x)$
and $\imm{\epsilon^\lambda}(x)$
are due to Littlewood~\cite{LittlewoodTGC} and Merris-Watkins~\cite{MerWatIneq},
and a simple expression for $\imm{\psi^\lambda}(x)$ follows immediately from a
standard definition of $\psi$.  
An expression for $\imm{\chi^\lambda}(x)$ as a coefficient of a 
generating function in two sets of variables was given by 
Goulden-Jackson~\cite{GJMaster}.
There is no conjectured simple formula for 
$\imm{\phi^\lambda}(x)$,
although 
a simple formula 
for particular partitions $\lambda$
was stated by Stembridge~\cite{StemConj}. 
Simple expressions for 
$\imm{\eta_q^\lambda}(x)$
and $\imm{\epsilon_q^\lambda}(x)$
are due to the fourth author and Konvalinka~\cite{KSkanQGJ},
as is a (less simple) 
expression for $\imm{\chi_q^\lambda}(x)$ as a coefficient in a 
generating function in two sets of variables.

In Section~\ref{s:imm} 
we discuss 
generating functions $\imm{\theta}(x)$ 
for $\sn$-class functions $\theta$ 
and 
generating functions $\imm{\theta_q}(x)$ 
for $\hnq$-traces $\theta_q$.
These generating functions belong to the ring
$\zxn$ and to a certain $q$-analog $\anq$ 
of $\zxn$ known as the {\em quantum matrix bialgebra}.
In Section~\ref{s:dsn}
we relate these generating functions to structures 
called {\em zig-zag networks}~\cite{SkanNNDCB},
which serve as combinatorial interpretations for
Kazhdan-Lusztig basis elements indexed by permutations \avoidingp.
In Section~\ref{s:pathposet} we introduce a partial order on paths in
zig-zag networks, and show how this poset and results in the literature lead 
to combinatorial interpretations for the evaluations
of $\sn$-class functions at the above Kazhdan-Lusztig basis elements
of $\zsn$.
For the remainder of the article,
we concentrate on combinatorial interpretations of trace evaulations
of the form $\theta_q(\qew C'_w(q))$ where $w$ \avoidsp.
In Sections~\ref{s:hnqinterph} -- \ref{s:hnqinterpe} 
we 
interpret
$\eta_q^\lambda(\qew C'_w(q))$
and $\epsilon_q^\lambda(\qew C'_w(q))$. 
In Section~\ref{s:chromsf} we recall the relationship between
$\sn$-class functions and Stanley's chromatic symmetric functions,
and prove that $\hnq$-traces are similarly related to the
Shareshian-Wachs chromatic quasisymmetric functions.
In Sections~\ref{s:hnqinterps} -- \ref{s:hnqinterpp} 
we use results of Shareshian-Wachs and Athanasiadis to
interpret 
$\chi_q^\lambda(\qew C'_w(q))$
and $\psi_q^\lambda(\qew C'_w(q))$. 
One of our interpretations proves a formula conjectured by 
Haiman~\cite[Conj.\,4.1]{HaimanHecke}.
Finally, in Section~\ref{s:hnqinterpm} we state several results 
concerning 
$\phi_q^\lambda(\qew C'_w(q))$.

\section{Generating functions for $\theta(w)$ and $\theta_q(T_w)$ when 
$\theta$ ($\theta_q$) is fixed}\label{s:imm}




For a fixed $\sn$-class function $\theta$, we
create a generating function for $\{ \theta(w) \,|\, w \in \sn \}$
by writing a matrix of variables $x = (x_{i,j})_{i,j \in [n]}$
and defining 
\begin{equation*}
\imm \theta(x) \defeq \sum_{w \in \sn} \theta(w) \permmon xw \ \in \ 
\zx \defeq \zxn,
\end{equation*}
where $w_1 \cdots w_n$ is the one-line notation of $w$. 
We call this polynomial the 
{\em $\theta$-immanant}.  The sign character 
($w \mapsto (-1)^{\ell(w)}$)
immanant and trivial character 
($w \mapsto 1$)
immanant are the determinant and permanent,
\begin{equation*}
\det(x) = \sum_{w \in \sn} (-1)^{\ell(w)} \permmon xw,
\qquad
\per(x) = \sum_{w \in \sn} \permmon xw.
\end{equation*}
Simple formulas for the $\epsilon^\lambda$-immanants 
and $\eta^\lambda$-immanants 
employ
determinants and permanents of square submatrices $x_{I,J}$ of $x$,
\begin{equation*}
x_{I,J} \defeq (x_{i,j})_{i \in I, j \in J},
\qquad 
I,J \subset [n] \defeq \{1,\dotsc, n\},
\qquad 
|I| = |J|.
\end{equation*}
In particular, 
we have the 
Littlewood-Merris-Watkins identities~\cite{LittlewoodTGC}, \cite{MerWatIneq} 
\begin{equation}\label{eq:inducedimmssubmat}
\imm{\epsilon^\lambda}(x) = 
\nTtksp
\sum_{(I_1, \dotsc, I_r)} 
\nTksp
\det(x_{I_1, I_1}) \cdots \det(x_{I_r, I_r}), \quad
\imm{\eta^\lambda}(x) = 
\nTtksp
\sum_{(I_1, \dotsc, I_r)} 
\nTksp
\per(x_{I_1, I_1}) \cdots \per(x_{I_r, I_r}),
\end{equation}
where $\lambda = (\lambda_1, \dotsc, \lambda_r) \vdash n$ and
the sums are over 
all sequences 
of pairwise disjoint subsets of $[n]$ satisfying $|I_j| = \lambda_j$.
We will call such a sequence an {\em ordered set partition} of $[n]$ 
{\em of type $\lambda$,} and will sometimes write $I \vdash [n]$
and $\ell(\lambda) = r$.

A simple
formula for the $\psi^\lambda$-immanant 
involves
a sum over all permutations of cycle type $\lambda$.  
Specifically, we have
\begin{equation}\label{eq:pimm}
\imm{\psi^\lambda}(x) = z_\lambda 
\nTksp \sumsb{w\\ \ctype{w} = \lambda} \nTksp
\permmon xw,
\end{equation}
where $z_\lambda$ is the product
$1^{\alpha_1} 2^{\alpha_2} \cdots n^{\alpha_n} \alpha_1! \cdots \alpha_n!$,
and $\lambda$ has $\alpha_i$ components (or {\em parts}) 
equal to $i$ for $i = 1, \dotsc, n$.  
No such simple formulas are known for the 
$\chi^\lambda$-immanants or $\phi^\lambda$-immanants in general,
although Stembridge gave a formula~\cite[Thm.\,2.8]{StemConj}
for $\imm{\phi^\lambda}(x)$ when $\lambda$ is the rectangular partition
$(k)^r = (k,\dotsc,k)$.  In this case we have
\begin{equation}\label{eq:stemmonrect}
\imm{\phi^{(k)^r}}(x) = \sum_{(I_1,\dotsc, I_k)} 
\det(x_{I_1,I_2}) \det(x_{I_2,I_3}) \cdots
\det(x_{I_k,I_1}),
\end{equation}
where the sum is over all ordered set partitions 
of $[n] = [kr]$ 
of type $(r)^k$.
(See \cite{GJMaster} for an expression for $\imm{\chi^\lambda}(x)$ 
as a coefficient of a generating function in two sets of variables.)



For a fixed $\hnq$-trace $\theta_q$, we
create a generating function for $\{ \theta_q(T_w) \,|\, w \in \sn \}$
as before,
except that we interpret
polynomials in 
$x = (x_{i,j})$ 
as elements of
the {\em quantum matrix bialgebra} $\anq$,
the noncommutative $\zqq$-algebra
generated by $n^2$ variables 
$x_{1,1}, \dotsc, x_{n,n}$,
subject to the relations 
\begin{equation}\label{eq:qringrelations}
\begin{alignedat}{2}
x_{i,\ell}x_{i,k} &= \smash{\qp12}x_{i,k} x_{i,\ell}, &\qquad
x_{j,k} x_{i,\ell} &= x_{i,\ell}x_{j,k}\\
x_{j,k}x_{i,k} &= \smash{\qp12}x_{i,k} x_{j,k} &\qquad
x_{j,\ell} x_{i,k} &= x_{i,k} x_{j,\ell} + \smash{(\qdiff)} x_{i,\ell}x_{j,k},
\end{alignedat}
\end{equation}
for all indices $1 \leq i < j \leq n$ and $1 \leq k < \ell \leq n$.
As a $\smash{\zqq}$-module, $\anq$ has a natural basis of monomials 
$x_{\ell_1,m_1} \cdots x_{\ell_r,m_r}$ in which index pairs
appear in lexicographic order.  
The relations (\ref{eq:qringrelations}) 
allow one
to express other monomials in terms of this natural basis.

As a generating function for $\{ \theta_q(T_w) \,|\, w \in \sn \}$,
we define
\begin{equation*}
\imm{\theta_q}(x) \defeq \sum_{w \in \sn} \theta_q(T_w) \qiew \permmon xw
\end{equation*}
in $\anq$, and call this 
the 
{\em $\theta_q$-immanant}.  The $\hnq$ sign character 
($T_w \mapsto (-1)^{\ell(w)}$)
immanant and 
trivial character 
($T_w \mapsto q^{\ell(w)}$)
immanant are called the {\em quantum determinant} and {\em quantum permanent},
\begin{equation*}
\qdet(x) = \sum_{w \in \sn} (-\qm12)^{\ell(w)} \permmon xw,
\qquad
\qper(x) = \sum_{w \in \sn} (\qp12)^{\ell(w)} \permmon xw.
\end{equation*}
Specializing $\anq$ at $\smash{\qp12}=1$ gives
the commutative polynomial ring $\zx$, 
with elements $\qdet(x)$ and $\qper(x)$ 
specializing to
the classical determinant $\det(x)$ and permanent $\per(x)$.
Simple formulas for the $\epsilon_q^\lambda$-immanants 
and $\eta_q^\lambda$-immanants 
employ
quantum 
determinants and 
quantum 
permanents of submatrices
of $x$.  
In particular, the fourth author and Konvalinka~\cite[Thm.~5.4]{KSkanQGJ}
proved quantum analogs of the Littlewood-Merris-Watkins identities 
(\ref{eq:inducedimmssubmat}),
\begin{equation}\label{eq:qinducedimmssubmat}
\ntksp
\imm{\epsilon_q^\lambda}(x) 
= 
\nTtksp
\sum_{(I_1,\dotsc,I_r)} 
\nTtksp
\qdet(x_{I_1,I_1}) \cdots \qdet(x_{I_r,I_r}),
\quad
\imm{\eta_q^\lambda}(x) 
= 
\nTtksp
\sum_{(I_1,\dotsc,I_r)} 
\nTtksp
\qper(x_{I_1,I_1}) \cdots \qper(x_{I_r,I_r}),
\ntksp 
\end{equation}
where the sums are as in (\ref{eq:inducedimmssubmat}).
See \cite{KSkanQGJ} for an expression for the
$\chi_q^\lambda$-immanant as a coefficient of a generating function 
in two sets of variables.
No formulas are known for the 
$\psi_q^\lambda$- 
or $\phi_q^\lambda$-
immanants.
It would be interesting to state a $q$-analog of (\ref{eq:stemmonrect}) 
for rectangular partitions $\lambda = (k)^r$.

\section{Planar 
networks
and path matrices}\label{s:dsn}

Call a directed planar graph $G$ a
{\em planar network of order $n$} if it
is acyclic and 
may be embedded in a disc with $2n$ boundary vertices labeled
clockwise as 
{\em source $1, \dotsc,$ source $n$} 
(with indegrees of $0$) and 
{\em sink $n, \dotsc,$ sink $1$}
(with outdegrees of $0$).
In figures, we will draw 
sources on the left and sinks on the right,
implicitly labeled $1, \dotsc, n$ 
from bottom to top.
Edges will be implicitly oriented from left to right.
Given a planar network $G$, 
define the {\em path matrix} $B = B(G) = (b_{i,j})$ of $G$ by
\begin{equation}\label{eq:pathmatrix}
b_{i,j} = \text{ number of paths in $G$ from source $i$ to sink $j$}.
\end{equation}
The 
path matrix of any planar network is totally nonnegative,
i.e., 
each square submatrix has nonnegative determinant.
Specifically, for sets
$I = \{ i_1, \dotsc, i_k \}$, $J = \{ j_1, \dotsc, j_k \} \subset [n]$,
and the corresponding submatrix
$B_{I,J} \defeq (b_{i,j})_{i \in I, j \in J}$, 
we have that $\det(B_{I,J})$ is equal to the number of families 
$(\pi_1, \dotsc, \pi_k)$ of mutually nonintersecting paths from
sources $i_1, \dotsc, i_k$ (respectively) 
to sinks $j_1, \dotsc, j_k$ (respectively).
This fact is known as {\em Lindstr\"om's Lemma}~\cite{LinVRep}.
(See also \cite{KMG}.)
We will call two planar networks $G_1$, $G_2$ {\em isomorphic} 
and will write $G_1 \cong G_2$ 
if $B(G_1) = B(G_2)$.

For example, consider two isomorphic planar networks,
their common path matrix $B$ and its submatrix $B_{23,13}$:
\begin{equation*}
\raisebox{-8mm}{
\includegraphics[height=18mm]{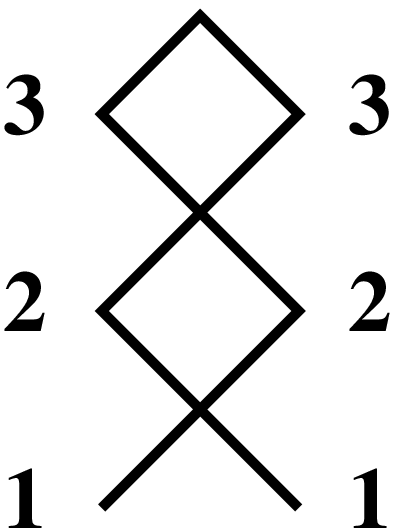}}\ ,
\qquad
\raisebox{-8mm}{
\includegraphics[height=18mm]{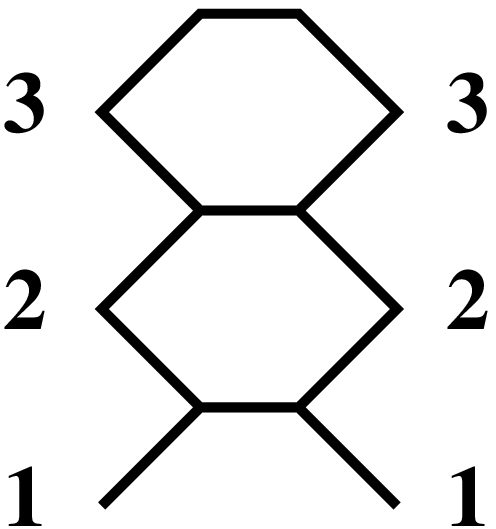}}\ ,
\qquad
B = \begin{bmatrix}
1 & 1 & 0\\
1 & 2 & 1\\
0 & 1 & 2
\end{bmatrix},
\qquad
B_{23,13} = \begin{bmatrix}
1 & 1\\
0 & 2
\end{bmatrix}.
\end{equation*}
We can interpret $\det(B_{23,13}) = 2$ as counting the two path families 
\begin{equation*}
\raisebox{-8mm}{
\includegraphics[height=18mm]{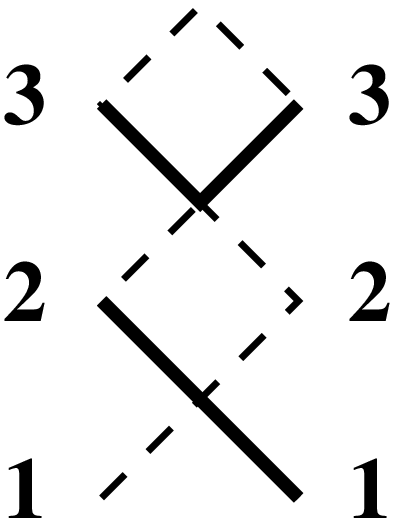}}\ ,
\qquad
\raisebox{-8mm}{
\includegraphics[height=18mm]{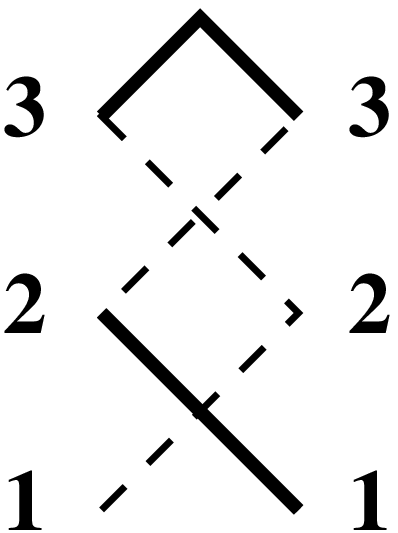}}\ ,
\end{equation*}
from sources $\{2, 3\}$ to sinks $\{1, 3\}$ in the first network.

An easy fact about planar networks is the following.
\begin{obs}\label{o:pnetintersect}
Let $G$ be a planar network of order $n$ 
and assume that for some indices $i, i', j, j' \in [n]$,
$G$ contains a path $\pi_i$ from source $i$ to sink $j$ and 
a path $\pi_{i'}$ from source $i'$ to sink $j'$.  
If these two paths intersect,
then $G$ also contains a path from source $i$ to sink $j'$ and
a path from source $i'$ to sink $j$.  
If $i' < i$ and $j < j'$ then the paths $\pi_i$ and $\pi_{i'}$ cross.
\end{obs}


For $[a,b]$ a subinterval of $[n]$,
let $G_{[a,b]}$
be the planar network
consisting of $a-1$ horizontal edges
from sources $1, \dotsc, a-1$ to corresponding sinks,
a ``star'' of $b-a+1$ 
edges from sources $a, \dotsc, b$ to an intermediate vertex
and $b-a+1$ more edges from this vertex to sinks $a,\dotsc,b$,
and $n-b$ more horizontal edges
from sources $b+1,\dotsc,n$ to corresponding sinks. 
For $n = 4$, there are seven such networks:
$G_{[1,4]}$, 
$G_{[2,4]}$, 
$G_{[1,3]}$, 
$G_{[3,4]}$, 
$G_{[2,3]}$, 
$G_{[1,2]}$, 
$G_{[1,1]} = \cdots = G_{[4,4]}$, respectively,
\begin{equation*}
\raisebox{-6mm}{
\includegraphics[height=12mm]{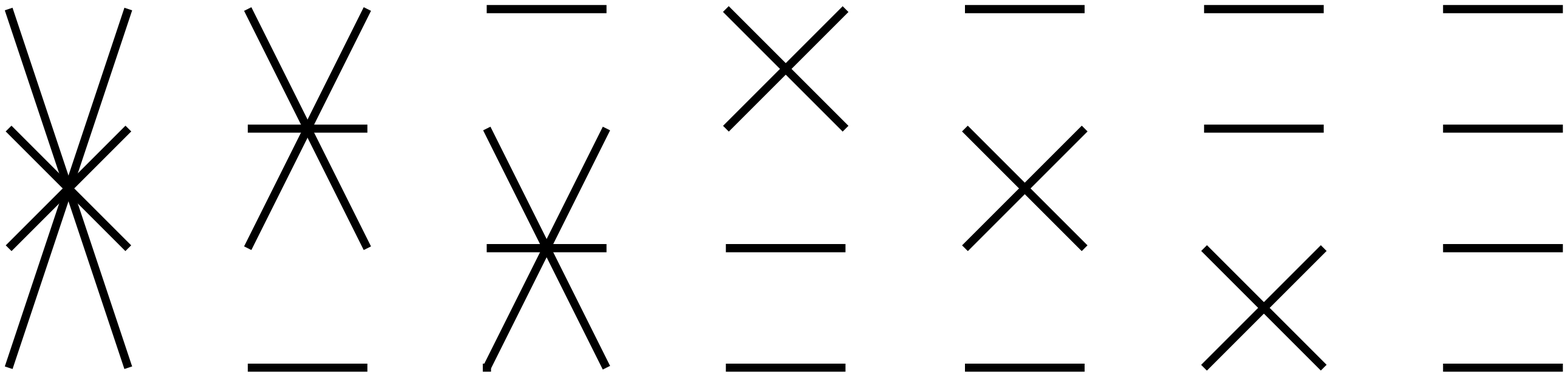}}.
\end{equation*}

Given planar networks $G$, $H$ of order $n$, in which all sources have
outdegree $1$ and all sinks have indegree $1$, define
$G \circ H$ to be the concatenation of $G$ and $H$, 
where for $i = 1,\dotsc, n$, sink $i$ of $G$ is dropped,
source $i$ of $H$ is dropped, and 
the unique edge in $G$ from vertex $x$ to sink $i$ 
and the unique edge in $H$ from source $i$ to vertex $y$ 
are merged to form a single edge from $x$ to $y$ in $G \circ H$.
Note that for star networks $G_{[c_1,d_1]}$, $G_{[c_2,d_2]}$ 
indexed by nonintersecting intervals, the two concatenations
$G_{[c_1,d_1]} \circ G_{[c_2,d_2]}$ and 
$G_{[c_2,d_2]} \circ G_{[c_1,d_1]}$ are isomorphic.
 
We will be interested in concatenations
\begin{equation}\label{eq:concat}
G = G_{[c_1,d_1]} \circ \cdots \circ G_{[c_t,d_t]}
\end{equation}
such that
\begin{enumerate}
\item the sequence $([c_1,d_1], \dotsc, [c_t,d_t])$
consists of $t$ distinct, pairwise nonnesting intervals, 
\item for $i < j < k$, 
if $[c_i,d_i] \cap [c_j,d_j] \neq \emptyset$
and $[c_j,d_j] \cap [c_k,d_k] \neq \emptyset$, then we have
$c_i < c_j < c_k$ 
(and $d_i < d_j < d_k$) 
or $c_i > c_j > c_k$
(and $d_i > d_j > d_k$).
\end{enumerate}
We define a relation $\precdot$
on the set of intervals 
appearing 
in the concatenation (\ref{eq:concat})
by declaring $[c_i,d_i] \precdot [c_k,d_k]$ 
if 
\begin{enumerate}
\item $i < k$, 
\item $[c_i,d_i] \cap [c_k,d_k] \neq \emptyset$, 
\item $[c_i,d_i] \cap [c_j,d_j] \cap [c_k,d_k] = \emptyset$ 
for $j = i+1, \dotsc, k-1$.
\end{enumerate}
This is the covering relation of a partial order $\preceq$.

For each planar network $G$ 
of the form
(\ref{eq:concat})
and satisfying the conditions following (\ref{eq:concat}), we 
define a related planar network $F$ by 
considering each covering 
pair $[c_i,d_i] \precdot [c_j,d_j]$,
with $| [c_i,d_i] \cap [c_j,d_j] | = k$, 
and deleting all but one of the $k$ paths 
from the central vertex of $G_{[c_i,d_i]}$ to the central vertex of $G_{[c_j,d_j]}$.
Following \cite{SkanNNDCB}, we call the resulting network $F$ 
a {\em zig-zag network}.
In the special case that we have
$c_1 > \cdots > c_t$,
we call $F$ a {\em descending star network}.
The descending star networks (up to isomorphism) of order $4$ are
\begin{equation}\label{eq:xfigures2}
\raisebox{-6mm}{
\includegraphics[height=12mm]{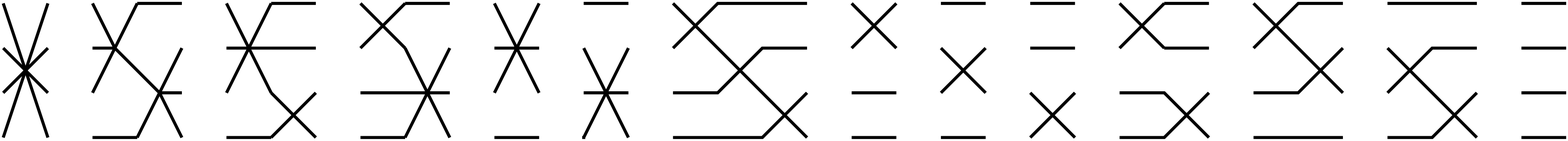}}\ .
\end{equation}
The zig-zag networks 
of order $4$ 
which are not descending star networks are
\begin{equation}\label{eq:xfigureszz}
\raisebox{-6mm}{
\includegraphics[height=12mm]{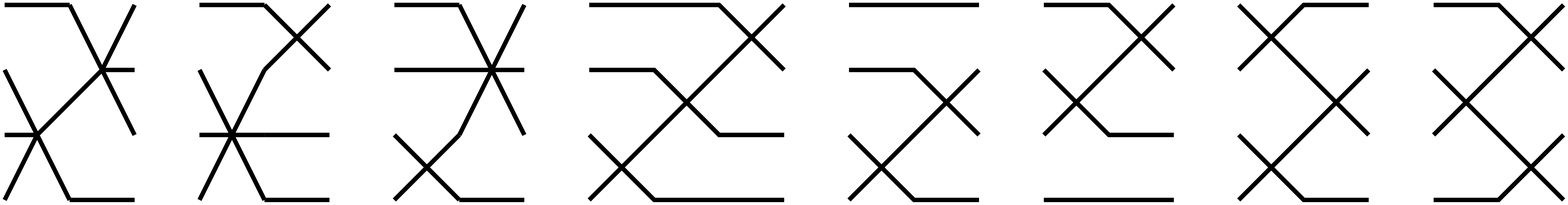}}\ .
\end{equation}

It is easy to see that if $F$ is a zig-zag network of order $n$,
then there is at most one interval in the concatenation 
(\ref{eq:concat}) containing $n$, and this interval must be
maximal or minimal (or both) in the partial order $\preceq$.

It was shown in \cite[Thm.\ 3.5, Lem.\ 5.3]{SkanNNDCB} 
that zig-zag networks of order $n$ correspond bijectively to 
\pavoiding permutations in $\sn$.  To summarize this bijection,
we let $F$ be the zig-zag 
network obtained from 
the concatenation $G$ in (\ref{eq:concat}), and construct
another concatenation of star networks as follows.
For $i = 1, \dotsc, t-1$, if the interval $[c_i,d_i]$ is covered by $[c_j,d_j]$
in the order $\preceq$ and if $|[c_i,d_i] \cap [c_j,d_j]| > 1$, then insert 
$G_{[c_i,d_i] \cap [c_j,d_j]}$ 
immediately after $G_{[c_i,d_i]}$ 
in the current concatenation.  
(If $[c_i,d_i]$ is also covered by a second interval $[c_k,d_k]$, 
then $G_{[c_i,d_i] \cap [c_k,d_k]}$ may be inserted before 
or after $G_{[c_i,d_i] \cap [c_j,d_j]}$.)
Call the resulting augmented network $G'$.
Now visually follow paths from sources to sinks in $G'$, 
passing ``straight'' through each star,
to create a \pavoiding permutation $w \in \sn$. 
See \cite[Sec.\,3]{SkanNNDCB} for a description of the inverse
of this bijection, which maps $w$ to $G'$.
We will let $F_w$ and $G'_w$ denote the zig-zag network 
and augmented star network corresponding to a fixed
\pavoiding permutation $w$, and we will let $w(F)$ denote the
\pavoiding permutation corresponding to a fixed zig-zag network $F$.

For example, 
suppose that 
$F$ is the zig-zag network obtained from the concatenation
$G = G_{[3,7]} \circ G_{[5,8]} \circ G_{[8,9]} \circ G_{[1,2]} \circ G_{[2,4]}$
of star networks of order $9$.
Drawing the poset $\preceq$ on these intervals from left to right, we have
\begin{equation*}
\raisebox{-12mm}{
\includegraphics[height=40mm]{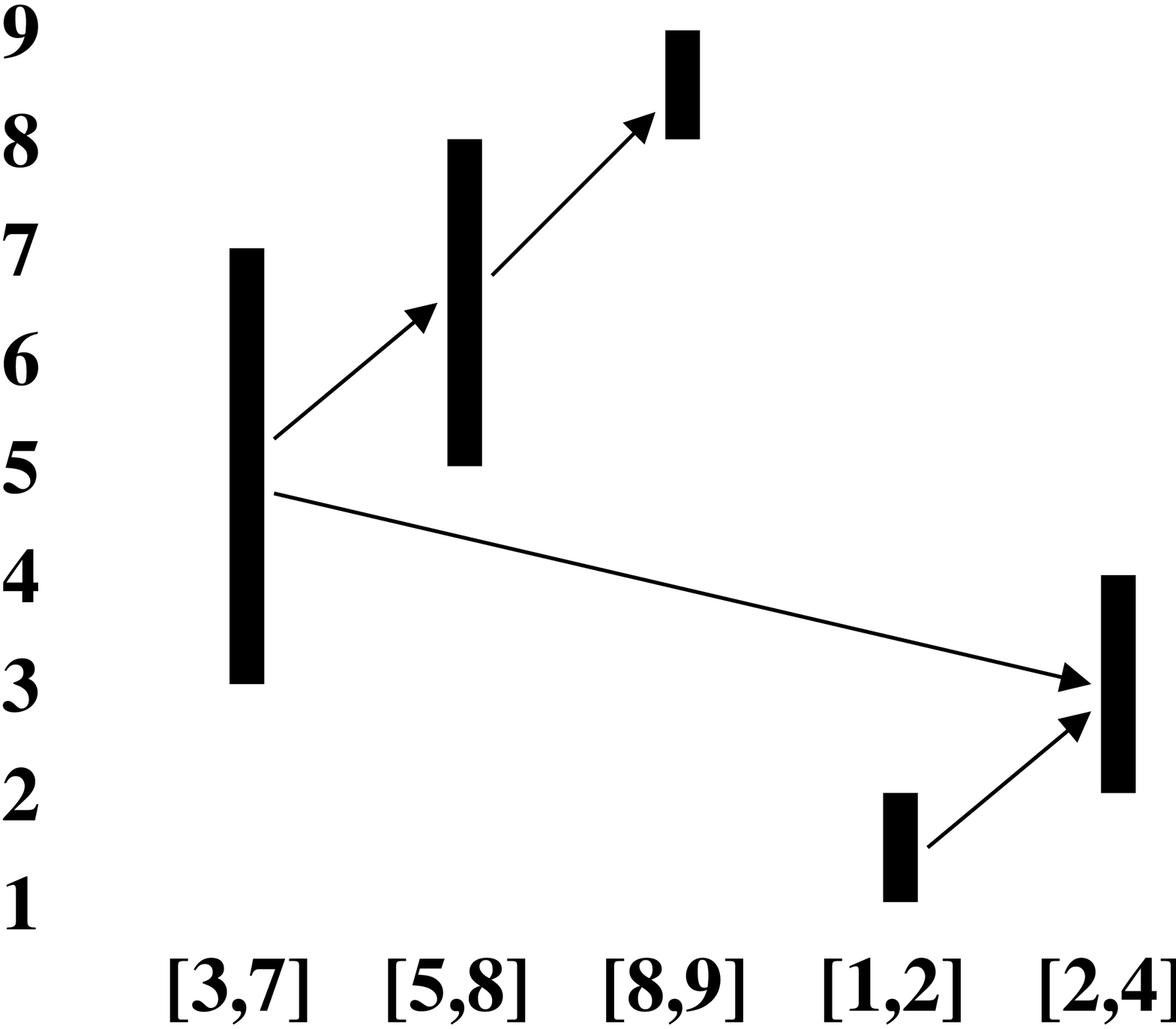}}\ .
\end{equation*}
Since the only covering pairs which intersect at more than an endpoint are
$[3,7] \precdot [5,8]$ and $[3,7] \precdot [2,4]$, we construct
$G'$ by inserting
$G_{[3,7] \cap [5,8]} = G_{[5,7]}$ and $G_{[3,7] \cap [2,4]} = G_{[3,4]}$ after $G_{[3,7]}$.
Thus we have
$G' = G_{[3,7]} \circ G_{[5,7]} \circ G_{[3,4]} \circ G_{[5,8]} \circ G_{[8,9]} \circ G_{[1,2]} \circ G_{[2,4]}$,
and we obtain the \pavoiding permutation $w = w(F) = 419763258$:
\begin{equation*}
F = \raisebox{-17mm}{
\includegraphics[height=34mm]{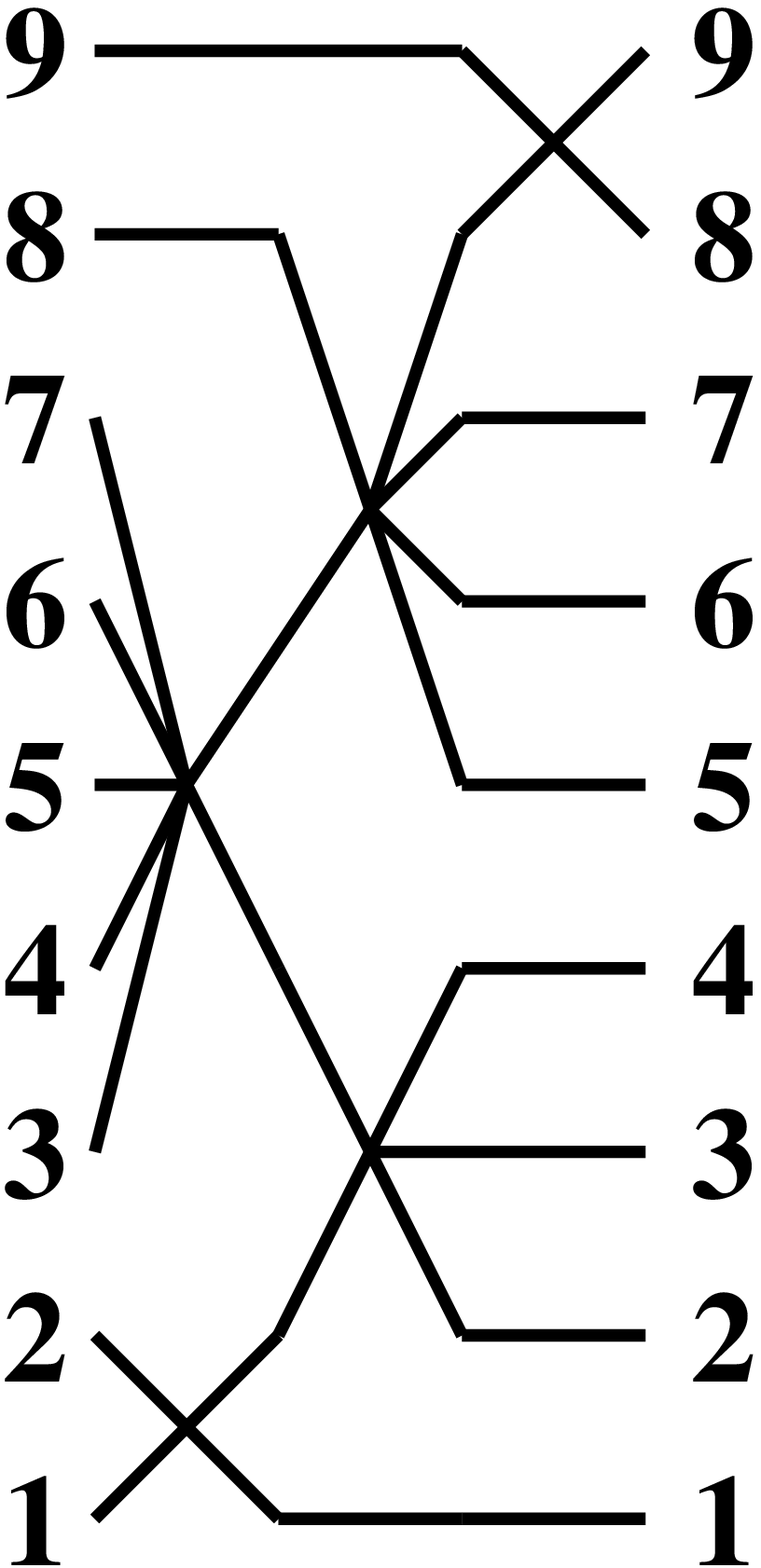}}\ , \quad
G = \raisebox{-17mm}{
\includegraphics[height=34mm]{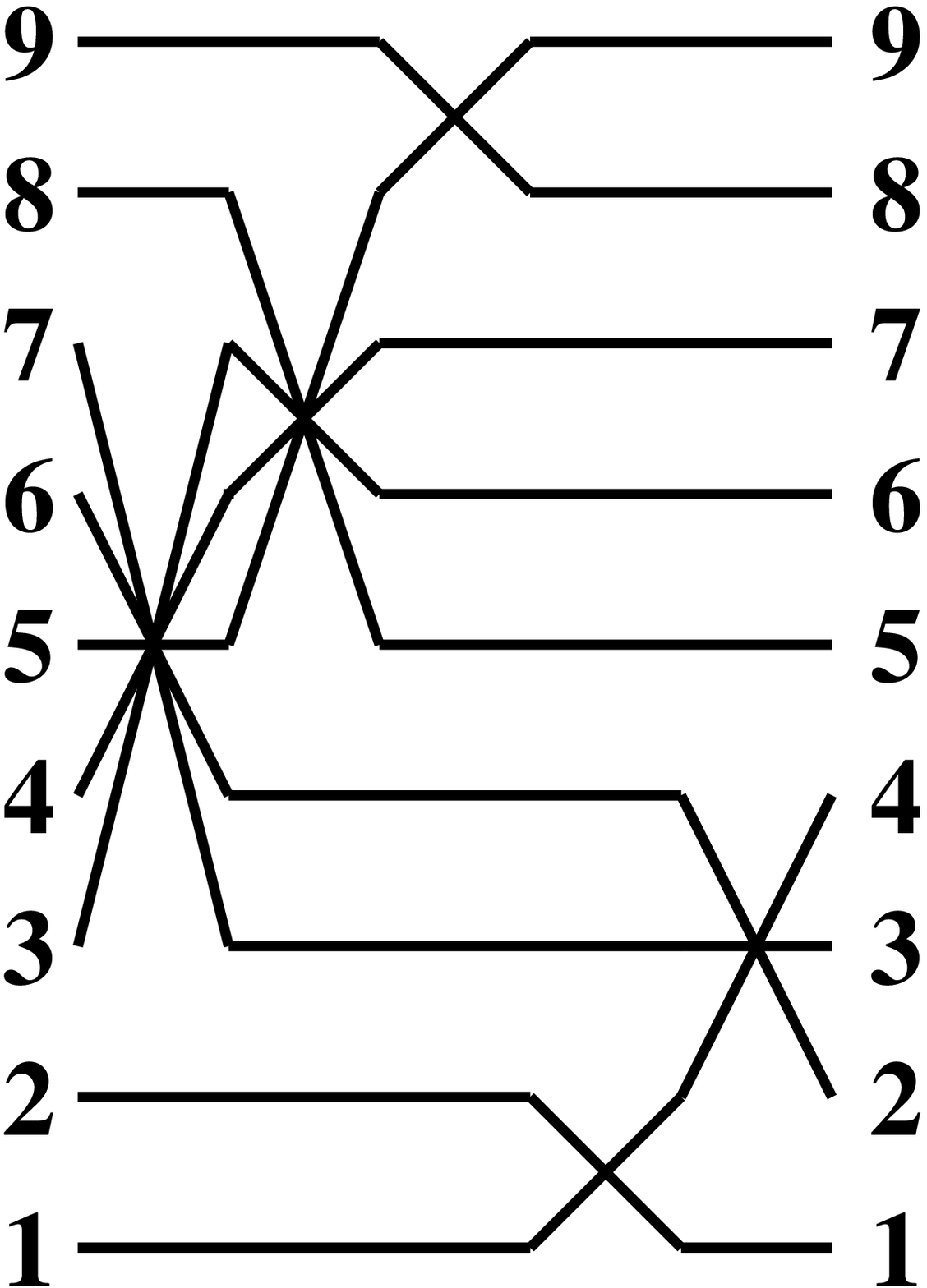}}\ , \quad
G' = \raisebox{-17mm}{
\includegraphics[height=34mm]{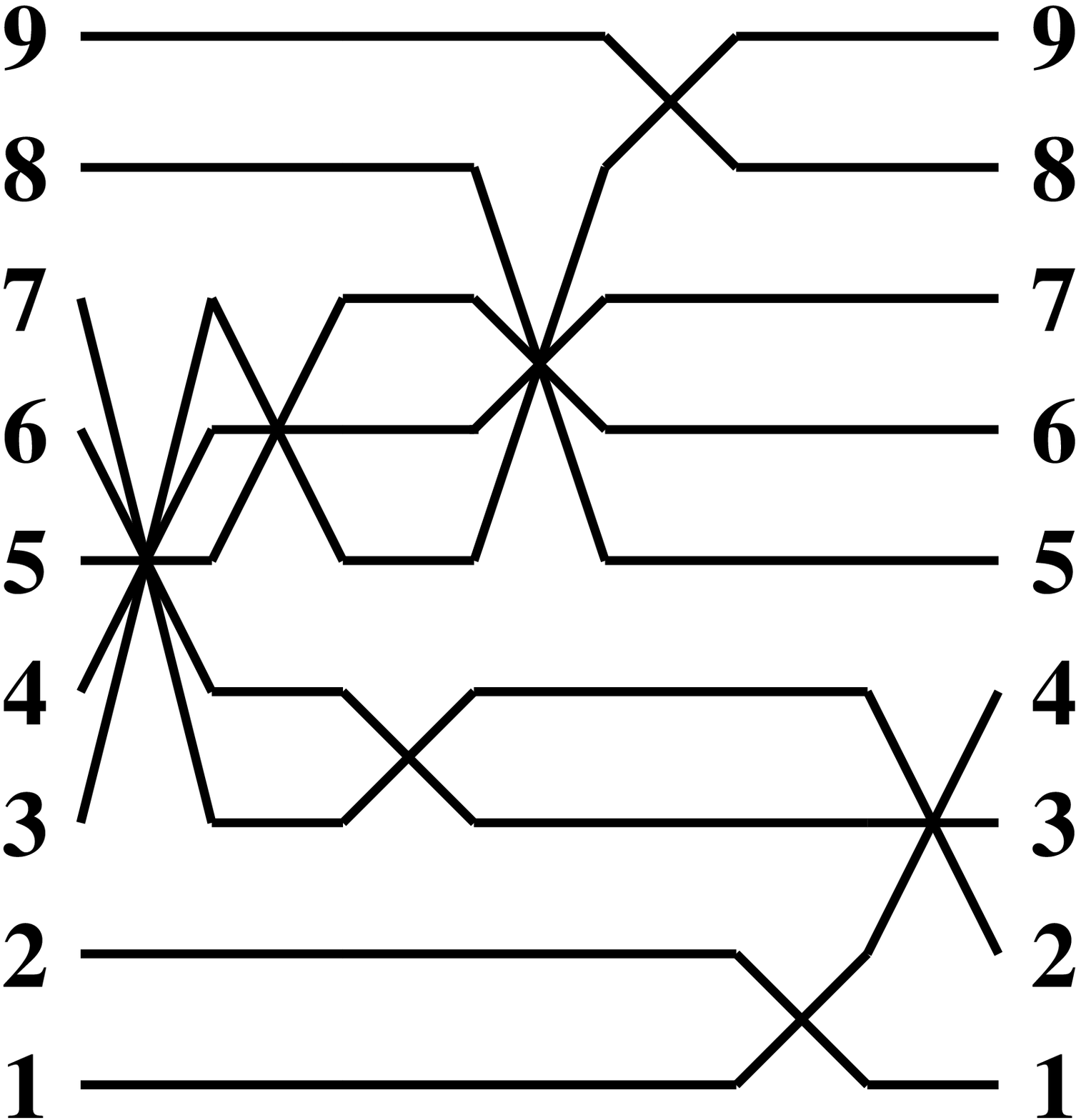}}\ , \quad
w = \binom{123456789}{419763258}.
\end{equation*}

Note that we have $G_{[3,4]} \circ G_{[5,7]} \cong G_{[5,7]} \circ G_{[3,4]}$,
since the intervals $[3,4]$, $[5,7]$ do not overlap.
In general, we have the following.
\begin{obs}\label{o:firstorlast}
If $[c_{i_1}, d_{i_1}], \dotsc, [c_{i_t},d_{i_t}]$ is a linear extension
of the poset $\preceq$ defined in terms of the concatenation 
(\ref{eq:concat}), then we have
\begin{equation*}
G_{[c_1,d_1]} \circ \cdots \circ G_{[c_t,d_t]} 
\cong
G_{[c_{i_1},d_{i_1}]} \circ \cdots \circ G_{[c_{i_t},d_{i_t}]},
\end{equation*}
and the corresponding zig-zag networks are isomorphic as well.
\end{obs} 

Call a sequence $\pi = (\pi_1, \dotsc, \pi_n)$ of source-to-sink paths
in a planar network $G$ of order $n$ a {\em path family}.
We will always assume that path $\pi_i$ begins at source $i$.
If for some $w \in \sn$ with one-line notation $w_1 \cdots w_n$, 
each component path $\pi_i$ 
terminates at sink $w_i$, 
we will say that $\pi$ has {\em type $w$} 
and we will write $\mathrm{type}(\pi) = w$.  
If the union of the paths of $\pi$ is equal to $G$,
we will say that $\pi$ {\em covers} $G$.
For example, the planar network $G_{[2,4]} \circ G_{[1,3]}$ 
can be covered by many different path families, 
including two of type $e$, and two of type $2341$:
\begin{equation}\label{eq:cover}
\raisebox{-10mm}{
\includegraphics[height=20mm]{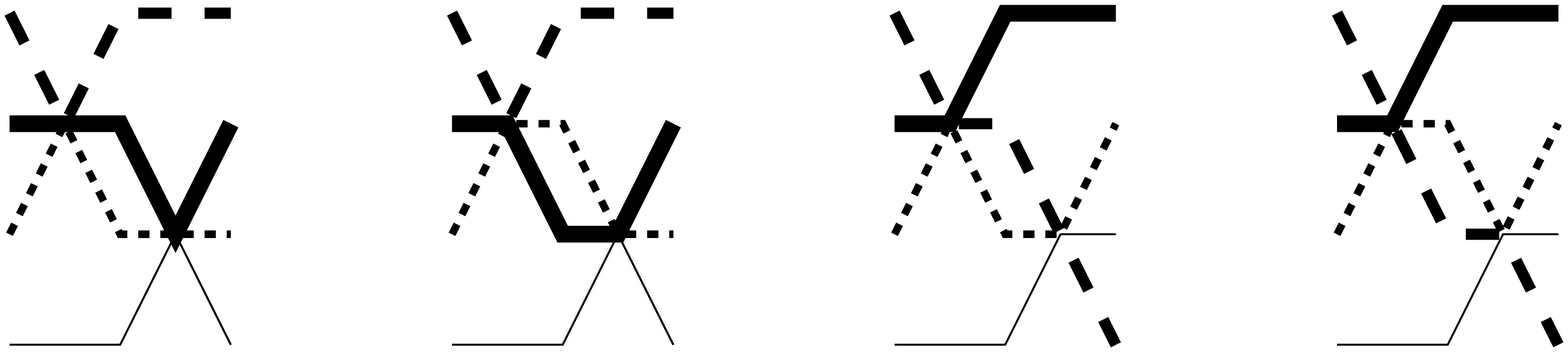}}.
\end{equation}

The result \cite[Lem.\,5.3]{SkanNNDCB} states that 
the path families covering a zig-zag network $F_w$ 
correspond bijectively to elements of
a principal order ideal in the Bruhat order:
\begin{thm}\label{t:lem53nndcb}
Fix $v, w \in \sn$ with $w$ \avoidingp.
There is a unique path family of type $v$ covering $F_w$ if and only if
we have $v \leq w$.  Otherwise, there is no such path family.
\end{thm}
Thus $w$ is the unique Bruhat-maximal permutation for which some path
family of type  $w$ covers $F_w$.
It follows also that there is exactly one path in $F_w$ from any source $i$ 
to the corresponding sink $i$, and 
at most one path 
from source $i$ to sink $j \neq i$. 
Whether or not such a path exists 
may be determined by the intervals in
the corresponding concatenation of star networks 
and the partial order $\preceq$.

\begin{obs}\label{o:sourcetosink}
Let $F$ be a zig-zag network which corresponds to the concatenation
(\ref{eq:concat}).  There exists a path in $F$ from source $i$ to sink $j$
if and only if we have one of the following.
\begin{enumerate}
\item $i, j \in [c_k,d_k]$ for some $k$.
\item $i \in [c_k, d_k]$, $j \in [c_\ell, d_\ell]$ and 
$[c_k,d_k] \prec [c_\ell, d_\ell]$ for some $k, \ell$.
\end{enumerate}
\end{obs}
It is easy to see that if $F$ is a descending star 
network with sources $i$ and $j$ belonging to the same connected component, 
then the second condition in Observation~\ref{o:sourcetosink} 
is equivalent to the inequality $i > j$.  
This fact allows us to refine Observation~\ref{o:pnetintersect} slightly.
\begin{lem}\label{l:dsnintersect}
Let $\pi_{\ell_1}$, $\pi_{\ell_2}$ be paths in a descending star network $F$
from sources $\ell_1 < \ell_2$ to sinks $m_1, m_2$, respectively.
Then the paths $\pi_{\ell_1}$ and $\pi_{\ell_2}$ intersect if and only if
there exists a path in $F$ from source $\ell_1$ to sink $m_2$.
\end{lem}
\begin{proof}
Assume that $F$ has order $n$ and 
corresponds to the concatenation (\ref{eq:concat}) of star networks.  
Let $x_1, \dotsc, x_t$ be the vertices in $F$
corresponding to the central vertices of the $t$ star networks in 
(\ref{eq:concat}), and 
for any index $j \in [n]$, 
let $f(j)$ and $g(j)$ denote the indices of the 
first and last intervals, respectively,
in $[c_1, d_1], \dotsc, [c_t, d_t]$ to contain $j$.
Then for any indices $\ell, m \in [n]$, 
the unique path in $F$ from source $\ell$ to sink $m$
contains the vertices $x_{f(\ell)}, \dotsc, x_{g(m)}$.

It is easy to see that if the intersection 
of $\pi_{\ell_1}$ and $\pi_{\ell_2}$ is nonempty, then there is a path
in $F$ from source $\ell_1$ to sink $m_2$.
Now suppose that there is a path in $F$ 
from source $\ell_1$ to sink $m_2$.
By Observation~\ref{o:sourcetosink},
we have $m_2 < \ell_1$ or we have an interval $[c_k, d_k]$ containing
both $\ell_1$ and $m_2$.
Either case implies that we have $f(\ell_2) \leq f(\ell_1) \leq g(m_2)$,
and the paths $\pi_{\ell_1}$ and $\pi_{\ell_2}$ share the vertex $x_{f(\ell_1)}$.
\end{proof}

The subset of zig-zag networks which are descending star networks
can be characterized using pattern avoidance.
\begin{thm}\label{t:dsn312}
Let $v \in \sn$ \avoidp.  Then $v$ avoids the pattern $312$ if and only if
$F_v$ is a descending star network.
\end{thm}
\begin{proof}
Let $G$ be the concatenation (\ref{eq:concat}) of star networks which
leads to $F_v$, and define $G'$ as in the example above.

Suppose first that $F_v$ is not a descending star network.
Then in the concatenation (\ref{eq:concat}) 
there exists an index
$i$ such that the interval $[c_i,d_i]$ is minimal in the poset $\preceq$,
and an index $j > i$ such that 
$c_i < c_j$ and $[c_i,d_i] \precdot [c_j,d_j]$. 
Considering the relationship between $G'$ and $v$, one sees 
that we have $v_{d_i} < c_j$ and that for some index
$\ell \geq d_j$ we have $v_\ell = c_j$.
On the other hand, since $G'$ is constructed by inserting $G_{[c_j,d_i]}$ 
between $G_{[c_i,d_i]}$ and $G_{[c_j,d_j]}$,
we have that $v_{c_i} \geq d_j$.
Thus the indices $c_i < d_i < \ell$ satisfy 
$v_{d_i} < v_{\ell} < v_{c_i}$
and $v$ does not avoid the pattern $312$.


Now suppose that $F_v$ is a descending star network.
We claim that if $F_v$ arises from the concatenation (\ref{eq:concat}),
then $v$ avoids the pattern $312$ and satisfies $v_{c_1} > \cdots > v_{d_1}$.
By inspection of (\ref{eq:xfigures2}), 
this is true for descending star networks of orders $1$ -- $4$.
Now assume this to be true for
each descending star network $F_w$ of order $1, \dotsc, n-1$, and 
let $F_v$ be a descending star network of order $n$. 
If in the first interval of (\ref{eq:concat})
we have $d_1 < n$, then $v_n = n$ and $F_{v_1 \cdots v_{n-1}}$ 
is a descending star network of order $n-1$.
Thus $F_v$ has the claimed properties.
If in the first interval of (\ref{eq:concat}) we have $d_1 = n$, then
consider the descending star network $F_w$ arising from the concatenation
$G_{[c_2,d_2]} \circ \cdots \circ G_{[c_t,d_t]}$.
Then
$v = s_{[c_1,d_1]} s_{[c_1,d_2]} w$, where $s_{[a,b]}$ is the unique Bruhat-maximal
permutation for which a path family of this type 
covers the star network $G_{[a,b]}$.
By 
the above argument, 
$w$ avoids the pattern $312$, and 
satisfies $w_{c_2} > \cdots > w_{d_2}$, 
and $w_i = i$ for $i = d_2+1,\dotsc, n$.
It follows that the letters in positions $c_1, \dotsc, n$ of $s_{[c_1,d_2]} w$
form an increasing sequence, and that $v$ satisfies 
$v_{c_1} > \cdots > v_n$.
Thus the subword $v_1 \cdots v_{c_1-1} = w_1 \cdots w_{c_1-1}$ 
avoids the pattern $312$, as does the subword $v_{c_1} \cdots v_n$.
If $v$ contains any subword $z_3z_1z_2$ that matches the pattern $312$,
then $z_3z_1$ must be a subword of 
$v_1 \cdots v_{c_1-1}$ while $z_2$ is a letter of the decreasing word
$v_{c_1} \cdots v_n$.
But in this case, $z_3z_1$ is a subword of $w_1 \cdots w_{c_1-1}$ while
$z_2$ is a letter of $w_{c_1} \cdots w_n$, which contradicts our assumption
that $w$ avoids the pattern $312$.
\end{proof}

It follows that there are 
$\tfrac 1{n+1}\tbinom{2n}n$ descending star networks of order $n$,
since there are this many $312$-avoiding permutations in $\sn$,
all of which \avoidp.
Also related to pattern avoidance
are the sizes of the stars in the concatenation (\ref{eq:concat}).

\begin{thm}\label{t:kstar}
Let $w \in \sn$ \avoidp.
Then $w_1 \cdots w_n$ contains a decreasing subsequence of size $k$
if and only if some interval in the concatenation (\ref{eq:concat})
corresponding to $F_w$ has cardinality at least $k$.
\end{thm}
\begin{proof}
Let $w_{i_1}, \dotsc, w_{i_k}$ be a decreasing subsequence of $w$.
Then there is a path in $F_w$ from source $i_j$ to sink $w_{i_j}$
for $j = 1,\dotsc,k$.  These paths pairwise intersect, since
$i_r < i_s$ if and only if $w_{i_r} > w_{i_s}$.
Since $F_w$ is acyclic, these paths must all intersect at a single vertex.
Such a vertex must correspond to the central vertex of 
a star network indexed by an interval of cardinality at least $k$
in the concatenation (\ref{eq:concat}).

The converse is clearly true if we have $t = 1$ in (\ref{eq:concat}).
Suppose that the converse holds for each zig-zag network corresponding
to a concatenation of $t-1$ star networks.
Let $F_w$ correspond to a concatenation (\ref{eq:concat}) of
$t$ star networks, and let $F_v$ correspond to the concatenation
$G_{[c_1, d_1]} \circ \cdots \circ G_{[c_{t-1}, d_{t-1}]}$.
Suppose that some interval $[c_i, d_i]$, $1 \leq i \leq t$, 
has cardinality at least $k$.
If $i \leq t-1$, then $v$ contains a decreasing subsequence
of size $k$.  
By \cite[Cor.\,3.7]{SkanNNDCB}, there is a reduced expression for $w$ 
which consists of a reduced expression for $v$, followed by some reduced
expression $s_{i_1} \cdots s_{i_k}$ for the permutation $v^{-1}w$.
It is well known that each permutation in the sequence 
$(v, vs_{i_1}, vs_{i_1}s_{i_2}, \dotsc, w)$ preserves all inversions of the previous
permutation and introduces one more.  It follows that $w$ also has a
decreasing subsequence of length $k$.
If $i = t$, then apply the above argument to $w^{-1}$, which corresponds
to the concatenation $G_{[c_t, d_t]} \circ \cdots \circ G_{[c_1, d_1]}$.
It is well known that $w$ has a decreasing subsequence of length $k$
if and only if $w^{-1}$ does.
\end{proof}

We may use path matrices of zig-zag networks to evaluate
$\sn$-class functions at Kazhdan-Lusztig basis elements 
$\{ C'_w(1) \,|\, w \text{ \avoidsp} \}$.
Specifically, if $B$ is the path matrix of $F_w$, then 
by \cite[Sec.\,4, Thm.\,5.4]{SkanNNDCB} we have
\begin{equation}\label{eq:thetaCimmB}
\theta(C'_w(1)) = \imm \theta (B).
\end{equation}
This fact is a crucial ingredient in the proofs of 
Theorems~\ref{t:thetavw}, \ref{t:sncfinterp},
which interpret the evaluations of certain $\sn$-class functions
at Kazhdan-Lusztig basis elements of $\zsn$.

In order to interpret the evaluations of $\hnq$-traces 
at Kazhdan-Lusztig basis elements of $\hnq$,
we will prove a $q$-extension of Equation (\ref{eq:thetaCimmB}) 
in Proposition~\ref{p:qimmchar}.
Namely, we will show that path matrices can also be used to  
evaluate $\hnq$-traces
at the (modified) Kazhdan-Lusztig basis elements 
$\{ \qew C'_w(q) \,|\, w \text{ \avoidsp} \}$.
A bit of care is required though: 
the evaluation $\imm{\theta_q}(B)$ does not make sense because
the substitution $x_{i,j} \mapsto b_{i,j}$
does not respect the relations (\ref{eq:qringrelations})
and therefore does 
not give a well-defined map from $\anq$ to $\zqq$.
Thus we define 
a $\zqq$-linear map for each $n \times n$ integer matrix $B$ by
\begin{equation*}
\begin{aligned}
\sigma_B: \anq &\rightarrow \zqq\\
\permmon xv &\mapsto \qev \permmon bv.
\end{aligned}
\end{equation*}
\begin{prop}\label{p:qimmchar}
Let $\theta_q$ be an $\hnq$-trace
and let $w \in \sn$ \avoidp.
Then the path matrix $B$ of $F_w$ 
satisfies
\begin{equation*}
\theta_q(\qew C'_w(q)) = \sigma_B(\imm{\theta_q}(x)).
\end{equation*}
\end{prop}
\begin{proof}
The right-hand side is equal to
\begin{equation*}
\sigma_B \Big( \sum_{v \in \sn} \theta_q(T_v) \qiev \permmon xv \Big) 
= \sum_{v \in \sn} \theta_q(T_v) \permmon bv. 
\end{equation*}
By Theorem~\ref{t:lem53nndcb} (\cite[Lem.~5.3]{SkanNNDCB}), 
the product $\permmon bv$ is $1$ 
when $v \leq w$ and is $0$ otherwise.  Thus the above 
expression is equal to
\begin{equation*}
\sum_{v \leq w}\theta_q(T_v) 
= \theta_q \Big( \sum_{v \leq w} T_v \Big).
\end{equation*}
Since $w$ \avoidsp, \cite{LakSan} implies that the Kazhdan-Lusztig polynomials 
$\{ P_{v,w}(q) \,|\, v \leq w \}$ are identically 1.
(See 
also 
\cite[Ch.~6]{BilleyLak}.)
Comparing to (\ref{eq:KLbasis}), we see that the parenthesized sum is equal to 
$\qew C'_w(q)$. 
\end{proof}
Stated from another point of view, Proposition~\ref{p:qimmchar} asserts
that for $w$ \avoidingp, the zig-zag network $F_w$ combinatorially encodes
the modified Kazhdan-Lusztig basis element $\qew C'_w(q)$ in the sense that
\begin{equation*}
\qew C'_w(q) = \sum_{\pi} T_{\mathrm{type}(\pi)},
\end{equation*}
where the sum is over all path families which cover $F_w$.

\section{Path posets, planar network tableaux and 
interpretation of $\sn$-class function evaluations}
\label{s:pathposet}

\ssec{Path posets}
In a planar network $G$ of order $n$, the source-to-sink paths
have a natural partial order $Q = Q(G)$.  Given paths $\pi_i$, $\rho_j$,
originating at sources $i$, $j$, respectively,
we define $\pi_i <_Q \rho_j$ if 
$i < j$ and $\pi_i$ and $\rho_j$ do not intersect.
Let $P(G)$ be the subposet of $Q(G)$ induced by paths
whose source and sink indices are equal.
For each zig-zag
network $F_w$, the poset $P(F_w)$ has exactly
$n$ elements:  
there is exactly one path from source $i$ to sink $i$, for $i = 1,\dotsc,n$.

The posets $P(G)$ corresponding to the descending star networks $G$ in 
(\ref{eq:xfigures2}) are
\begin{equation}\label{eq:uio}
\raisebox{-3mm}{
\includegraphics[width=130mm]{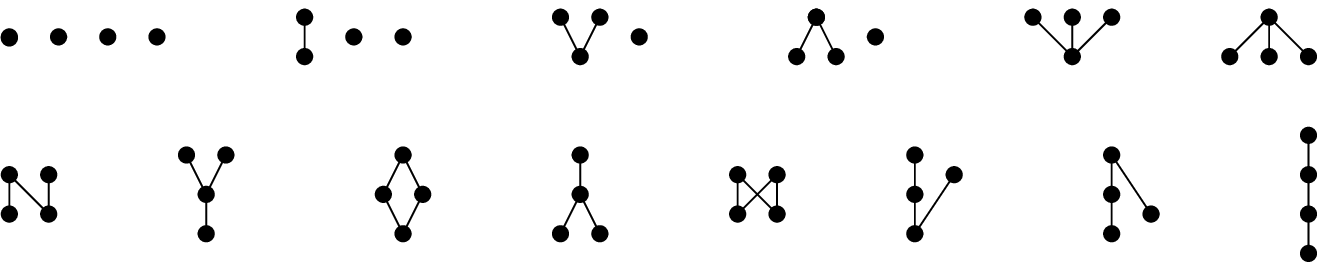}}\ .
\end{equation}
These are precisely the 
$(\mathbf{3}+\mathbf{1})$-free,
$(\mathbf{2}+\mathbf{2})$-free posets on four elements,
where
{\em $(\mathbf{a}+\mathbf{b})$-free} means that
no induced subposet is a  
disjoint union of an $a$-element chain and a $b$-element chain.
Such posets are also called {\em unit interval orders}.
It is known that unit interval orders may be naturally labeled by the
numbers $[n]$ so that $i <_P j$ implies $i < j$ as integers
and so that the conditions $i < j < k$ and $i$ incomparable to $k$ in $P$
imply that $\{ i, j, k \}$ is an antichain in $P$~\cite[Sec.\,3]{DK}.
(See \cite[Prop.\,2.4]{SkanReed} for the algorithm.)
It is known that there are $\tfrac 1{n+1}\tbinom{2n}n$ unit interval 
orders on $n$ elements~\cite[Sec.\,4]{DK}.



\begin{thm}\label{t:dsnuio}
The posets 
$\{ P(F_v) \,|\, v \in \sn \text{ avoids } 312 \}$
are precisely the 
unit interval orders on $n$ elements.
\end{thm}
\begin{proof}
Let $\mathcal P$ and $\mathcal U$ be the two sets of posets in the
theorem.
We define a map $\zeta: \mathcal U \rightarrow \mathcal P$
as follows.
By the above fact on incomparability, 
we may naturally label $P \in \mathcal U$ by
$1, \dotsc, n$ so that its maximal antichains are
$[a_1, b_1], \dotsc, [a_t,b_t]$, 
with $a_1 > \cdots > a_t$.
Now define $F$ to be the descending star network
corresponding to the concatenation 
$G_{[a_1,b_1]} \circ \cdots \circ G_{[a_t, b_t]}$,
and let $(\pi_1, \dotsc, \pi_n)$ be the unique path family of type $e$ that
covers $F$.  Let $\zeta(P)$ be the poset $P(F)$.

We claim that for each poset $P \in \mathcal U$, 
the map $i \mapsto \pi_i$ is an isomorphism of 
$P$ and $\zeta(P)$.
To see this, note that $i <_P j$ if and only if $i < j$ as integers
and $i, j$ belong to no common 
antichain 
in $P$,  
i.e., if and only if $i < j$ and $i,j$ belong to no common interval
$[a_k,b_k]$ defining the concatenation 
$G_{[a_1,b_1]} \circ \cdots \circ G_{[a_t,b_t]}$ of star networks.
But this is true if and only if we have $\pi_i <_{\zeta(P)} \pi_j$.
It follows that the 
$\tfrac 1{n+1}\tbinom{2n}n$ images $\{ \zeta(P) \,|\, P \in \mathcal U \}$
are pairwise nonisomorphic.
On the other hand, by Theorem~\ref{t:dsn312}, we have
$|\mathcal P| \leq \tfrac 1{n+1}\tbinom{2n}n$. 
The claim follows.
%
%
\end{proof}

While a star network is covered by a unique path family of type $e$,
a concatenation of star networks need not be.
Nevertheless, such a concatenation is covered by a unique noncrossing path
family of type $e$: the concatenation of the unique path families of type
$e$ that cover the component star networks.  
Given a concatenation $G$ of star networks, 
let $\hat P(G)$ be the $n$-element subposet of $P(G)$
induced by the paths in the unique noncrossing path family of type $e$
which covers $G$.
For example, the first two figures in (\ref{eq:cover}) show 
the two path families of type $e$ which cover $G_{[2,4]} \circ G_{[1,3]}$.  
Call these 
$(\pi_1, \pi_2, \pi_3, \pi_4)$ 
and 
$(\pi_1, \pi'_2, \pi'_3, \pi_4)$, respectively.
Ordering all six of these source-$i$-to-sink-$i$ paths, 
or only the mutually noncrossing paths $\{ \pi_1, \pi_2, \pi_3, \pi_4 \}$, 
we form the posets
\begin{equation*}
P(G_{[2,4]} \circ G_{[1,3]}) = 
\raisebox{-9mm}{
\includegraphics[height=18mm]{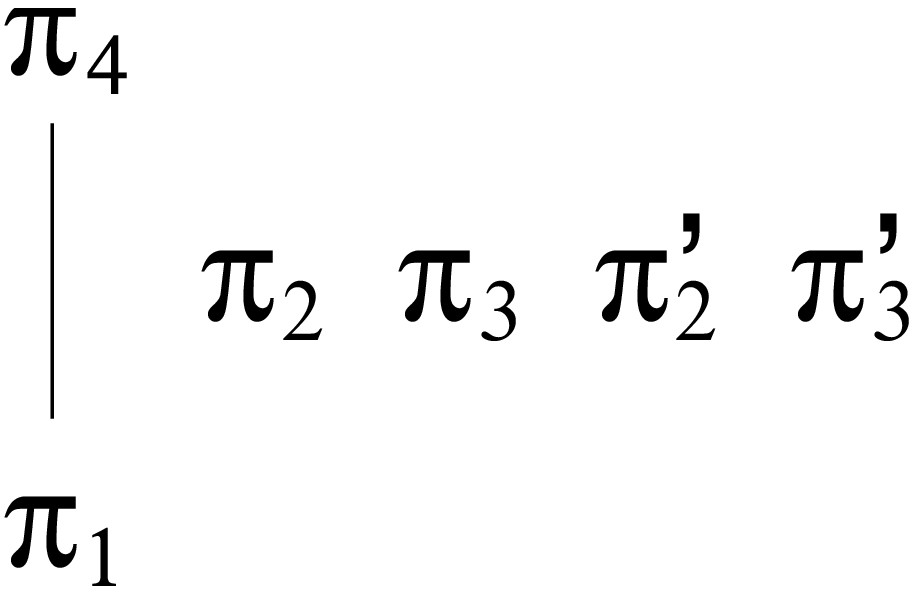}}, \qquad
\hat P(G_{[2,4]} \circ G_{[1,3]}) = 
\raisebox{-9mm}{
\includegraphics[height=18mm]{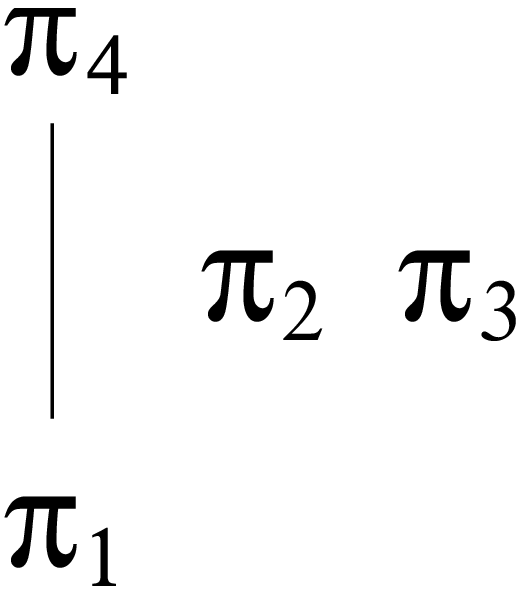}}, 
\end{equation*}
respectively.

It is easy to show that 
for a concatenation $G$ of star networks,
the poset $\hat P(G)$ does not
depend upon the ordering of the factors of $G$.
\begin{prop}\label{p:permisom}
Let $[a_1,b_1],\dotsc, [a_t,b_t]$ be 
subintervals of $[n]$.
Then for any permutation $u \in \mfs t$, we have
$\hat P(G_{[a_1,b_1]} \circ \cdots \circ G_{[a_t,b_t]}) \cong
\hat P(G_{[a_{u_1},b_{u_1}]} \circ \cdots \circ G_{[a_{u_t},b_{u_t}]})$.
\end{prop}
\begin{proof}
Define $G = G_{[a_1,b_1]} \circ \cdots \circ G_{[a_t,b_t]}$,
$H = G_{[a_{u_1},b_{u_1}]} \circ \cdots \circ G_{[a_{u_t},b_{u_t}]}$,
and let $(\rho_1, \dotsc, \rho_n)$, $(\tau_1, \dotsc, \tau_n)$
be the unique noncrossing path families of type $e$ 
covering $G$, $H$, respectively.
For $i = 1, \dotsc, t$, let $(\pi_1^{(i)}, \dotsc, \pi_n^{(i)})$ 
be the unique noncrossing path family of type $e$ covering 
$G_{[a_i,b_i]}$.
For $j, k \in [n]$, the definition of $\hat P(G)$ implies we have 
$\rho_j \leq_{\hat P(G)} \rho_k$ if and only if
$\pi_j^{(i)} \leq_{\hat P(G_{[a_i,b_i]})} \pi_k^{(i)}$ 
for $i = 1, \dotsc, t$.
But this condition is also equivalent to 
$\smash{\tau_j \leq_{\hat P(H)} \tau_k}$.
\end{proof}


It is also easy to show that the path poset of a zig-zag network $F$
may be obtained directly from the concatenation of star networks 
which lead to $F$ as in Section~\ref{s:dsn}.
\begin{prop}\label{p:hatisom}
Let $F$ be a zig-zag network constructed from a concatenation $G$ of
star networks as after (\ref{eq:concat}).
Then $P(F)$ is isomorphic to $\hat P(G)$.
\end{prop}
\begin{proof}
Let $(\rho_1, \cdots, \rho_n)$ and $(\tau_1,\dotsc,\tau_n)$ be the unique
noncrossing path families of type $e$ in $G$ and $F$, respectively,
and let $x_1,\dotsc,x_t$ be the vertices of $F$ which correspond to the
central vertices of the $t$ star networks in (\ref{eq:concat}).
For $i < j$, we have $\rho_i <_{\hat P(G)} \rho_j$ if and only if $\rho_i$, 
$\rho_j$ share none of the central vertices of the $t$ star networks.
But this condition holds if and only if $\tau_i$, $\tau_j$ share none of
the vertices $x_1, \dotsc, x_t$.
\end{proof}

By Propositions~\ref{p:permisom} and \ref{p:hatisom},
it is possible to have an isomorphism of path posets for nonisomorphic 
zig-zag networks.  Define an equivalence relation $\sim$ on \pavoiding
permutations by 
\begin{equation}\label{eq:posetequiv}
v \sim w \quad \text{ if and only if } \quad P(F_v) \cong P(F_w).
\end{equation}
For example, it is easy to see the equivalence of the four permutations 
corresponding to the 
seventh descending star network in (\ref{eq:xfigures2}) 
and the fourth, seventh, and eighth zig-zag networks in (\ref{eq:xfigureszz}):
in each case the path poset is isomorphic to the seventh unit interval order in
(\ref{eq:uio}).
It is also easy to see that we have $w \sim w^{-1}$ for $w$ \avoidingp:
the networks $F_w$ and $F_{w^{-1}}$ differ only by reflection in a vertical line. 
 
\begin{thm}\label{t:pfvpfu}
Each equivalence class of the relation $\sim$ (\ref{eq:posetequiv})
contains exactly one representative which avoids the pattern $312$.
\end{thm}
\begin{proof}
Fix $w \in \sn$ \avoidingp, let 
$G = G_{[c_1,d_1]} \circ \cdots \cdots \circ G_{[c_t,d_t]}$ 
be the concatenation of star networks which leads to the zig-zag network 
$F_w$ as after (\ref{eq:xfigureszz}),
and let $u \in \mfs t$ be the unique permutation
satisfying $c_{u_1} > \cdots > c_{u_t}$.
Then the concatenation 
$G_{[c_{u_1},d_{u_1}]} \circ \cdots \cdots \circ G_{[c_{u_t},d_{u_t}]}$ leads to
a descending star network $F_v$.  
By Theorem~\ref{t:dsn312}, $v$ avoids the pattern $312$,
and by Propositions~\ref{p:permisom} and \ref{p:hatisom} we have
$P(F_v) \cong P(F_w)$.
By Theorem~\ref{t:dsnuio}, $v$ is the only $312$-avoiding permutation in
its equivalence class.
\end{proof}

\ssec{Planar network tableaux}
To combinatorially interpret evaluations of 
$\sn$-class functions and $\hnq$-traces,
we will repeatedly fill a (French) Young diagram 
with a path family
$\pi = (\pi_1,\dotsc,\pi_n)$ 
covering a zig-zag network
$F_w$,
and will call the resulting structures {\em $F_w$-tableaux},
or more specifically {\em $\pi$-tableaux}.
(See, e.g., \cite[Sec.\,2]{RemTrans} for French notation.)
If 
$\pi$ 
has type $v$,
then we will also say that each $\pi$-tableau has {\em type $v$.}
Since $\pi$ 
induces a subposet $Q_\pi$ of the poset $Q(F_w)$,
$\pi$-tableaux form a special case of Gessel and Viennot's 
{\em $P$-tableaux}~\cite{GV}: they are
$Q_\pi$-tableaux.

Several properties which $\pi$-tableaux may posess can be defined 
for $P$-tableaux where $P$ is 
an arbitrary poset.
We say that a $P$-tableau $U$ has {\em shape $\lambda$} 
for some partition $\lambda = (\lambda_1,\dotsc,\lambda_r)$ if it
has $\lambda_i$ cells in row $i$ for all $i$.
If $U$ has $\lambda_i$ cells in {\em column $i$} for all $i$, 
we say that $U$ has shape $\lambda^\tr$, 
where 
we define $\lambda^\tr$ 
to be the partition whose $i$th part is equal to
the number of cells in row $i$ of $U$.
Call an element $x \in P$
a {\em nontrivial record} in a row of $U$
if it is greater in $P$ than all entries
appearing to its left in the same row,
and if it is not the leftmost entry of its row.
\begin{itemize}
\item Call $U$
{\em column-strict} ({\em row-strict}\,)
if whenever elements $x$, $y$ appear consecutively from
bottom to top in a column (left to right in a row), then we have 
$x <_P y$.
\item Call $U$
{\em row-semistrict}
if whenever elements $x$, $y$
appear consecutively from left to right in a row,
then we have $x <_P y$ or $x$ incomparable to $y$ in $P$.
\item Call $U$
{\em cyclically row-semistrict}
if it is row-semistrict and the above condition 
also holds
when $x$, $y$
are the rightmost and leftmost (respectively) 
entries in a row.
\item Call $U$
{\em standard}
if it is both column-strict and row-semistrict.
\item Call $U$
{\em record-free} 
if no row contains a nontrivial record.
\end{itemize}
Another property of $\pi$-tableaux depends upon each path being labeled by
its source vertex.
Call a row of $U$
{\em left anchored} ({\em right anchored}\,)
if its leftmost (rightmost) entry has the least 
source vertex
of all paths in the row.
\begin{itemize}
\item Call $U$
{\em left-anchored} ({\em right-anchored}\,) 
if each row is left-anchored (right-anchored).
\end{itemize}
More
properties of $\pi$-tableaux depend upon the fact that
each element of the poset $Q_\pi$ has a source vertex and a 
(potentially different) sink vertex.
Given a $\pi$-tableau $U$,  
let $L(U)$ and $R(U)$ denote the Young tableaux of integers
obtained from $U$ by replacing paths $\pi_1,\dotsc,\pi_n$
with their corresponding source and sink indices, respectively.
\begin{itemize}
\item 
Call $U$ {\em row-closed} 
if each row of $R(U)$ is a permutation of the corresponding row of $L(U)$.
\item 
Call $U$ {\em left row-strict} 
if $L(U)$ is row-strict as a $\mathbb Z$-tableau. 
\item Call $U$ {\em cylindrical} if
for each row of $L(U)$ containing indices
$i_1, \dotsc, i_k$ from left to right, 
the corresponding row of $R(U)$ contains
$i_2, \dotsc, i_k, i_1$ from left to right.
\end{itemize}

It will be convenient to let 
$\mathcal T(F_w,\lambda)$ 
be the set of $F_w$-tableaux of shape $\lambda$.

\begin{lem}\label{l:dropparens}
Let $v \in \sn$ avoid the pattern $312$, and fix $\lambda \vdash n$.
Within the set $\mathcal T(F_v,\lambda)$, tableaux 
which are row-closed and left row-strict
correspond bijectively to 
tableaux which are row-semistrict of type $e$.
\end{lem}
\begin{proof}
Observe that for a row-closed, left row-strict tableau
$V \in \mathcal T(F_v,\lambda)$, 
the tableau $V_k$ is itself a row-closed, left row-strict tableau
in $\mathcal T(F_v|_{I_k}, \lambda_k)$, where $I_k$ is the set of indices
appearing in $L(V_k)$.
Similarly, for a row-semistrict tableau $V' \in \mathcal T(F_v,\lambda)$
of type $e$, the tableau $V'_k$ is itself a row-semistrict tableau of type $e$ 
in $\mathcal T(F_v|_{I_k}, \lambda_k)$.
Letting $r = \lambda_k$, we have that
$F_v|_{I_k}$ is a descending star network of the form $F_u$
for some $u \in \mfs r$ avoiding the pattern $312$.
We therefore construct a bijection which preserves the 
set of path indices appearing in each row of a tableau,
and we state it as a product of bijections on one-rowed tableaux.

Map each  
left row-strict tableau $U \in \mathcal T(F_u,(r))$ 
to a row-semistrict tableau $U'$ of type $e$ 
in $\mathcal T(F_u,(r))$ as follows.
Let $\rho = (\rho_1, \dotsc, \rho_r)$ be the unique path family of 
type $e$ covering $F_u$.
Let 
$U$ contain
the path family $\pi = (\pi_1, \dotsc, \pi_r)$ from left to right, 
and define $w \in \mfs r$ to be the word of right indices of these paths.
Write $w^{-1}$ in cycle notation, with each cycle starting with its 
greatest element, and cycles ordered by increasing greatest elements.
(See Cycle Structure subsection in \cite[Sec.\,1.3]{StanEC1}).  
Drop the parentheses, and interpret the resulting
string $x = x_1 \cdots x_r$ of letters as the one-line notation of an
element of $\mfs r$.  Then write the paths $\rho_{x_1}, \dotsc, \rho_{x_r}$
from left to right in $U'$.

To see that $U'$ is row-semistrict, assume that we have 
$\rho_{x_i} >_{P(F_u)} \rho_{x_{i+1}}$ for some $i$.  Then there is no
path from source $x_{i+1}$ to sink $x_i$ in $F_u$.
If $x_i$ and $x_{i+1}$ belong to the same cycle of $w^{-1}$,
then $w_{x_{i+1}} = x_i$ and there is a path from $x_{i+1}$ to $x_i$ in $F_u$,
a contradiction. 
If $x_i$ and $x_{i+1}$ do not belong to the same cycle of $w^{-1}$,
then $x_{i+1} > x_i$ as integers, contradicting the assumed 
inequality in $P(F_u)$.

To see that the map $U \mapsto U'$ is a bijection, we construct its inverse.
Let $V$ be a row-semistrict tableau of type $e$ in $\mathcal T(F_u,(r))$, 
containing paths $\rho_{x_1}, \dotsc, \rho_{x_r}$ from left to right. 
Define $w \in \sn$ to be the permutation whose cycle notation is given by
\begin{equation*}
(x_1, \dotsc, x_{i_1 - 1})(x_{i_1},\dotsc,x_{i_2 - 1}) \cdots (x_{i_k},\dotsc,x_r),
\end{equation*}
where $x_1, x_{i_1}, x_{i_2}, \dotsc, x_{i_k}$ are 
the 
records of the word
$x_1 \cdots x_r$, i.e., $x_{i_j} = \max \{ x_1,\dotsc,x_{i_j} \}$.
Then write $w^{-1} = w^{-1}_1 \cdots w^{-1}_r$ in one-line notation and
define $V'$ to be the tableau in $\mathcal T(F_u,(r))$ 
whose $i$th 
entry is the unique path in $F_u$ from source $i$ to sink $w^{-1}_i$.  
It is clear that this map, if well defined, 
is inverse to the map $U \mapsto U'$, 
and therefore that the two are bijections. 
(See \cite[Sec.\,1.3]{StanEC1}).  

To see that the necessary paths exist in $F_u$,
consider a cycle $(x_j, \dotsc, x_{j+\ell})$ of $w$ and the pairs 
$(i,w^{-1}_i) \in 
\{(x_{j+a+1}, x_{j+a})\,|\, 1 < a \leq \ell \} \cup \{(x_{j+\ell}, x_j)\}$.
Since $V$ is row-semistrict, i.e.,
$\rho_{x_{j+a}} \not >_{P(F_u)} \rho_{x_{j+a+1}}$,
any integer inequality $x_{j+a} > x_{j+a+1}$ 
implies that there are paths in $F_u$ from sources $x_{j+a}$ and $x_{j+a+1}$ 
to (both) sinks $x_{j+a}$ and $x_{j+a+1}$.
In particular, there are paths in $F_u$ from sources $x_j$ and $x_{j+1}$
to sinks $x_j$ and $x_{j+1}$.  
Now assume that there are paths from 
source $x_{j+s}$ to sink $x_{j+s-1}$ and from source $x_j$ to sink $x_{j+s}$,
and consider 
the pair $(x_{j+s+1}, x_{j+s})$.  
If $x_{j+s+1} < x_{j+s}$, then by the above argument there is a path
from source $x_{j+s+1}$ to sink $x_{j+s}$.  
Since there are paths from sources $x_j$ and $x_{j+s+1}$ to sink $x_{j+s}$,
the two sources must belong to the same connected component of $F_u$.
By the comment following Observation~\ref{o:sourcetosink}, we also
have a path from source $x_j$ to sink $x_{j+s+1}$.  
If on the other hand $x_{j+s+1} < x_{j+s}$, 
then since $x_j$ is the maximum index in its cycle
we have $x_j > x_{j+s+1} > x_{j+s}$.  Since there are paths in $F_u$ from
source $x_j$ to sink $x_{j+s}$ and from source $x_{j+s+1}$ to sink $x_{j+s+1}$,
Observation~\ref{o:pnetintersect}
implies that there are also a paths from source
$x_j$ to sink $x_{j+s+1}$ and from source $x_{j+s+1}$ to $x_{j+s}$.
By induction, we have that for $a = 1, \dotsc, \ell$,
there is a path from source $x_{j+a}$ to sink $x_{j+a-1}$,
and that there is also a path from source $x_j$ to sink $x_{j+\ell}$.
\end{proof}

\ssec{Interpretation of $\sn$-class function evaluations}
The equivalence relation in (\ref{eq:posetequiv}) has applications in
the enumeration of certain $F$-tableaux and in the evaluation of $\sn$-class
functions.

\begin{thm}\label{t:thetavw}
Let $v, w \in \sn $ \avoidp\ and satisfy $v \sim w$,
and let $X$ be a property of $F$-tableaux 
which depends only upon the poset $P(F)$ (rather than on $Q(F)$).
Then $F_v$-tableaux and $F_w$-tableaux having property $X$ are in
bijective correspondence.  Moreover, for any $\sn$-class function $\theta$
we have $\theta(C'_v(1)) = \theta(C'_w(1))$.
\end{thm}
\begin{proof}
Since the property $X$ depends only upon the poset $P(F_v) \cong P(F_w)$,
we have a bijection between the sets of 
$F_v$-tableaux 
and $F_w$-tableaux having property $X$.

Now apply Lindstr\"om's Lemma and (\ref{eq:thetaCimmB}) to
the first Littlewood-Merris-Watkins identity in (\ref{eq:inducedimmssubmat})
to see that
for all $\lambda \vdash n$, the evaluations
$\epsilon^\lambda(C'_v(1))$
and $\epsilon^\lambda(C'_w(1))$
are equal to the numbers of 
column-strict $F_v$-tableaux and $F_w$-tableaux, respectively,
of type $e$ and shape $\lambda$.
Since column-strictness of these tableaux depends only upon 
$P(F_v) \cong P(F_w)$,
the above bijection implies that we have 
$\epsilon^\lambda(C'_v(1)) = \epsilon^\lambda(C'_w(1))$ for all $\lambda \vdash n$.
Since $\{ \epsilon^\lambda \,|\, \lambda \vdash n \}$ is a basis
for the space of $\sn$-class functions, each class function
$\theta$ 
satisfies
$\theta(C'_v(1)) = \theta(C'_w(1))$.
\end{proof}

For some $\sn$-class functions $\theta$, and all 
\pavoiding permutations $w$, we may combinatorially
interpret $\theta(C'_w(1))$
in terms of a zig-zag network $F_w$ as follows.
(See \cite[p.\,288]{StanEC2} 
for information on the majorization order, used in (v-a) below.)


\begin{thm}\label{t:sncfinterp}
Let $w \in \sn $ \avoidp, and
fix $\lambda = (\lambda_1, \dotsc, \lambda_r) \vdash n$.
Then we have the following.
\begin{enumerate}
\item[$(i)$] $\epsilon^\lambda(C'_w(1)) = 
\# \{ U \in \mathcal T(F_w,\lambda^\tr\,) \,|\, 
U \text{ column-strict of type } e \}$.
\item[$(ii$-$a)$] $\eta^\lambda(C'_w(1)) = 
\# \{ U \in \mathcal T(F_w,\lambda) \,|\, 
U \text{ row-closed, left row-strict}\, \}$.
\item[$(ii$-$b)$] $\eta^\lambda(C'_w(1)) = 
\# \{ U \in \mathcal T(F_w,\lambda) \,|\, 
U \text{ row-semistrict of type } e \}$.
\item[$(iii)$] $\chi^\lambda(C'_w(1)) = 
\# \{ U \in \mathcal T(F_w,\lambda) \,|\, 
U \text{ standard of type } e \}$.
\item[$(iv$-$a)$] $\psi^\lambda(C'_w(1)) = 
\# \{ U \in \mathcal T(F_w,\lambda) \,|\, 
U \text{ cylindrical}\, \}$.
\item[$(iv$-$b)$] $\psi^\lambda(C'_w(1)) = 
\# \{ U \in \mathcal T(F_w,\lambda) \,|\, 
U \text{ cyclically row-semistrict of type } e \}$.
\item[$(iv$-$c)$] $\psi^\lambda(C'_w(1)) = 
\# \{ U \in \mathcal T(F_w,\lambda) \,|\, 
U \text{ record-free, row-semistrict of type } e \}$.
\item[$(iv$-$d)$] $\psi^\lambda(C'_w(1)) = 
\lambda_1 \cdots \lambda_r \cdot
\# \{ U \in \mathcal T(F_w,\lambda) \,|\, 
U \text{ right-anchored, row-semistrict of type } e \}$.
\item[$(v$-$a)$] Suppose $\lambda_1 \leq 2$. We have
$\phi^\lambda(C'_w(1)) = 
\# \{ U \in \mathcal T(F_w,\lambda) \,|\, 
U \text{ column-strict of type } e \}$
if for all $\mu$ majorized by $\lambda$ we have
$\mathcal T(F_w,\mu) = \emptyset$;
otherwise we have $\phi^\lambda(C'_w(1)) = 0$.
\item[$(v$-$b)$] For $\lambda = (k)^r$, we have 
$\phi^\lambda(C'_w(1)) = 
\# \{ U \in \mathcal T(F_w,\lambda) \,|\, 
U \text{ column-strict, cylindrical}\, \}$.
\end{enumerate}
\end{thm}
\begin{proof} 
%
%
(i) See the proof of Theorem~\ref{t:thetavw}.


(ii-a) 
Apply the definition (\ref{eq:pathmatrix}) of path matrix
and (\ref{eq:thetaCimmB}) to 
the second Littlewood-Merris-Watkins identity in (\ref{eq:inducedimmssubmat}).

(ii-b) Since row-semistrictness in $F_w$-tableaux of type $e$
is a property of the poset 
$P(F_w)$,
we may apply 
Theorem~\ref{t:thetavw} and
Lemma~\ref{l:dropparens} to
the interpretation in (ii-a).

(iii) 
Applying (i) and Theorem~\ref{t:dsnuio} 
to Gasharov's~\cite[Thm.\,2]{GashInc},
we obtain the claimed interpretation for $312$-avoiding permutations.
(See also Section~\ref{s:chromsf}.)
Since standardness of $F_w$-tableaux depends only upon
$P(F_w)$, we may apply Theorem~\ref{t:thetavw}
to extend the result to \pavoiding permutations as well. 

(iv-a) 
Apply the definition (\ref{eq:pathmatrix}) of path matrix to 
the identity (\ref{eq:pimm}).

(iv-b) 
Let $v \sim w$ avoid the pattern $312$.
We define a map from
cylindrical tableaux in $\mathcal T(F_v,\lambda)$ 
to cyclically row-semistrict tableaux in $\mathcal T(F_v,\lambda)$ having
type $e$ as follows. 
For $U \in \mathcal T(F_v, \lambda)$ 
cylindrical with rows $k$ of $L(U)$ and $R(U)$ containing indices
$i_1, \dotsc, i_{\lambda_k}$ and
$i_2, \dotsc, i_{\lambda_k}, i_1$ (respectively) from left to right,
and $\rho = (\rho_1,\dotsc, \rho_n)$ the unique path family of type $e$
covering $F_v$,
create a cyclically row-semistrict tableau $U' \in \mathcal T(F_v, \lambda)$ 
by inserting $\rho_{i_1}, \dotsc, \rho_{i_{\lambda_k}}$ into row $k$,
from right to left.  
This map is bijective since 
in the descending star network $F_v$, 
there exists a (unique) path 
from source $i_j$ to sink $i_m$ if and only if $i_j > i_m$
or $\rho_{i_j}$ and $\rho_{i_m}$ intersect.
Thus the claimed interpretation holds for $312$-avoiding permutations.
Since cyclical row-semistrictness in $F_w$-tableaux of type $e$
depends only upon $P(F_w)$,
we may apply Theorem~\ref{t:thetavw}
to extend the result to \pavoiding permutations as well.


(iv-c) Shareshian and Wachs~\cite[Sec.\,4]{SWachsChromQ}
have shown that for $312$-avoiding permutations,
this formula is equivalent to Stanley's \cite[Thm.\,2.6]{StanSymm}.
Since the claimed property of $F_w$-tableaux depends only upon
$P(F_w)$, we may apply Theorem~\ref{t:thetavw}
to extend the result to \pavoiding permutations as well. 
(See also, \cite[Lem.\,6]{AthanPSE}.)

(iv-d) The number of tableaux in (iv-b) is equal to the cardinality
of the subset that are right-anchored, times $\lambda_1 \cdots \lambda_r$.
This subset is precisely
the right-anchored row-semistrict $F_w$-tableaux
of type $e$ and shape $\lambda$.
Alternatively, we may use the Shareshian-Wachs argument of (iv-c).

(v-a) This was first stated in \cite{CSSkanPathTab}, 
and will be proved in
Theorem~\ref{t:qphipartial}.
A different interpretation 
was given in 
\cite[Thm.\,2.5.1]{WolfgangThesis}.
 
(v-b) 
Apply Lindstr\"om's Lemma and 
(\ref{eq:thetaCimmB}) 
to Stembridge's
identity (\ref{eq:stemmonrect}).
\end{proof}

Conspicuously absent from Theorem~\ref{t:sncfinterp} is an
interpretation of monomial class function evaluations 
of the form $\phi^\lambda(C'_w(1))$ which holds for all
$\lambda \vdash n$.
As we have mentioned in the first table of Section~\ref{s:intro},
these integers are conjectured to be nonnegative.
The problem of interpreting them has been posed 
from 
different points of view 
by Haiman, Stanley and Stembridge 
\cite[Conj.\,2.1]{HaimanHecke},
\cite[Conj.\,5.1]{StanSymm},
\cite[Conj.\,5.5]{StanStemIJT},
\cite[Conj.\,2.1]{StemConj}.
Any extension of the statements
in Theorem~\ref{t:sncfinterp} (v-a), (v-b) would be interesting.
\begin{prob}\label{p:monevalinterp}
For (special cases of) $w \in \sn$ and $\lambda \vdash n$, 
find a combinatorial proof that $\phi^\lambda(C'_w(1))$ is nonnegative.
\end{prob}


\ssec{Inversions in path tableaux}
For $\lambda \vdash n$ and
$w \in \sn$ \avoidingp,
Theorem~\ref{t:sncfinterp} (i) -- (iv-d) interprets 
$\epsilon^\lambda(C'_w(1))$,
$\eta^\lambda(C'_w(1))$,
$\chi^\lambda(C'_w(1))$, and
$\psi^\lambda(C'_w(1))$
as cardinalities of certain sets of $F_w$-tableaux.
Using these same sets of $F_w$-tableaux and 
variations of the permutation statistic $\inv$,
we show 
in Sections~\ref{s:hnqinterph}, 
\ref{s:hnqinterpe}, 
\ref{s:hnqinterps}, 
\ref{s:hnqinterpp}
that 
$\epsilon_q^\lambda(\qew C'_w(q))$, 
$\eta_q^\lambda(\qew C'_w(q))$,
$\chi_q^\lambda(\qew C'_w(q))$, and
$\psi_q^\lambda(\qew C'_w(q))$
are generating functions for
tableaux on which the statistics take the values $k = 0,1,\dotsc$.

Specifically, we adapt the permutation statistic 
$\inv$ for use on path tableaux as follows.
Let $\pi = (\pi_1, \dotsc, \pi_n)$ be a path family of type $v$ in some
zig-zag network $F$,
and let $U$ be a $\pi$-tableau.
Let $(\pi_i, \pi_j)$ be a pair of intersecting paths in $F$ such that 
$\pi_i$ appears in a column of $U$ to the left of the column containing $\pi_j$.
Call $(\pi_i, \pi_j)$ a {\em (left) inversion} in $U$ if 
we have $i > j$
and a {\em right inversion} in $U$ if
we have $v_i > v_j$.
Let $\inv(U)$ denote the number of 
inversions in $U$,
and let $\rinv(U)$ denote the number of right inversions in $U$.

Sometimes we will compute inversions in a one-rowed tableau formed by
concatenating all of the rows of a path tableau $U$.
Let $U_i$ be the $i$th row of $U$, and
let $U_1 \circ \cdots \circ U_r$ and $U_r \circ \cdots \circ U_1$
be the $F$-tableaux of shape $n$ consisting of the rows of $U$ concatenated
in increasing and decreasing order, respectively.
We will also compute inversions in the {\em transpose} $U^\tr$ of a 
path tableau $U$, whose rows are the columns of $U$.
It is easy to see that inversions in these one-rowed and transposed tableaux
are related by the identities
\begin{equation}\label{eq:invdecomp}
\begin{gathered}
\inv(U_1 \circ \cdots \circ U_r) 
= \inv(U_1) + \cdots + \inv(U_r) + \inv(U^\tr),\\
\rinv(U_1 \circ \cdots \circ U_r) 
= \rinv(U_1) + \cdots + \rinv(U_r) + \rinv(U^\tr).
\end{gathered}
\end{equation}

\section{
Interpretation of $\eta_q^\lambda(\qew C'_w(q))$}
\label{s:hnqinterph}


Let $w \in \sn$ \avoidp, and
let $B$ be the path matrix of $F_w$.
Using (\ref{eq:qinducedimmssubmat}) and Proposition~\ref{p:qimmchar}, 
we have
\begin{equation}\label{eq:hqlambeta}
\eta_q^\lambda(\qew C'_w(q)) 
= \sigma_B(\imm{\eta_q^\lambda}(x)) 
= \sum_{(I_1, \dotsc, I_r)} \sigma_B (\qper(x_{I_1,I_1}) \cdots \qper(x_{I_r,I_r}) ),
\end{equation}
where the sum is over all ordered set partitions $(I_1, \dotsc, I_r)$ of
$[n]$ of type $\lambda = (\lambda_1, \dotsc, \lambda_r)$.
Let $\slambda$ denote the Young subgroup of $\sn$ generated by
\begin{equation*}
\{ s_1, \dotsc, s_{n-1} \} \ssm 
\{ s_{\lambda_1}, s_{\lambda_1 + \lambda_2}, 
s_{\lambda_1 + \lambda_2 + \lambda_3}, 
\dotsc, 
s_{n-\lambda_r} \},
\end{equation*}
and let $\slambdamin$ be the set of Bruhat-minimal representatives of cosets
of the form $\slambda u$, i.e., the elements $u \in \sn$ for which each
of the subwords
\begin{equation}\label{eq:subwordsofu}
u_1 \cdots u_{\lambda_1}, \qquad 
u_{\lambda_1 +1} \cdots u_{\lambda_1 + \lambda_2}, \quad
\dotsc, \quad
u_{n-\lambda_r + 1} \cdots u_n
\end{equation}
is strictly increasing.
It is clear that such elements correspond bijectively to the 
ordered set partitions $(I_1, \dotsc, I_r)$ in (\ref{eq:hqlambeta}).
Expanding the product of permanents, we obtain monomials of the form
$\quv$ times
\begin{equation*}
x^{u,v} \defeq \doublepermmon xuv,
\end{equation*}
where $v$ is the concatenation, in order, of 
rearrangements of the $r$ words (\ref{eq:subwordsofu}).  
Thus $v$ may be written as $yu$ with $y \in \slambda$,
or as $uy$ with $y \in u^{-1}\slambda u$.
Now the sum in 
(\ref{eq:hqlambeta}) becomes
\begin{equation}\label{eq:hqlambetaalt}
\sum_{u \in \slambdamin} \sum_{y \in u^{-1}\slambda u} \sigma_B (q_{u,uy} x^{u,uy}).
\end{equation}
Let us therefore consider evaluations of the form 
$\sigma_B (\quv x^{u,v})$.

To combinatorially interpret these evaluations, 
let $\pi = (\pi_1, \dotsc, \pi_n)$ be a path family 
(of arbitrary type)
which covers a zig-zag
network $F$, and
define $U(u,\pi)$ to be the $\pi$-tableau of shape $(n)$
containing $\pi$ in the order $\pi_{u_1}, \dotsc, \pi_{u_n}$.
Clearly the left tableau of $U(u,\pi)$ is $u_1 \cdots u_n$.
If the right tableau is $v_1 \cdots v_n$ then $\pi$ has type $u^{-1}v$.
If $s_i$ is a left descent for $u$, then right inversions in $U(u,\pi)$
and $U(s_iu,\pi)$ are related as follows.

\begin{prop}\label{p:usupi}
Fix $u, v \in \sn$, 
let $F$ be a zig-zag
network, and let $\pi = (\pi_1, \dotsc, \pi_n)$ be a path family
of type $u^{-1}v$ which covers $F$. 
If $s_i u < u$ then we have
\begin{equation*}
\rinv(U(u,\pi)) = 
\begin{cases}
\rinv(U(s_iu,\pi)) - 1 &\text{if $s_i v > v$,} \\
\rinv(U(s_iu,\pi)) 
&\parbox[t]{.5 \textwidth}
{if $s_iv < v$ and no path family
of type $u^{-1}s_iv$ covers $F$,}\\
\rinv(U(s_iu,\pi)) + 1 
&\parbox[t]{.5 \textwidth}
{if $s_iv < v$ and some path family
of type $u^{-1}s_iv$ covers $F$.}
\end{cases}
\end{equation*}
\end{prop}
\begin{proof}
The tableaux $U(u,\pi)$ and $U(s_iu, \pi)$ are identical except 
that $\pi_{u_i}$ appears before $\pi_{u_{i+1}}$ in $U(u,\pi)$.
Thus we have
\begin{equation*}
\rinv(U(u,\pi)) = \begin{cases}
\rinv(U(s_iu,\pi)) - 1 
&\parbox[t]{.5 \textwidth}
{if $(\pi_{u_i}, \pi_{u_{i+1}})$ is a right inversion 
in $U(s_iu,\pi)$ but not in $U(u,\pi)$,}\\
\rinv(U(s_iu,\pi)) + 1 
&\parbox[t]{.5 \textwidth}
{if $(\pi_{u_i}, \pi_{u_{i+1}})$ is a right inversion 
in $U(u,\pi)$ but not in $U(s_iu,\pi)$,}\\
\rinv(U(s_iu,\pi)) 
&\parbox[t]{.5 \textwidth}
{if $(\pi_{u_i}, \pi_{u_{i+1}})$ is not a right inversion 
in $U(s_iu,\pi)$ or $U(u,\pi)$.}
\end{cases}
\end{equation*}
Since $s_i u < u$, we have $u_i > u_{i+1}$.  

If $s_iv > v$, then we have $v_i < v_{i+1}$, and 
Observation~\ref{o:pnetintersect} implies that the paths
$\pi_{u_i}$ and $\pi_{u_{i+1}}$ intersect, forming
a right inversion in $U(s_iu,\pi)$ and not in $U(u,\pi)$.

If $s_iv < v$, then we have $v_i > v_{i+1}$, and the paths
$\pi_{u_i}$ and $\pi_{u_{i+1}}$ 
do not form a right inversion in $U(s_iu,\pi)$.
Suppose that some path family $\pi' = (\pi'_1, \dotsc, \pi'_n)$
of type $u^{-1}s_iv$ covers $F$.  
Then the tableau
$U(u,\pi')$ satisfies 
\begin{equation*}
L(U(u,\pi')) = (u_1, \dotsc, u_n), 
\qquad
R(U(u,\pi')) = (v_1, \dotsc, v_{i-1}, v_{i+1}, v_i, v_{i+2}, \dotsc, v_n ).
\end{equation*}
By the uniqueness of source-to-sink paths in zig-zag
networks,
this tableau is identical to the tableau $U(u,\pi)$ except for the paths
$\pi'_{u_i}$ and $\pi'_{u_{i+1}}$ in positions $i$ and $i+1$, which terminate
at sinks $v_{i+1} < v_i$, respectively.
By Observation~\ref{o:pnetintersect}, the paths $\pi'_{u_i}, \pi'_{u_{i+1}}$
cross and the paths $\pi_{u_i}$, $\pi_{u_{i+1}}$ intersect,
forming a right inversion in the tableau $U(u,\pi)$.
On the other hand, suppose that no path family of type $u^{-1}s_iv$ 
covers $F$.
Since $\pi$ has type $u^{-1}v$,  
we can deduce that either there is no path in $F$ 
from source $u_i$ to sink $v_{i+1}$
or there is no path from source $u_{i+1}$ to sink $v_i$.
By Observation~\ref{o:pnetintersect}, the paths $\pi_{u_i}$, $\pi_{u_{i+1}}$ 
do not intersect and therefore do not form a right inversion in $U(u,\pi)$. 
\end{proof}
 
Now we evaluate $\sigma_B(\quv x^{u,v})$, first in the case that $u = e$.
\begin{prop}\label{p:sigaqevxev}
Let $w$ in $\sn$ \avoidp,
let $B$ be the path matrix of $F_w$,
and fix $v$ in $\sn$.  Then we have
\begin{equation*}
\sigma_B (\qev x^{e,v}) = 
\begin{cases}
q^{\rinv(U(e,\pi))} &\text{if there exists a path family $\pi$ of type $v$ which covers $F_w$},\\
0 &\text{otherwise}.
\end{cases}
\end{equation*}
\end{prop}
\begin{proof}
By definition we have
\begin{equation}\label{eq:qlvaev}
\sigma_B (\qev x^{e,v}) = \qev \qev b^{e,v} = q^{\ell(v)}b^{e,v}.
\end{equation}
First assume that there exists a (unique) path family $\pi$
of type $v$ that covers $F_w$.
Then we have $b^{e,v} = 1$.  In the tableau $U(e,\pi)$, 
paths appear in the order $(\pi_1, \dotsc, \pi_n)$.
Now observe that for each inversion in $v$, i.e., 
each pair $(v_i, v_j)$ with $i < j$ and $v_i > v_j$,
the paths $\pi_i$ (from source $i$ to sink $v_i$)
and $\pi_j$ (from source $j$ to sink $v_j$) cross in $F_w$
and therefore form a right inversion in $U(e,\pi)$. 
Conversely, for each noninversion $(v_i, v_j)$ in $v$, the paths
$\pi_i$ and $\pi_j$ do not form a right inversion in $U(e,\pi)$.
Thus we have $\ell(v) = \rinv(U(e,\pi))$, and
the expression in (\ref{eq:qlvaev}) is equal to 
$q^{\rinv(U(e,\pi))}$.

Now assume that there is no path family of type $v$ which covers $F_w$.
Then we have $b^{e,v} = 0$ and
the expressions in (\ref{eq:qlvaev}) are equal to $0$.
\end{proof}

More generally, we evaluate $\sigma_B(\quv x^{u,v})$ as follows.
\begin{prop}\label{p:sigaquvxuv}
Let $w$ in $\sn$ \avoidp,
let $B$ be the path matrix of $F_w$,
and fix $u$, $v$ in $\sn$.  Then we have
\begin{equation*}
\sigma_B (\quv x^{u,v}) = 
\begin{cases}
q^{\rinv(U(u,\pi))} &\text{if there exists a path family $\pi$ of type $u^{-1}v$ 
covering $F_w$},\\
0 &\text{otherwise}.
\end{cases}
\end{equation*}
\end{prop}
\begin{proof}
We use induction on the length of $u$.
By Proposition~\ref{p:sigaqevxev}, the claimed formula 
holds when $u$ has length $0$.
Now assume that the formula holds when $u$ has length $1, \dotsc, k-1$,
and consider $u$ of length $k$.
Choosing a left descent $s_i$ of $u$, we may write
\begin{equation*}
\begin{aligned}
\sigma_B (\quv x^{u,v}) &= 
\begin{cases}
\sigma_B (\quv x^{s_iu, s_iv}) &\text{if $s_i v > v$,} \\
\sigma_B (\quv x^{s_iu, s_iv} + (\qdiff) \quv x^{s_iu,v} ) &\text{if $s_i v < v$,}
\end{cases}\\
&= \begin{cases}
q^{-1} \sigma_B (q_{s_i u, s_i v} x^{s_iu, s_iv}) &\text{if $s_i v > v$,} \\
\sigma_B (q_{s_i u, s_i v} x^{s_iu, s_iv}) + (1 - q^{-1}) \sigma_B (q_{s_iu, v} x^{s_iu,v}) 
&\text{if $s_i v < v$.}
\end{cases}
\end{aligned}
\end{equation*}

Suppose first that we have $s_i v > v$.
Then 
by induction we have 
\begin{equation*}
\sigma_B (q_{s_i u, s_i v} x^{s_iu, s_iv}) = 
\begin{cases}
q^{\rinv(U(s_i u,\pi))} 
&\parbox[t]{.5 \textwidth}{if there exists a path family $\pi$ of 
type $u^{-1}v$ covering $F_w$,}\\
0 &\text{otherwise,}
\end{cases}
\end{equation*}
and Proposition~\ref{p:usupi} implies that the claim is true in this case.

Now suppose that we have $s_i v < v$
and consider path families of types $u^{-1}v$ and $u^{-1}s_iv$ 
which cover $F_w$.
If there are no path families of types $u^{-1}v$ and $u^{-1}s_iv$
which cover $F_w$, then
by induction we have 
\begin{equation}\label{eq:sigaaeqx1}
\sigma_B (q_{s_i u, s_i v} x^{s_iu, s_iv}) = \sigma_B (q_{s_iu, v} x^{s_iu,v}) = 0.
\end{equation}
If there exists a path family $\pi = (\pi_1,\dotsc,\pi_n)$ of type
$u^{-1}v$ which covers $F_w$, but no path family of type $u^{-1}s_iv$
which covers $F_w$, then by induction and Proposition~\ref{p:usupi}
we have
\begin{equation}\label{eq:sigaaeqx2}
\begin{gathered}
\sigma_B (q_{s_i u, s_i v} x^{s_iu, s_iv}) = q^{\rinv(U(s_iu,\pi))} = q^{\rinv(U(u,\pi))},\\
\sigma_B (q_{s_iu, v} x^{s_iu,v}) = 0.
\end{gathered}
\end{equation}
If there exists a path family
$\pi' = (\pi'_1,\dotsc, \pi'_n)$ 
of type $u^{-1}s_iv$ which covers $F_w$,
then the paths $\pi'_{u_i}$, $\pi'_{u_{i+1}}$ 
from sources $u_i > u_{i+1}$ to sinks $v_{i+1} < v_i$ (respectively)
cross.
It follows that there exists a path family 
$\pi = (\pi_1, \dotsc, \pi_n)$ of type $u^{-1}v$
which covers $F_w$, and which agrees with $\pi'$ except that
the paths $\pi_{u_i}$, $\pi_{u_{i+1}}$ from sources $u_i > u_{i+1}$
to sinks $v_i > v_{i+1}$ (respectively)
intersect but do not cross.
By induction and the existence of $\pi'$ and $\pi$ we have
\begin{equation*}
\begin{gathered}
\sigma_B (q_{s_i u, s_i v} x^{s_iu, s_iv}) = q^{\rinv(U(s_i u,\pi))},\\ 
\sigma_B (q_{s_iu, v} x^{s_iu,v}) = 
q^{\rinv(U(s_iu, \pi'))}. 
\end{gathered}
\end{equation*}
The tableaux 
$U(s_i u, \pi)$ and $U(s_i u, \pi')$ agree except in positions $i$ and $i+1$,
where paths $\pi'_{u_{i+1}}$ and $\pi'_{u_i}$ form a right inversion,
but $\pi_{u_{i+1}}$ and $\pi_{u_i}$ do not.
This fact and Proposition~\ref{p:usupi} imply that we have
\begin{equation*}
\rinv(U(s_iu, \pi')) = 
\rinv(U(s_iu, \pi)) + 1 = 
\rinv(U(u,\pi)),
\end{equation*}
and
\begin{equation}\label{eq:sigaaeqx4}
\begin{gathered}
\sigma_B (q_{s_i u, s_i v} x^{s_iu, s_iv}) = q^{\rinv(U(u,\pi))-1},\\ 
(1 - q^{-1})\sigma_B (q_{s_iu, v} x^{s_iu,v}) = 
q^{\rinv(U(u, \pi))} - q^{\rinv(U(u, \pi))-1}. 
\end{gathered}
\end{equation} 
Thus by Equations~(\ref{eq:sigaaeqx1}), 
(\ref{eq:sigaaeqx2}), and (\ref{eq:sigaaeqx4})
the claim is true when $s_iv < v$.
\end{proof}

Now we have the following $q$-analog of Theorem~\ref{t:sncfinterp} (ii-a).

\begin{thm}\label{t:qeta}
Let $w \in \sn$ \avoidp.
Then for $\lambda = (\lambda_1, \dotsc, \lambda_r) \vdash n$ we have
\begin{equation*}
\eta_q^\lambda(\qew C'_w(q)) = \sum_U q^{\rinv(U_1 \circ \cdots \circ U_r)},
\end{equation*}
where the sum is over all row-closed, left row-strict 
$F_w$-tableaux 
of shape $\lambda$.
\end{thm}
\begin{proof}
Let $B$ be the path matrix of $F_w$ and let $(I_1,\dotsc,I_r)$ be a set
partition of $[n]$ of type $\lambda$.
By 
(\ref{eq:subwordsofu}) -- (\ref{eq:hqlambetaalt}),
there is a permutation $u \in \slambdamin$ 
corresponding to $(I_1,\dotsc,I_r)$ such that 
we have
\begin{equation*}
\sigma_B(\qper(x_{I_1,I_1}) \cdots \qper(x_{I_r,I_r})) = 
\sum_{y \in u^{-1}\slambda u} \sigma_B ( q_{u,uy} x^{u, uy} ).
\end{equation*}
By Proposition~\ref{p:sigaquvxuv}, this is equal to
\begin{equation}\label{eq:ypirinvuupi}
\sum_{(y,\pi)} q^{\rinv(U(u,\pi))},
\end{equation}
where the 
sum is over pairs $(y,\pi)$ such that 
$y \in u^{-1}\slambda u$ and $\pi$ is a path family of type $y$ 
which covers $F$.
If such a path family $\pi$ exists for a given permutation $y$, 
it is necessarily unique.  
Thus as $y$ varies over $u^{-1}\slambda u$
we have that
$U(u, \pi)$
varies over all bijective path tableaux
of shape $(n)$ which satisfy
\begin{enumerate}
\item For $j=1,\dotsc,r$, 
the paths in positions 
$\lambda_1 + \cdots + \lambda_{j-1}+1, \dotsc,$ 
$\lambda_1 + \cdots + \lambda_j$ are indexed by $I_j$, in increasing order.
\item The sequence of sink indices of these same paths 
are a rearrangement of $I_j$.
\end{enumerate}
Thus the expression
in (\ref {eq:ypirinvuupi})
may be rewritten as
\begin{equation}\label{eq:rinvudotsu}
\sum_U q^{\rinv(U_1 \circ \cdots \circ U_r)},
\end{equation}
where this last sum is over all 
row-closed, left row-strict 
$F_w$-tableaux $U$ of shape $\lambda$
for which path indices of $U_j$ are $I_j$ for $j=1,\dotsc,r$.
Summing over ordered set partitions and using (\ref{eq:hqlambeta}),
we have the desired result.
\end{proof}

For example, consider the network
\begin{equation}\label{eq:F3421}
F_{3421} = \raisebox{-10mm}{
\includegraphics[height=20mm]{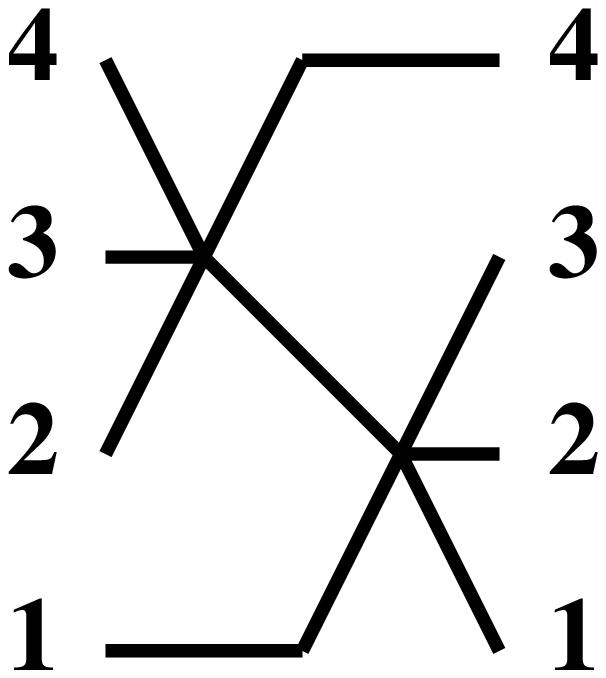}}\ .
\end{equation}
It is easy to verify that 
there are twenty row-closed, left row-strict
$F_{3421}$-tableaux of shape $31$.
Four of these are
\begin{equation*}
\raisebox{-6mm}{\includegraphics[height=12mm]{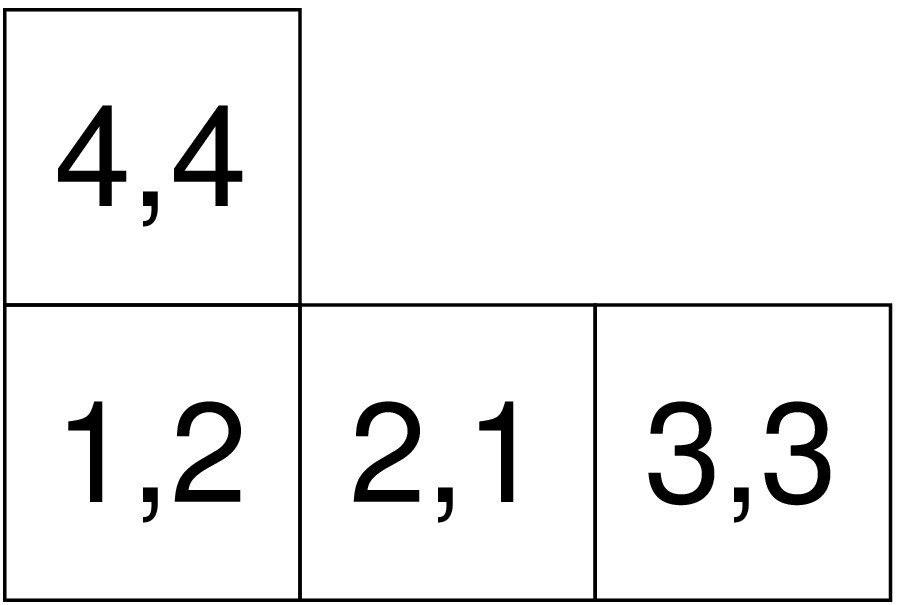}}\ ,\quad
\raisebox{-6mm}{\includegraphics[height=12mm]{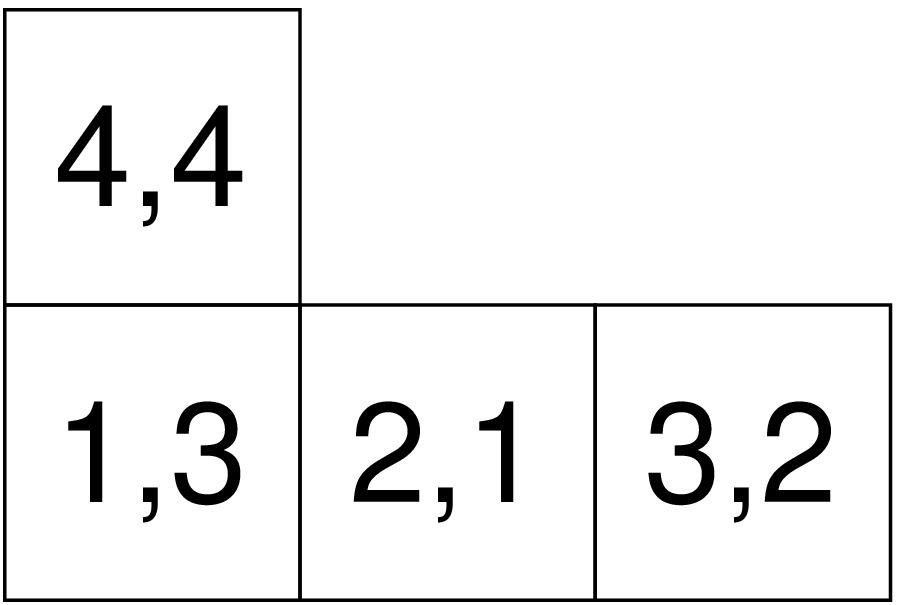}}\ ,\quad
\raisebox{-6mm}{\includegraphics[height=12mm]{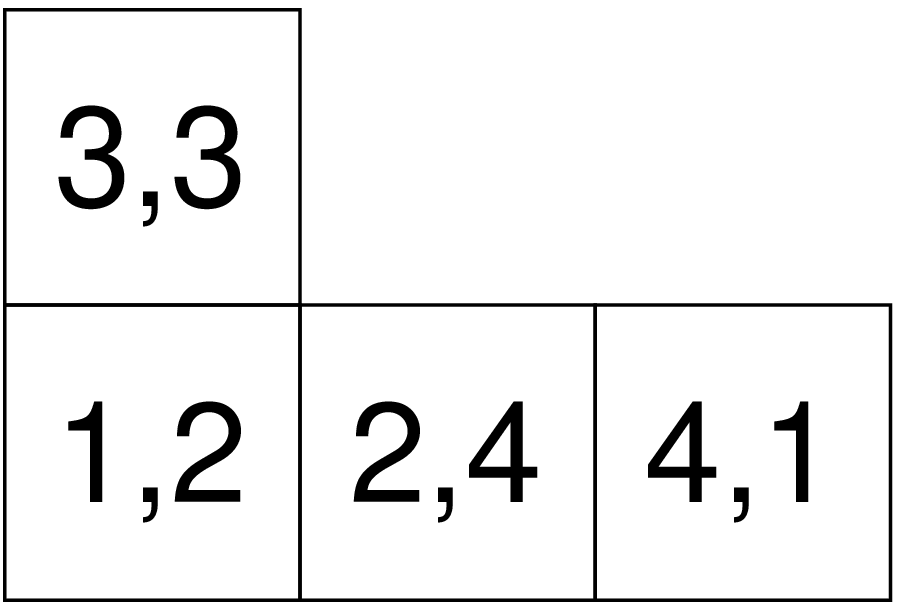}}\ ,\quad
\raisebox{-6mm}{\includegraphics[height=12mm]{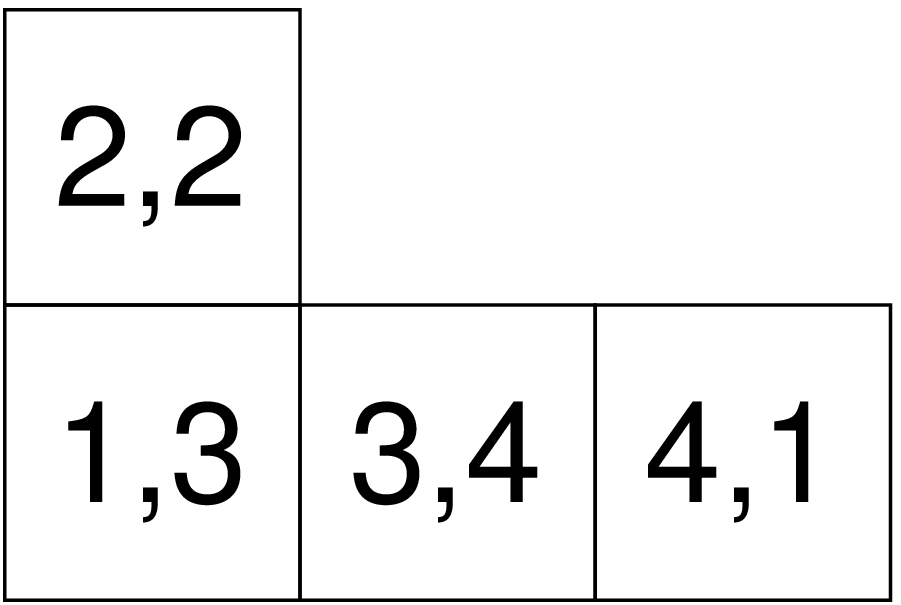}}\ ,
\end{equation*}
where $i,j$ represents the unique path from source $i$ to sink $j$.
These tableaux $U$ of shape $31$
yield tableaux $U_1 \circ U_2$ of shape $4$,
\begin{equation*}
\raisebox{-3mm}{\includegraphics[height=6mm]{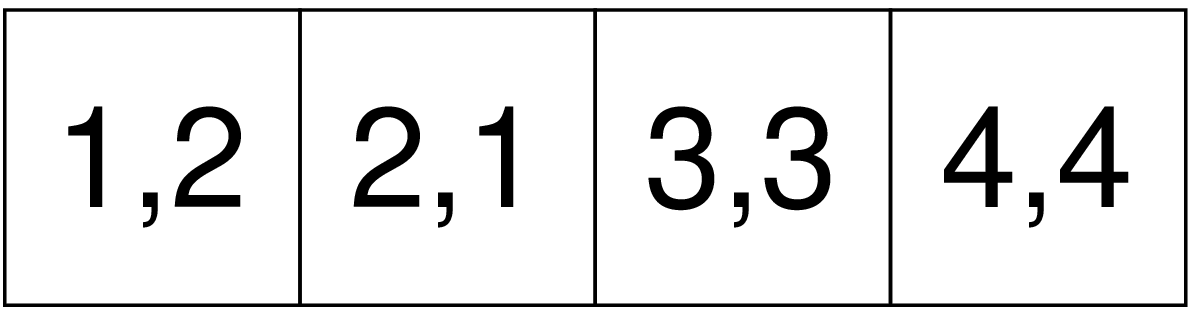}}\ ,\quad
\raisebox{-3mm}{\includegraphics[height=6mm]{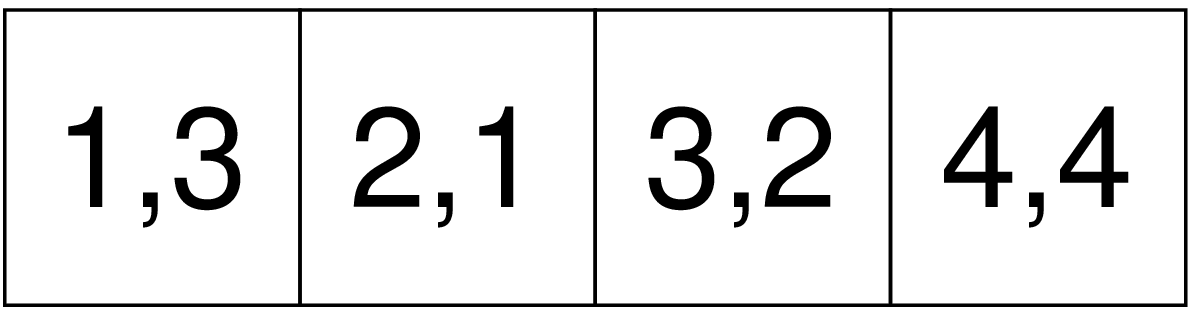}}\ ,\quad
\raisebox{-3mm}{\includegraphics[height=6mm]{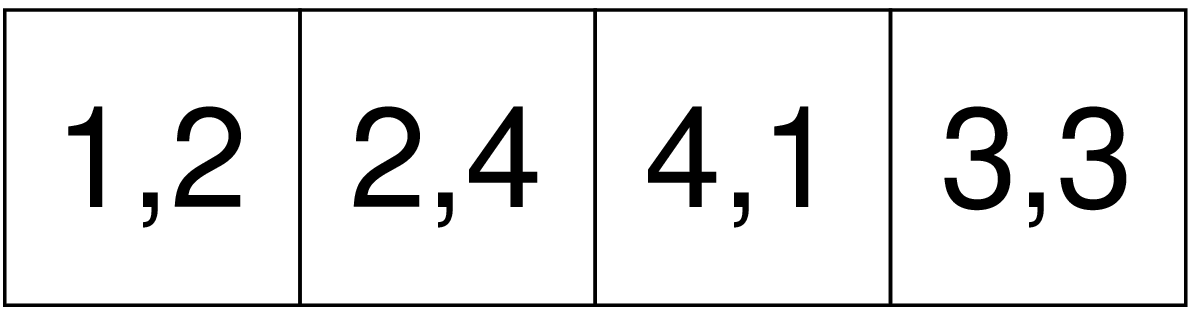}}\ ,\quad
\raisebox{-3mm}{\includegraphics[height=6mm]{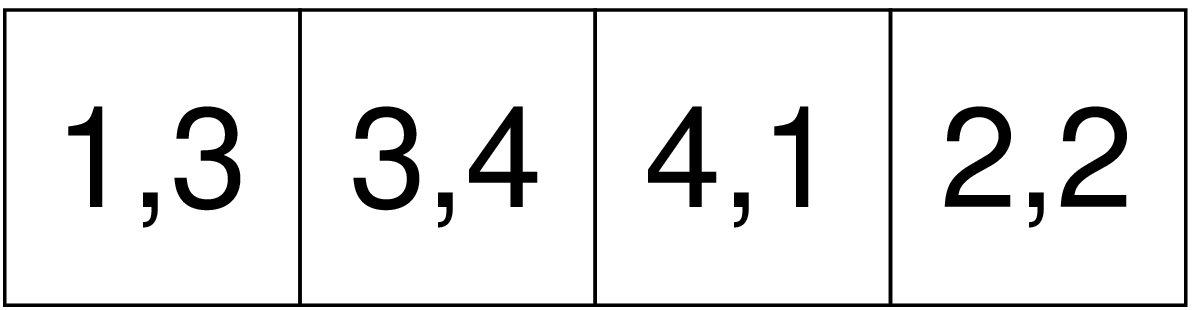}}\ ,
\end{equation*}
which have $1$, $2$, $3$, and $4$ right inversions, respectively.
Together, 
the tableaux 
contribute $q + q^2 + q^3 + q^4$
to $\eta_q^{31}(q_{e,3421} C'_{3421}(q)) = 1 + 3q + 6q^2 + 6q^3 + 3q^4 + q^5$.

Expanding $\eta_q^\lambda$ in terms of irreducible characters
and Kostka numbers, $\eta_q^\lambda = \sum K_{\mu,\lambda} \chi_q^\mu$,
and using Haiman's result~\cite[Lem\,1.1]{HaimanHecke},
we have that the coefficients of $\eta_q^\lambda( \qew C'_w(q))$ 
are symmetric and unimodal
about $\qew$ for all $w \in \sn$.
In the case that $w$ \avoidsp, it would be interesting to explain
this phenomenon combinatorially in terms of Theorem~\ref{t:qeta}.

It would also be interesting to extend 
Theorem~\ref{t:qeta} to include 
a $q$-analog of Theorem~\ref{t:sncfinterp} (ii-b).
In particular, the identity
\begin{equation*}
\eta_q^{(n)} = \sum_{\lambda \vdash n} \phi_q^\lambda
\end{equation*}
suggests that an answer to Problem~\ref{p:monevalinterp} and its $q$-analog
are related to a set partition of tableaux counted by $\eta^{(n)}$.
It is not clear whether such a partition is more easily expressed
in terms of left row-strict tableaux of shape $(n)$, 
or row-semistrict tableaux of type $e$ and shape $(n)$.
\begin{prob}
Find a statistic $\STAT$ on $F$-tableaux such that we have
\begin{equation*}
\eta_q^\lambda(\qew C'_w(q)) = \sum_U q^{\STAT(U)},
\end{equation*}
where the sum is over all row-semistrict $F_w$-tableaux 
of type $e$ and shape $\lambda$.
\end{prob}

As a consequence of Theorem~\ref{t:qeta}, we obtain the following
$q$-analog of Theorem~\ref{t:thetavw}.

\begin{thm}\label{t:qthetavw}
Let $v, w \in \sn$ \avoidp{} and satisfy $v \sim w$, 
and let $(X, \STAT)$
be a property of $F$-tableaux and a statistic on $F$-tableaux 
which depend only upon the poset $P(F)$.
Then $F_v$-tableaux 
and $F_w$-tableaux 
having property $X$ and 
satisfying $\STAT(U) = k$ are in bijective correspondence.  Moreover,
for each $\hnq$-trace $\theta_q$ we have 
\begin{equation*}
\theta_q(\qev C'_v(q)) = \theta_q(\qew C'_w(q)).
\end{equation*}
\end{thm}
\begin{proof}
Since the pair $(X, \STAT)$ depends 
only upon the poset $P(F_v) \cong P(F_w)$,
we have a bijection between the sets of 
$F_v$-tableaux 
and $F_w$-tableaux 
having property $X$ and satisfying $\STAT(U) = k$.
In particular, we have 
$\eta_q^\lambda(\qev C'_v(q)) = \eta_q^\lambda(\qew C'_w(q))$, 
and since $\{ \eta_q^\lambda \,|\, \lambda \vdash n \}$ is a basis of the
space of $\hnq$-traces, we have
$\theta_q(\qev C'_v(q)) = \theta_q(\qew C'_w(q))$ for all $\hnq$-traces 
$\theta_q$ as well.
\end{proof}
Let $\theta_q$ be an $\hnq$-trace.
If the posets $P(F_v)$, $P(F_w)$ of two zig-zag networks are dual,
rather than isomorphic, we still have
$\theta_q(\qev C'_v(q)) = \theta_q(\qew C'_w(q))$.  
In this case $v$ and $w$ satisfy $v \sim w_0ww_0$ and $\ell(v) = \ell(w)$, 
where $w_0$ is the longest element of $\sn$. 
Since natural basis elements of $\hnq$ satisfy 
$T_{w_0}T_w T_{w_0}^{-1} = T_{w_0ww_0}$, and since any $\hnq$-trace $\theta_q$
satisfies $\theta_q(gh) = \theta_q(hg)$ for all $h, g \in \hnq$,
we have that the equations
\begin{equation*}
\theta_q(q_{e,w_0ww_0} C'_{w_0ww_0}(q)) = 
\theta_q(\qew T_{w_0} C'_w(q) T_{w_0}^{-1}) = 
\theta_q(\qew C'_w(q))
\end{equation*}
hold for all $w \in \sn$.

\section{Interpretation of $\epsilon_q^\lambda(\qew C'_w(q))$}\label{s:hnqinterpe}


Let $w \in \sn$ \avoidp\ and
let $B$ be the path matrix of $F_w$.  Following the computations
of Equations (\ref{eq:hqlambeta}) -- (\ref{eq:hqlambetaalt}),
we have
\begin{equation}\label{eq:eqlambepsilon}
\begin{aligned}
\epsilon_q^\lambda(\qew C'_w(q)) 
= \sigma_B(\imm{\epsilon_q^\lambda}(x)) 
&= \sum_{(I_1,\dotsc,I_r)} \sigma_B(\qdet(x_{I_1,I_1}) \cdots \qdet(x_{I_r,I_r})) \\
&= \sum_{u \in \slambdamin} \sum_{y \in u^{-1}\slambda u} 
(-1)^{\ell(uy)-\ell(u)}\sigma_B (q_{u,uy}^{-1} x^{u,uy}),
\end{aligned}
\end{equation}
where the first sum is over all ordered set partitions $(I_1,\dotsc, I_r)$
of $[n]$ of type $\lambda$.
Let us therefore consider evaluations of the form 
$\sigma_B (\qiuv x^{u,v})$.

To combinatorially interpret these evaluations, 
let $\pi = (\pi_1, \dotsc, \pi_n)$ be a path family 
(of arbitrary type)
which covers a zig-zag
network $F$ and
define $U(u,\pi,\lambda)$ to be the $\pi$-tableau of shape $\lambda$
containing $\pi$ in the order $\pi_{u_1}, \dotsc, \pi_{u_n}$.
That is, 
$U(u, \pi, \lambda)_j$
contains the $\lambda_j$ paths whose indices are
\begin{equation*}
u_{\lambda_1 + \cdots + \lambda_{j-1}+1}, \dotsc, u_{\lambda_1 + \cdots + \lambda_j}.
\end{equation*}
Clearly 
$L(U(u,\pi,\lambda))$ 
contains the numbers 
$u_1, \dotsc, u_{\lambda_1}$ in row 1, 
$u_{\lambda_1 + 1}, \dotsc, u_{\lambda_1+\lambda_2}$ in row 2, etc. 
If $R(U(u,\pi,\lambda))$ 
contains the numbers
$v_1, \dotsc, v_{\lambda_1}$ in row 1, 
$v_{\lambda_1 + 1}, \dotsc, v_{\lambda_1+\lambda_2}$ in row 2, etc.,
then $\pi$ has type $u^{-1}v$.
In terms of this notation, our earlier tableau $U(u,\pi)$ 
defined after (\ref{eq:hqlambetaalt}) is equal to 
$U(u,\pi,(n))$.
If $u \in \slambdamin$ corresponds to $(I_1, \dotsc, I_r)$ as in
(\ref{eq:subwordsofu}) -- (\ref{eq:hqlambetaalt}),
then the path indices in row $j$ of $U(u,\pi,\lambda)$
are simply the elements of $I_j$, in increasing order.

\begin{prop}\label{p:sigaqiuvxuv}
Let $w \in \sn$ \avoidp, 
and let $B$ be the path matrix of $F_w$. 
Fix $\lambda \vdash n$, $u \in \slambdamin$, and $y \in u^{-1}\slambda u$.
Then we have
\begin{equation*}
\sigma_B (q_{u,uy}^{-1} x^{u,uy}) = 
\begin{cases}
q^{\rinv(U(u,\pi,\lambda)^\tr)} 
&\text{if there exists a path family $\pi$ of type $y$ 
covering $F_w$},\\
0 &\text{otherwise}.
\end{cases}
\end{equation*}
\end{prop}
\begin{proof}
By Proposition~\ref{p:sigaquvxuv} we have
\begin{equation*}
\begin{aligned}
\sigma_B (q_{u,uy}^{-1} x^{u,uy}) &= q_{u,uy}^{-2} \sigma_B (q_{u,uy} x^{u,uy}) \\
&= \begin{cases} 
q^{\ell(u)-\ell(uy)} q^{\rinv(U(u,\pi,(n)))} 
&\text{if some path family of type $y$ covers $\hat F$,}\\
0 &\text{otherwise.}
\end{cases}
\end{aligned}
\end{equation*}

Letting $V = U(u,\pi,\lambda)$ 
and using (\ref{eq:rinvudotsu}),
we may rewrite the above exponent of $q$ as
$\rinv(V_1 \circ \cdots \circ V_r) + \ell(u) - \ell(uy)$.
By (\ref{eq:invdecomp}), this is
\begin{equation*}
\rinv(V_1) + \cdots + \rinv(V_r) + \rinv(V^\tr) - (\ell(uy) - \ell(u)).
\end{equation*}
We claim that this expression reduces further to $\rinv(V^\tr)$.
To see this, 
recall that the tableaux $L(V_1) \circ \cdots \circ L(V_r)$ 
and $R(V_1) \circ \cdots \circ R(V_r)$ contain
the one-line notations of $u$ and $uy$, respectively.
Since $u$ belongs to $\slambdamin$ and each tableau $R(V_j)$ is a permutation
of the corresponding tableau $L(V_j)$,
we have 
\begin{equation*}
\begin{gathered}
\begin{aligned}
\ell(u) 
&= \# \{ (k,k') \,|\, k > k', 
\text{ $k$ in an earlier row of $L(V)$ than $k'$} \}\\
&= \# \{ (k,k') \,|\, k > k', 
\text{ $k$ in an earlier row of $R(V)$ than $k'$} \}.
\end{aligned}\\
\begin{aligned}
\ell(uy) &= \sum_{j=1}^r \inv(R(V_j)) + 
\# \{ (k,k') \,|\, k > k', 
\text{ $k$ in an earlier row of $R(V)$ than $k'$} \}\\
&= \sum_{j=1}^r \inv(R(V_j)) + \ell(u).
\end{aligned}
\end{gathered}
\end{equation*}


Now fix a row $V_j$ of $V$ and consider right inversions in $V_j$.
Let $\pi_i$, $\pi_{i'}$ be two paths in this row, with $\pi_i$ appearing first.
Since $u$ belongs to $\slambdamin$, we have $i < i'$.
Let $k$ and $k'$ be the corresponding sink indices.
If we have $k > k'$, then the paths cross and form a right inversion in $V_j$.
On the other hand, if we have $k < k'$, 
then the paths do not form a right inversion in $V_j$, even if they intersect.
Thus we have $\rinv(V_j) = \inv(R(V_j))$ for all $j$ and
\begin{equation}\label{eq:luylu}
\rinv(V_1) + \cdots + \rinv(V_r) = \ell(uy) - \ell(u),
\end{equation}
as desired.
\end{proof}

While the final sum in (\ref{eq:eqlambepsilon}) has signs,
we will use a sign-reversing involution to cancel some of the
the terms there, and to obtain a signless sum more amenable to 
combinatorial interpretation.
Fix a partition $\lambda = (\lambda_1, \dotsc, \lambda_r) \vdash n$,
an ordered set partition $I$ of $[n]$ of type $\lambda$, 
and a zig-zag
network $F$,
and let $\mathcal T_I = \mathcal T_I(F)$ be 
the set of all row-closed, left row-strict $F$-tableaux $U$ of shape $\lambda$
such that $L(U)_j = I_j$ (as sets) for $j = 1,\dotsc,r$.
Observe that all row-strict $F$-tableaux of type $e$ 
satisfying $L(U)_j = I_j$ (as sets) for $j = 1,\dotsc,r$
belong to $\mathcal T_I$.
Let us define an involution
\begin{equation*}
\zeta: \mathcal T_I \rightarrow \mathcal T_I
\end{equation*}
as follows.  
\begin{enumerate}
\item If $U$ is a row-strict tableau of type $e$, then define $\zeta(U) = U$.
\item Otherwise, 
  \begin{enumerate}
  \item Let $i$ be the least
index such that $U_i$ is not row-strict.
  \item 
Let $(j,j')$ be the 
lexicographically least
pair of indices in $L(U)_i$ such that 
$\pi_j$ and $\pi_{j'}$ 
intersect.
  \item Let $(k,k')$ be the sink indices of paths $\pi_j$ and $\pi_{j'}$, respectively.
  \item Define $\zeta(U)$ to be the tableau obtained from $U$ by replacing
$\pi_j$ and $\pi_{j'}$ by the unique paths in $F$ from
source $j$ to sink $k'$ and source $j'$ to sink $k$.
  \end{enumerate}
\end{enumerate}

\begin{prop}\label{p:zeta}
The involution $\zeta$ satisfies 
$\rinv(\zeta(U)^\tr) = \rinv(U^\tr)$.
\end{prop}
\begin{proof}
For each $F$-tableau $U \in \mathcal T_I$ 
satisfying $\zeta(U) = U$, the 
claimed equality is obvious.
Now let $U$ be an $F$-tableau not fixed by $\zeta$. Define the indices
$i$, $j$, $j'$, $k$, $k'$ 
and paths 
$\pi_j$, $\pi_{j'}$ as in the definition of $\zeta$,
and let $\pi'_j$, $\pi'_{j'}$ be the two new paths created in the final step
of the definition of $\zeta$.
Since the tableaux $U^\tr$ and $\zeta(U)^\tr$ agree everywhere except in the
two cells in column $i$ containing the paths 
\begin{equation}\label{eq:4paths}
\pi_j, \quad \pi_{j'}, \quad \pi'_j, \quad \pi'_{j'},
\end{equation}
it is clear 
that the right inversions of these tableaux are equal except possibly for 
inversions involving a path $\pi_h$ in a column other than $i$ 
and one of the paths (\ref{eq:4paths}).
We claim that these remaining right inversions in $U^\tr$ and $\zeta(U)^\tr$
correspond bijectively.  In particular, we have
\begin{enumerate}[(a)]
\item $\pi_h$ forms a right inversion with $\pi_j$ in $U^\tr$
if and only if
it forms a right inversion with $\pi'_{j'}$ in $\zeta(U)^\tr$.
\item $\pi_h$ forms a right inversion with $\pi_{j'}$ in $U^\tr$
if and only if
it forms a right inversion with $\pi'_{j}$ in $\zeta(U)^\tr$.
\end{enumerate}


To see this, observe that $j < j'$, 
and consider the intersection of $\pi_h$ with $\pi_j$ and $\pi_{j'}.$
By Lemma~\ref{l:dsnintersect} 
we have four cases: 
\begin{enumerate}
\item $\pi_h$ intersects both $\pi_j$ and $\pi_{j'}$.
\item $\pi_h$ intersects only $\pi_{j'}$.
\item $\pi_h$ intersects only the path whose sink is $\min(k,k')$.
\item $\pi_h$ intersects neither $\pi_j$ nor $\pi_{j'}$.
\end{enumerate}
Now considering the intersections of $\pi_h$ with $\pi'_j$ and $\pi'_{j'}$,
we see that the above cases
imply respectively that $\pi_h$ intersects
both $\pi'_j$ and $\pi'_{j'}$,
only $\pi'_{j'}$ (not $\pi'_j$),
only the path in $\{\pi'_j, \pi'_{j'} \}$ whose
sink index is $\min(k,k')$, 
and neither $\pi'_j$ nor $\pi'_{j'}$.
In all cases, the equivalences (a) and (b) are true.
\end{proof}

Since the tableaux $U$ and $\zeta(U)$ agree except in one row,
we have the following.

\begin{prop}\label{p:obv}
Let $U \in \mathcal T_I$ be a tableau 
not fixed by $\zeta$ and let $i$ be the index
satisfying $U_i \neq \zeta(U)_i$.  Then we have
\begin{equation*}
|\rinv(\zeta(U)_j) - \rinv(U_j)| = 
\begin{cases}
1 &\text{if $j = i$},\\
0 &\text{otherwise}.
\end{cases}
\end{equation*}
\end{prop}
\begin{proof}
Obvious.
\end{proof}


Now we have the following $q$-analog of Theorem~\ref{t:sncfinterp} (i).

\begin{thm}\label{t:qepsilon}
Let $w \in \sn$ \avoidp.
Then for 
$\lambda \vdash n$ 
we have
\begin{equation*}
\epsilon_q^\lambda(\qew C'_w(q)) = \sum_U q^{\inv(U)},
\end{equation*}
where the sum is over all column-strict $F_w$-tableaux 
of 
type $e$
and 
shape $\lambda^\tr$.
\end{thm}
\begin{proof}
Let $B$ be the path matrix of $F_w$ and let $(I_1,\dotsc,I_r)$ be a set
partition of $[n]$ of type $\lambda$.
By 
(\ref{eq:subwordsofu}) -- (\ref{eq:hqlambetaalt}),
there is a permutation $u \in \slambdamin$ 
corresponding to $(I_1,\dotsc,I_r)$ such that 
we have
\begin{equation*}
\sigma_B(\qdet(x_{I_1,I_1}) \cdots \qdet(x_{I_r,I_r})) = 
\sum_{y \in u^{-1}\slambda u} (-1)^{\ell(uy)-\ell(u)} \sigma_B ( q_{u,uy}^{-1} x^{u, uy} ).
\end{equation*}
By Proposition~\ref{p:sigaqiuvxuv}
this is equal to
\begin{equation}\label{eq:ypiinvuupi}
\sum_{(y,\pi)} (-1)^{\ell(uy)-\ell(u)} q^{\rinv(U(u,\pi,\lambda)^\tr)},
\end{equation}
where the 
sum is over pairs $(y,\pi)$ such that 
$y \in u^{-1}\slambda u$ and 
$\pi$ is a path family of type $y$ which covers $F_w$.
If such a path family exists, it is necessarily unique.  
Thus as $y$ varies over $u^{-1}\slambda u$
we have that
$U(u, \pi,\lambda)$
varies over all 
tableaux in $\mathcal T_I$. 
Thus by (\ref{eq:luylu}) this sum is equal to
\begin{equation}\label{eq:Vprecancel}
\sum_{V \in \mathcal T_I} (-1)^{\rinv(V_1) + \cdots + \rinv(V_r)} q^{\rinv(V^\tr)}.
\end{equation}

Now consider a tableau $V \in \mathcal T_I$ which satisfies 
$\zeta(V) \neq V$. By Propositions~\ref{p:zeta} -- \ref{p:obv},
the term of the above sum corresponding to the tableau $\zeta(V)$ is
\begin{equation*}
(-1)^{\rinv(\zeta(V)_1) + \cdots + \rinv(\zeta(V)_r)} q^{\rinv(\zeta(V)^\tr)}
= -(-1)^{\rinv(V_1) + \cdots + \rinv(V_r)} q^{\rinv(V^\tr)}.
\end{equation*}
Thus all terms corresponding to tableaux $V$ and $\zeta(V) \neq V$
cancel one another in the sum (\ref{eq:Vprecancel}), leaving terms
only for the tableaux $V \in \mathcal T_I$ 
which are fixed by $\zeta$, i.e., the row-strict tableaux.  
For these tableaux we have
\begin{equation*}
\rinv(V_1) = \cdots = \rinv(V_r) = 0, \qquad
\rinv(V^\tr) = \inv(V^\tr).
\end{equation*}
Furthermore, $V$ is row-strict of shape $\lambda$ if and only if
$V^\tr$ is column-strict of shape $\lambda^\tr$.  Thus we may again rewrite
(\ref{eq:ypiinvuupi}) as
\begin{equation*}
\sum_U q^{\inv(U)},
\end{equation*}
where the sum is over all column-strict $F_w$-tableaux of shape $\lambda^\tr$
satisfying $U_j = I_j$ (as sets).
Summing over ordered set partitions and using (\ref{eq:eqlambepsilon}),
we have the desired result.
\end{proof}


For example, consider the descending star network $F_{3421}$ in
(\ref{eq:F3421}).
It is easy to verify that 
there are exactly two column-strict
$F_{3421}$-tableaux of type $e$ and shape $31$:
\begin{equation}\label{eq:epsilontabs31}
\raisebox{-6mm}{\includegraphics[height=12mm]{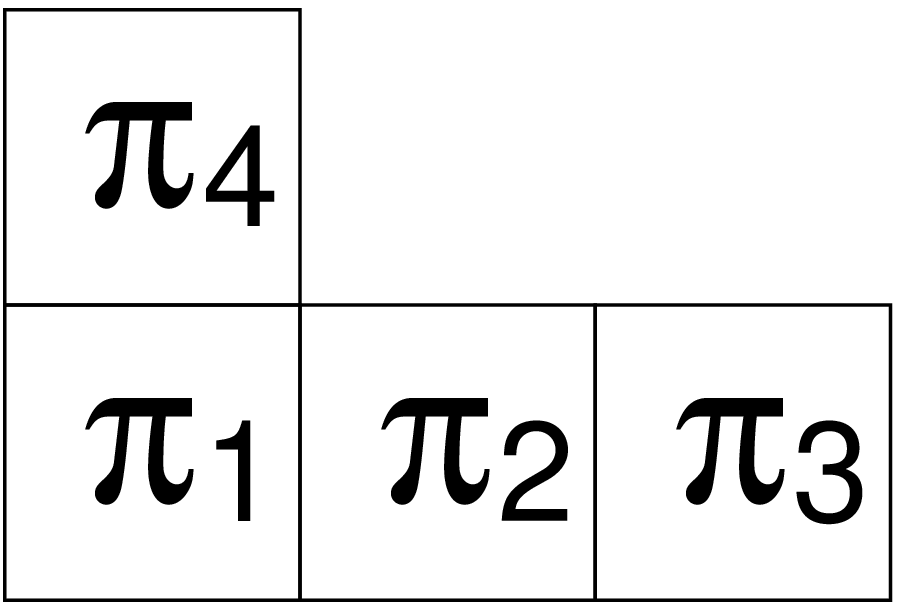}}\ ,\quad
\raisebox{-6mm}{\includegraphics[height=12mm]{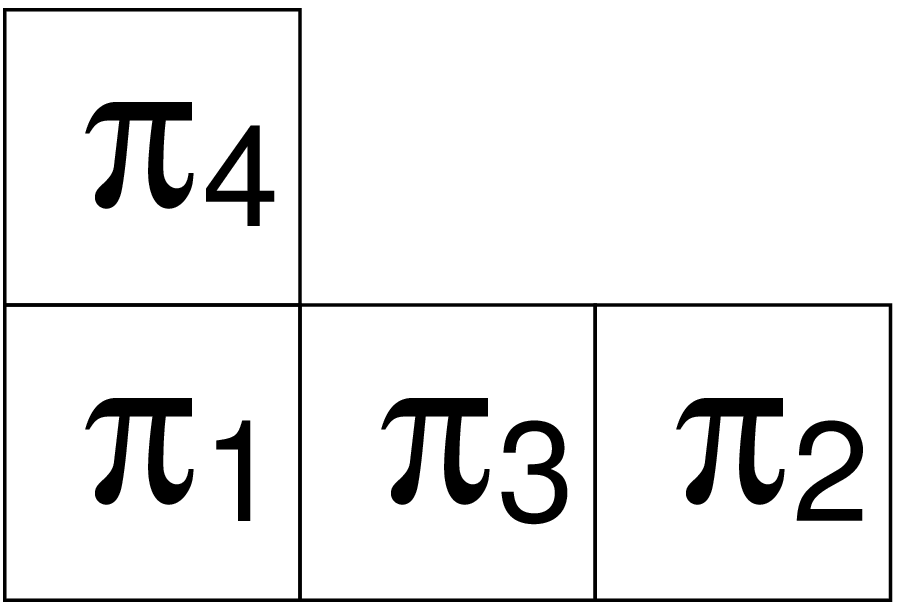}}\ ,
\end{equation}
where $\pi_i$ is the unique path from source $i$ to sink $i$.
These tableaux 
have $2$ and $3$ inversions, respectively.
Since $31^\tr = 211$,  
the tableaux together give 
$\epsilon_q^{211}(q_{e,3421} C'_{3421}(q)) = q^2 + q^3$.
It is clear that all twenty-four $F_{3421}$-tableaux of type $e$ and shape $4$ 
are column-strict.  Counting inversions in these tableaux gives
$\epsilon_q^{1111}(q_{e,3421} C'_{3421}(q)) = 1 + 3q + 8q^2 + 8q^3 + 3q^4 + q^5$.
It is also clear that there are no $F_{3421}$-tableaux of type $e$ and 
shapes $22$, $211$, or $1111$. Thus we have
$\epsilon_q^{22}(q_{e,3421} C'_{3421}(q)) = 
\epsilon_q^{31}(q_{e,3421} C'_{3421}(q)) = 
\epsilon_q^{4}(q_{e,3421} C'_{3421}(q)) = 0$.


Expanding $\epsilon_q^\lambda$ in terms of irreducible characters
and Kostka numbers, $\epsilon_q^\lambda = \sum K_{\mu^\tr,\lambda} \chi_q^\mu$,
and using Haiman's result~\cite[Lem\,1.1]{HaimanHecke},
we have that $\epsilon_q^\lambda( \qew C'_w(q))$ is symmetric and unimodal
about $\qew$ for all $w \in \sn$.
In the case that $w$ \avoidsp, it would be interesting to explain
this phenomenon combinatorially in terms of Theorem~\ref{t:qepsilon}.

\begin{cor}\label{c:edeck}
Fix 
$\lambda = (\lambda_1, \dotsc, \lambda_r) \vdash n$
and $w \in \sn$ \avoidingp.  If $w_1 \cdots w_n$ has 
a decreasing subsequence of length greater than 
$r$,
then 
$\epsilon_q^\lambda(\qew C'_w(q)) = 0$.
\end{cor}
\begin{proof}
Let $(\pi_1,\dotsc,\pi_n)$ be the unique path family of type $e$ which covers
$F_w$.  By Theorem~\ref{t:kstar} 
there exist $r+1$ paths $\pi_i, \dotsc, \pi_{i+r}$
which share a vertex.  No two of these can appear together in a column
of a column-strict $F_w$-tableau. Thus no such tableau has
shape $\lambda^\tr$.
\end{proof}

More generally, 
it is known that we have $\epsilon_q^\lambda(\qew C'_w(q)) = 0$
unless $\lambda \leq \sh(w)$ in the majorization order, where
$\sh(w)$ is the partition associated to $w$ 
by the Robinson-Schensted row insertion algorithm. 
(See, e.g., \cite[Prop.\,4.1\,(3)]{HaimanHecke}.)
This implies Corollary~\ref{c:edeck}, since a decreasing subsequence
of length greater than $r$
in $w_1 \dotsc w_n$ implies that
we have $\lambda \nleq \sh(w)$.

\section{Connections to chromatic symmetric and 
quasisymmetric functions}\label{s:chromsf}


The evaluations of $\sn$-class functions and $\hnq$ traces
at Kazhdan-Lusztig basis elements
are closely related to certain symmetric and quasisymmetric functions
defined by 
Stanley~\cite{StanSymm}
and Shareshian and Wachs~\cite{SWachsChromQ}.

Let 
$\Lambda_n$ 
be the $\mathbb Z$-module of homogeneous 
degree $n$ symmetric functions.
In \cite{StanSymm} Stanley
defined certain {\em chromatic symmetric functions}
$\{ X_G \,|\, G \text{ a simple graph on $n$ vertices}\,\}$ in $\Lambda_n$
and studied expansions of these in various bases of $\Lambda_n$.
Given $G = (V,E)$ and defining $\mathbb P$ to be the set of 
positive integers, we call a function $\kappa: V \rightarrow \mathbb P$
a {\em proper coloring} of $G$ if $\kappa(u) \neq \kappa(v)$ whenever 
$(u,v) \in E$.
Then we have the definition
\begin{equation*}
X_G = \sum_\kappa x_{\kappa(1)} \cdots x_{\kappa(n)},
\end{equation*}
where the sum is over all proper colorings of $G$.
When $G$ is the incomparability graph of a poset $P$, we will
write $X_P = X_{\text{inc}(P)}$.
Stanley showed~\cite[Prop.\,2.4]{StanSymm} that in this case, 
we have the equivalent definition
\begin{equation}\label{eq:cPdef}
X_P = \sum_{\lambda \vdash n} c_{P,\lambda} m_\lambda,
\end{equation}
where $c_{P,\lambda}$ is the number of ordered set partitions of $P$ 
whose blocks are 
chains of cardinalities $\lambda_1, \lambda_2, \dotsc$.
These symmetric functions are related to $\sn$-class function 
evaluations as follows.
\begin{thm}\label{t:chromeq}
Let $P$ be an $n$-element unit interval order, let $v \in \sn$ 
be the 
$312$-avoiding permutation satisfying $\zeta(F_v) = P$
as in the proof of Theorem~\ref{t:dsnuio},
and let $w \in \sn$ be any \pavoiding permutation satisfying $w \sim v$ 
as in (\ref{eq:posetequiv}).
Then we have 
\begin{equation}\label{eq:chromeq}
X_P = \sum_{\lambda \vdash n} \epsilon^\lambda(C'_w(1))m_\lambda.
\end{equation}
\end{thm}
\begin{proof}
It is easy to see that $c_{P,\lambda}$ is equal to the number of column-strict
$P$-tableaux of shape $\lambda^\tr$.  
By Theorems~\ref{t:dsnuio},
\ref{t:thetavw}, this is
the number of column-strict $F_w$-tableaux of shape $\lambda^\tr$.
Thus by Theorem~\ref{t:sncfinterp} (i) we have
$c_{P,\lambda} = \epsilon^\lambda(C'_w(1))$.
\end{proof}
Expanding $X_P$ in other bases of $\mathbb Q \otimes \Lambda_n$, 
including the forgotten basis $\{ f_\lambda \,|\, \lambda \vdash n \}$, 
we see that other class function evaluations appear as coefficients.
\begin{cor}\label{c:chromeq}
Let $P$, $v$, $w$ be as in Theorem~\ref{t:chromeq}.
Then we have
\begin{equation*}
X_P = 
\sum_{\lambda \vdash n} 
\eta^{\lambda}(C'_w(1)) f_\lambda
= \sum_{\lambda \vdash n}  
\chi^{\lambda^\tr} (C'_w(1)) s_\lambda
= \sum_{\lambda \vdash n}  
\frac{(-1)^{n - \ell(\lambda)}}{z_\lambda} \psi^{\lambda}(C'_w(1)) p_\lambda
= \sum_{\lambda \vdash n}  
\phi^{\lambda}(C'_w(1)) e_\lambda.
\end{equation*}
\end{cor}
\begin{proof}
The transition matrices relating the class function bases
$\{\epsilon^\lambda \,|\, \lambda \vdash n \}$,
$\{\eta^\lambda \,|\, \lambda \vdash n \}$, 
$\{\chi^{\lambda^\tr} \,|\, \lambda \vdash n \}$, 
$\{(-1)^{n-\ell(\lambda)}z_\lambda^{-1} \psi^\lambda \,|\, \lambda \vdash n \}$, 
$\{\phi^\lambda \,|\, \lambda \vdash n \}$, 
respectively,
are inverse to those relating the symmetric function bases
$\{m_\lambda \,|\, \lambda \vdash n \}$,
$\{f_\lambda \,|\, \lambda \vdash n \}$,
$\{s_\lambda \,|\, \lambda \vdash n \}$,
$\{p_\lambda \,|\, \lambda \vdash n \}$,
$\{e_\lambda \,|\, \lambda \vdash n \}$,
respectively.
\end{proof}

We remark that Theorem~\ref{t:chromeq} and Corollary~\ref{c:chromeq}
do not hold for arbitrary $w$ and $P$.  
Not all chromatic symmetric functions $X_P$ can
be expressed 
as $\sum_{\lambda \vdash n} \epsilon^\lambda(C'_w(1)) m_\lambda$
for appropriate $w \in \sn$,
nor can all 
symmetric functions of this form
be expressed as $X_P$
for an appropriate labeled poset $P$.


Stanley and Stembridge~\cite[Conj.\,5.1]{StanSymm}, 
\cite[Conj.\,5.5]{StanStemIJT}
conjectured that $X_P$ is elementary nonnegative 
when $P$ is
$(\mathbf{3}+\mathbf{1})$-free, and
Gasharov~\cite[Thm.\,2]{GashInc} proved 
the weaker statement that $X_P$ is Schur nonnegative in this case.
Guay-Paquet~\cite[Thm.\,5.1]{GPModRel} showed that the above conjecture
and result are equivalent to the analogous statements in which $P$ is assumed
to be a unit interval order.
Thus by Theorem~\ref{t:thetavw} and Corollary~\ref{c:chromeq}
the nonnegativity statements
are consequences
of Haiman's conjecture and result that 
for all $w \in \sn$, $\lambda \vdash n$
we have
$\phi_q^\lambda(\qew C'_w(q)) \in \mathbb N[q]$~\cite[Conj.\,2.1]{HaimanHecke} 
and
$\chi_q^\lambda(\qew C'_w(q)) \in \mathbb N[q]$~\cite[Lem.\,1.1]{HaimanHecke}.
Specifically, the nonnegativity statements are obtained from Haiman's by 
restricting to the case that
$w$ 
avoids the pattern $312$
and by specializing at $q=1$.

Perhaps it is simpler to see relationships between the above statements 
once they are reformulated in terms of a common framework of 
$\sn$-class functions and symmetric functions in $\Lambda_n$.
For $g \in \zsn$, 
define $X_g \defeq \sum_{\lambda \vdash n} \epsilon^\lambda(g) m_\lambda$.
Haiman proved that $X_{C'_w(1)}$ is Schur nonnegative for all $w$ 
and conjectured that $X_{C'_w(1)}$ is elementary nonnegative for all $w$.
Gasharov proved Schur nonnegativity for $w$ avoiding the pattern $312$ 
while Stanley and Stembridge conjectured elementary nonnegativity in this case.

Let $\mathrm{QSym}_n$ be the $\mathbb Z$-module 
of homogeneous degree $n$ {\em quasisymmetric functions}
in the commuting indeterminates $x_1, x_2, \dotsc$,
i.e., the generalization of homogeneous degree $n$ symmetric functions 
$\Lambda_n$ in which monomials of the forms
$x_1^{\alpha_1} \cdots x_k^{\alpha_k}$ and
$x_{i_1}^{\alpha_1} \cdots x_{i_k}^{\alpha_k}$ 
are required to 
have the same
coefficient only when $i_1 < \cdots < i_k$.
In \cite[Sec.\,4]{SWachsChromQ}, 
Shareshian and Wachs defined a $q$-analog
$X_{G,q}$ 
of the chromatic symmetric function $X_G = X_{G,1}$, with
$\{ X_{G,q} \,|\, G \text{ a simple labeled graph on $n$ vertices}\,\}$ 
belonging to $\mathbb Z[q] \otimes \mathrm{QSym}_n$,
and studied expansions of these functions in various bases.
Assume
that $V$ is labeled by $[n]$ and let
$\asc(\kappa) = \# \{(u,v) \in E \,|\, u<v \text{ and } \kappa(u)<\kappa(v) \}$.
Then we have the definition
\begin{equation*}
X_{G,q} = \sum_\kappa q^{\asc(\kappa)} x_{\kappa(1)} \cdots x_{\kappa(n)},
\end{equation*}
where the sum is over all proper colorings of $G$.
In \cite[Thm.\,4.5]{SWachsChromQF} 
Shareshian and Wachs showed 
that this function 
is in fact 
symmetric 
with coefficients in $\mathbb Z[q]$
when
$G$ is the incomparability graph of an appropriately labeled
$n$-element unit interval order
$P$.  
Specifically, we require
for each pair $x,y \in P$ satisfying
\begin{equation}\label{eq:altrespect}
\# \{ z \in P \,|\, z < x \} - 
\# \{ z \in P \,|\, z > x \}
<
\# \{ z \in P \,|\, z < y \} - 
\# \{ z \in P \,|\, z > y \},
\end{equation}
that the label of $x$ be less than that of $y$.
(The equivalence of this requirement to that stated in
\cite[Thm.\,4.5]{SWachsChromQF} 
follows from comparison of 
\cite[Props.\,4.1 -- 4.2]{SWachsChromQF}
and the definition of 
{\em natural unit interval order}~\cite[Sec.\,4]{SWachsChromQF} to 
results in \cite[p.\,33]{Fish} and 
\cite[Obs.\,2.1 -- 2.3, Prop.\,2.4]{SkanReed}.)
When $G$ is the incomparability graph of a labeled poset $P$, 
we will write $X_{P,q} = X_{\text{inc}(P),q}$, and 
we may give an alternate definition of
$\{ X_{P,q} \,|\, P \text{ an {\em n}-element poset}\,\}$
which is analogous to (\ref{eq:cPdef}).
(See also \cite[Eq.\,(6.2)]{SWachsChromQF}.)
\begin{prop}\label{p:ascinv}
Let $P$ be a unit interval order, labeled as in (\ref{eq:altrespect}).
Then we have
\begin{equation*}
X_{P,q} = \sum_{\lambda \vdash n} c_{P,\lambda}(q) m_\lambda,
\end{equation*}
where 
\begin{equation*}
c_{P,\lambda}(q) = \sum_U q^{\inv(U)}
\end{equation*}
and the sum is over column-strict $P$-tableaux of shape $\lambda^\tr$.
\end{prop}
\begin{proof}
Each proper coloring $\kappa$ of $\mathrm{inc}(P)$
may be viewed as an assignment of colors to elements of $P$ so that
each subset of elements having a given color forms a chain.
By \cite[Thm.\,4.5]{SWachsChromQF}, $X_{P,q}$ is symmetric.
Thus for $\lambda \vdash n$, the coefficient in $X_{P,q}$ of $m_\lambda$ 
is well-defined: it is the coefficient of 
$x_1^{\lambda_r} x_2^{\lambda_{r-1}} \cdots x_r^{\lambda_1}$,
i.e., the sum of $q^{\asc(\kappa)}$, over all colorings $\kappa$ 
that assign 
color $1$ to a $\lambda_r$-element chain,
color $2$ to a $\lambda_{r-1}$-element chain$,\dotsc,$
color $r$ to a $\lambda_1$-element chain.
Each such coloring corresponds to 
a column-strict 
$P$-tableau $U$ of shape $\lambda$.  Specifically, 
each $\lambda_{r+1-i}$-element chain of color $i$ 
corresponds to column $r+1-i$ of $U$, for $i = 1, \dotsc, r$.
Now observe that a pair $(u,v)$ in $P$, with $u < v$ as integers,
forms an ascent of $\kappa$ if and only if it forms an inversion in $U$.  
Specifically, $(u,v)$ is an edge in $\mathrm{inc}(P)$ 
if and only if $u$, $v$ are incomparable in $P$, and we have
$\kappa(u) < \kappa(v)$ if and only if $u$ appears in a column of $U$
to the right of the column containing $v$.
\end{proof}

Just as
Stanley's 
chromatic symmetric functions $X_P$ are related 
to $\sn$-class function evaluations in Theorem~\ref{t:chromeq},
the Shareshian-Wachs 
chromatic quasisymmeric functions $X_{P,q}$
are related to $\hnq$-trace evaluations.
\begin{thm}\label{t:qchromeq}
Let $P$ be an $n$-element unit interval order labeled as in 
(\ref{eq:altrespect}),
let $v \in \sn$ 
be the corresponding 
$312$-avoiding permutation as in Theorem~\ref{t:dsnuio},
and let $w \in \sn$ be any \pavoiding permutation satisfying $w \sim v$ 
as in (\ref{eq:posetequiv}).
Then we have 
\begin{equation*}
X_{P,q} = \sum_{\lambda \vdash n} \epsilon_q^\lambda(\qew C'_w(q))m_\lambda.
\end{equation*}
\end{thm}
\begin{proof}
By Proposition~\ref{p:ascinv}, $c_{P,\lambda}(q)$ 
is equal to the sum of $q^{\inv(U)}$ over column-strict
$P$-tableaux of shape $\lambda^\tr$.  
By Theorems~\ref{t:dsnuio}, \ref{t:qthetavw}, 
we may 
sum over column-strict 
$F_w$-tableaux.
Now Theorem~\ref{t:qepsilon} gives
$c_{P,\lambda}(q) = \epsilon^\lambda(\qew C'_w(q))$.
\end{proof}

Expanding $X_P$ in other bases of $\mathbb Q[q] \otimes \Lambda_n$
and following the proof of Corollary~\ref{c:chromeq}, 
we see that other trace evaluations appear as coefficients.
\begin{cor}\label{c:qchromeq}
Let $P$, $v$, $w$ be as in Theorem~\ref{t:qchromeq}.
Then we have
\begin{equation*}
\begin{aligned}
X_{P,q} &= 
\sum_{\lambda \vdash n} 
\eta_q^{\lambda}(\qew C'_w(q)) f_\lambda
= \sum_{\lambda \vdash n}  
\chi_q^{\lambda^\tr} (\qew C'_w(q)) s_\lambda\\
&= \sum_{\lambda \vdash n}  
\frac{(-1)^{n - \ell(\lambda)}}{z_\lambda} \psi_q^{\lambda}(\qew C'_w(q)) p_\lambda
= \sum_{\lambda \vdash n}  
\phi_q^{\lambda}(\qew C'_w(q)) e_\lambda .
\end{aligned}
\end{equation*}
\end{cor}

As before,
Theorem~\ref{t:qchromeq} and Corollary~\ref{c:qchromeq}
do not hold for arbitrary $w$ and $P$.  
Not all chromatic symmetric functions $X_{P,q}$ can
be expressed 
as $\sum_{\lambda \vdash n} \epsilon_q^\lambda(\qew C'_w(q)) m_\lambda$
for appropriate $w \in \sn$,
nor can all 
symmetric functions of this form
be expressed as $X_{P,q}$
for an appropriate labeled poset $P$.

Shareshian and Wachs conjectured~\cite[Conj.\,4.9]{SWachsChromQ}
that 
$X_{P,q}$ belongs to 
$\spn_{\mathbb N[q]}\{e_\lambda \,|\, \lambda \vdash n \}$
when $P$ is a unit interval order labeled 
as in (\ref{eq:altrespect}), and proved~\cite[Thm.\,6.3]{SWachsChromQF}
that 
$X_{P,q}$ belongs to
$\spn_{\mathbb N[q]}\{s_\lambda \,|\, \lambda \vdash n \}$
for such posets.
These statements do not hold for
the more general
$(\mathbf 3 + \mathbf 1)$-free posets, since the functions $X_{P,q}$ are
not always symmetric in this case.
Nevertheless,  
the result of Guay-Paquet~\cite[Thm.\,5.1]{GPModRel} shows that 
the statements generalize those of Stanley, Stembridge, and Gasharov
mentioned after Corrolary~\ref{c:chromeq}.
By Theorem~\ref{t:thetavw} and Corollary~\ref{c:qchromeq}, the statements
are special cases (corresponding to $w$ avoiding the pattern $312$)
of Haiman's conjecture and result that 
for all $w \in \sn$, $\lambda \vdash n$
we have
$\phi_q^\lambda(\qew C'_w(q)) \in \mathbb N[q]$~\cite[Conj.\,2.1]{HaimanHecke} 
and
$\chi_q^\lambda(\qew C'_w(q)) \in \mathbb N[q]$~\cite[Lem.\,1.1]{HaimanHecke}.
Shareshian and Wachs also conjectured~\cite[Sec.\,7]{SWachsChromQF},
and Athanasiadis proved~\cite[Thm.\,4]{AthanPSE}
that 
$X_{P,q}$ belongs to 
$\spn_{\mathbb N[q]}
\{(-1)^{n - \ell(\lambda)} z_\lambda^{-1} p_\lambda \,|\, \lambda \vdash n \}$
when $P$ is a unit interval order labeled as in (\ref{eq:altrespect}).
By Theorem~\ref{t:thetavw} and Corollary~\ref{c:qchromeq}
this is equivalent to the assertion that
we have $\psi_q^\lambda(\qew C'_w(q)) \in \mathbb N[q]$ for $w$ avoiding
the pattern $312$.  
Thus this result is a special case of the (unpublished) conjecture
that $\psi_q^\lambda(\qew C'_w(q)) \in \mathbb N[q]$ for all $w \in \sn$,
$\lambda \vdash n$, 
which is a weakening of Haiman's conjecture~\cite[Conj.\,2.1]{HaimanHecke}
since $\psi_q^\lambda$ is a nonnegative linear combination of monomial traces
(\ref{eq:moncfdef}).



Perhaps it is simpler to see relationships between the above statements 
once they are reformulated in terms of a common framework of 
$\hnq$-traces and symmetric functions in $\mathbb Z[q] \otimes \Lambda_n$.
For $g \in \hnq$, 
define $X_g \defeq \sum_{\lambda \vdash n} \epsilon_q^\lambda(g) m_\lambda$.
Haiman proved that $X_{\qew C'_w(q)}$ belongs to 
$\spn_{\mathbb N[q]} \{ s_\lambda \,|\,\lambda \vdash n \}$ for all $w$ 
and conjectured that $X_{\qew C'_w(q)}$ belongs to
$\spn_{\mathbb N[q]} \{ e_\lambda \,|\,\lambda \vdash n \}$ for all $w$
(and therefore to
$\spn_{\mathbb N[q]} \{ (-1)^{n-\ell(\lambda)} z_\lambda^{-1} p_\lambda \,|\,\lambda \vdash n \}$ for all $w$). 
Shareshian and Wachs proved the Schur nonnegativity result for $w$ 
avoiding the pattern $312$ 
and conjectured the elementary and power sum statements in this case.
Athanasiadis proved the power sum statement in this case.

\section{Interpretation of $\chi_q^\lambda(\qew C'_w(q))$}\label{s:hnqinterps}
Combining results in Sections~\ref{s:pathposet}, \ref{s:hnqinterph}, 
\ref{s:chromsf}
with those of
Shareshian and Wachs
now leads to the following $q$-analog of Theorem~\ref{t:sncfinterp} (iii).

\begin{thm}\label{t:qchi}
Let $w \in \sn$ \avoidp.
For $\lambda \vdash n$ we have
\begin{equation}\label{eq:chiqinterp}
\chi_q^\lambda(\qew C'_w(q)) = \sum_U q^{\inv(U)},
\end{equation}
where the sum is over all standard $F_w$-tableaux 
of type $e$ and shape $\lambda$.
\end{thm}
\begin{proof}
Let $P = P(F_w)$.
By Corollary~\ref{c:qchromeq}, $\chi_q^\lambda(\qew C'_w(q))$ is equal 
to the coefficient of $s_{\lambda^\tr}$ in $X_{P,q}$.
By \cite[Thm.\,6.3]{SWachsChromQF}
and Theorems~\ref{t:dsnuio}, \ref{t:qthetavw}, 
this is precisely the claimed sum.
\end{proof}


For example, consider again the descending star network $F_{3421}$ in
(\ref{eq:F3421}).
It is easy to verify that there are exactly two standard
$F_{3421}$-tableaux of type $e$ and shape $31$:
the two column-strict $F_{3421}$-tableaux of type $e$ and shape $31$
in (\ref{eq:epsilontabs31}) are also row-semistrict.
Thus we have
$\chi_q^{31}(q_{e,3421} C'_{3421}(q)) 
= \epsilon_q^{211}(q_{e,3421} C'_{3421}(q)) = q^2 + q^3$.
On the other hand, not all twenty-four $F_{3421}$-tableaux of type $e$ and 
shape $4$ are row-semistrict: the six tableaux with $\pi_4$ immediately
preceding $\pi_1$ are not.
It is easy to verify that the eighteen remaining tableaux
give $\chi_q^{4}(q_{e,3421} C'_{3421}(q)) = 1 + 3q + 5q^2 + 5q^3 + 3q^4 + q^5$.
Since there are no column-strict $F_{3421}$-tableaux of shapes $22$, $211$, or
$1111$, there are no standard $F_{3421}$-tableaux of these shapes either,
and we have
$\chi_q^{22}(q_{e,3421} C'_{3421}(q)) = 
\chi_q^{211}(q_{e,3421} C'_{3421}(q)) = 
\chi_q^{1111}(q_{e,3421} C'_{3421}(q)) = 0$.

Combining Theorems~\ref{t:kstar} and \ref{t:qchi}, we have the following
analog of Corollary~\ref{c:edeck}. 
\begin{cor}\label{c:sdeck}
Fix $\lambda = (\lambda_1, \dotsc, \lambda_r) \vdash n$
and $w \in \sn$ \avoidingp.  If $w_1 \cdots w_n$ has 
a decreasing subsequence of length greater than $\lambda_1$, then
we have $\chi_q^\lambda(\qew C'_w(q)) = 0$.
\end{cor}

More generally, 
it is known that we have $\chi_q^\lambda(\qew C'_w(q)) = 0$
unless $\lambda \geq \sh(w)^\tr$ in the majorization order.
(See the comment following Corollary~\ref{c:edeck}.)
This implies Corollary~\ref{c:sdeck}, since a decreasing subsequence
of length greater than $\lambda_1$
in $w_1 \dotsc w_n$ implies that
we have $\lambda \ngeq \sh(w)^\tr$.


\section{Interpretation of $\psi_q^\lambda(\qew C'_w(q))$}\label{s:hnqinterpp}

Combining results in Sections~\ref{s:pathposet}, \ref{s:hnqinterph}, 
\ref{s:chromsf}
with those of
Shareshian, Wachs, and Athanasiadis
now leads to 
$q$-analogs of Theorem~\ref{t:sncfinterp} (iv-c)--(iv-d).
We will sometimes find it useful to 
reflect path tableaux in a vertical line,
and will write $U^R$ for the {\em reverse} of tableau $U$.
For instance, we have
\begin{equation*}
U = \raisebox{-6mm}{\includegraphics[height=12mm]{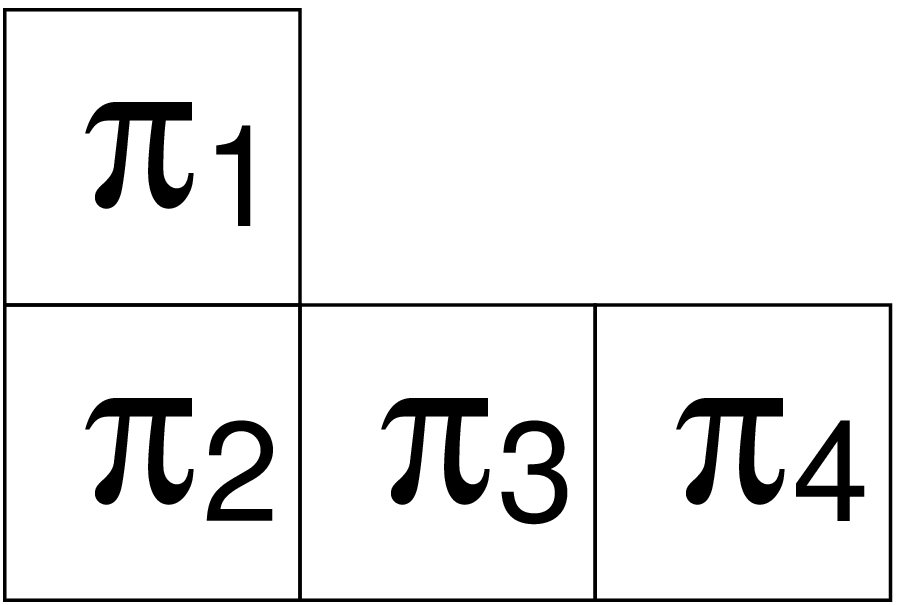}}\ , \qquad
U^R = \raisebox{-6mm}{\includegraphics[height=12mm]{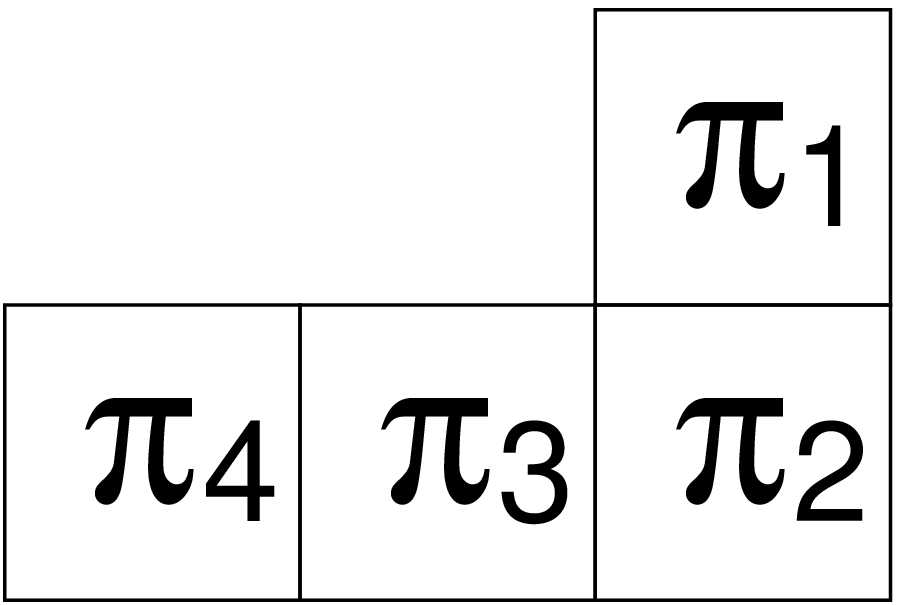}}\ ,
\end{equation*}
where each path $\pi_i$ retains its original source, sink, and orientation.
Thus $U_i^R$ will denote the reverse of the $i$th row of $U$.
Note that while $U^R$ may not be a tableau, 
because its cells are right-justified rather than left justified,
the functions $\inv$ and $\rinv$ may still be applied to $U^R$ as
at the end of Section~\ref{s:pathposet}.
We will also use standard notation for the $q$-analogs of the nonnegative
integers and factorial function.  For $a \in \mathbb N$ we define
$[a]_q = 1 + q + \cdots + q^{a-1}$
for $a \geq 1$, and $[0]_q = 0$.
We also define
$[a]_q! = [a]_q [a-1]_q \cdots [1]_q$ for $a \geq 1$, and $[0]_q! = 1$.

\begin{thm}\label{t:swqpsi}
Let $w \in \sn$ \avoidp. 
For 
$\lambda = (\lambda_1, \dotsc, \lambda_r) \vdash n$, we have
\begin{equation}\label{eq:swqpsi1}
\psi_q^\lambda(\qew C'_w(q)) = \sum_U q^{\inv(U_1 \circ \cdots \circ U_r)},
\end{equation}
where the sum is over all record-free, row-semistrict
$F_w$-tableaux 
of type $e$ and shape $\lambda$, and
\begin{equation}\label{eq:swqpsi2}
\psi_q^\lambda(\qew C'_w(q)) = [\lambda_1]_q \cdots [\lambda_r]_q
\sum_U q^{\inv(U_1^R \circ \cdots \circ U_r^R)},
\end{equation}
where the sum is over all right-anchored, row-semistrict
$F_w$-tableaux 
of type $e$ and shape $\lambda$, 
\end{thm}
\begin{proof}
Let $P = P(F_w)$.  By Corollary~\ref{c:qchromeq}, 
$\psi_q^\lambda(\qew C'_w(q))$ is equal to 
$(-1)^{n-r}z_\lambda$ times 
the coefficient of $p_\lambda$ in $X_{P,q}$.
By \cite[Thm.\,3.1]{AthanPSE} and Theorems~\ref{t:dsnuio}, \ref{t:qthetavw}, 
this is equal to the right-hand side of 
(\ref{eq:swqpsi1}).
By \cite[Lem.\,7.7]{SWachsChromQF} and 
Theorems~\ref{t:dsnuio}, \ref{t:qthetavw}, 
it is also equal to the right-hand side of
(\ref{eq:swqpsi2}). 
\end{proof}

For example, consider the descending star network 
$F_{3421}$ in (\ref{eq:F3421}) 
and the sum in (\ref{eq:swqpsi1}).
It is easy to verify that there are 
eighteen record-free, row-semistrict $F_{3421}$-tableaux 
of type $e$ and shape $31$.
Four of these are
\begin{equation}\label{eq:rfrsstabs31}
\raisebox{-6mm}{\includegraphics[height=12mm]{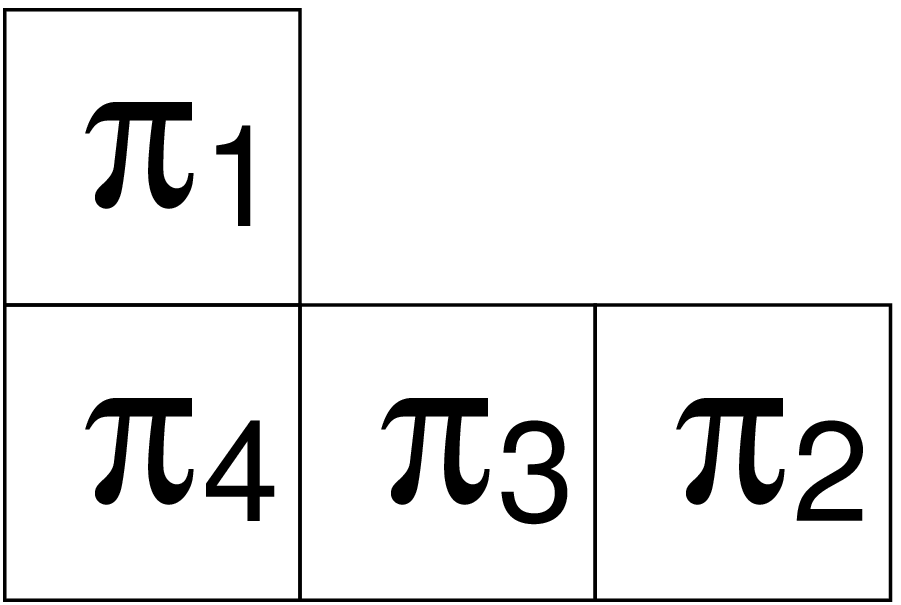}}\ ,\quad
\raisebox{-6mm}{\includegraphics[height=12mm]{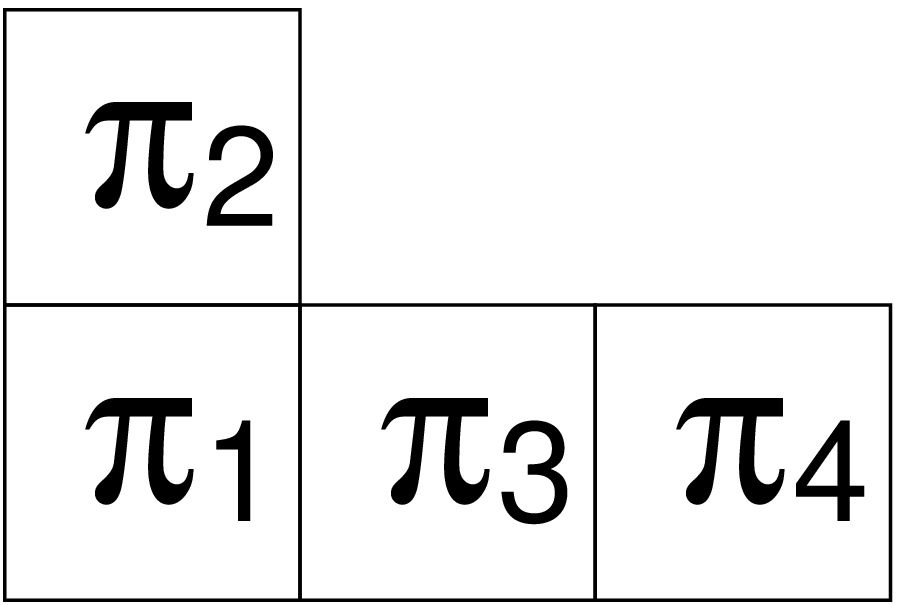}}\ ,\quad
\raisebox{-6mm}{\includegraphics[height=12mm]{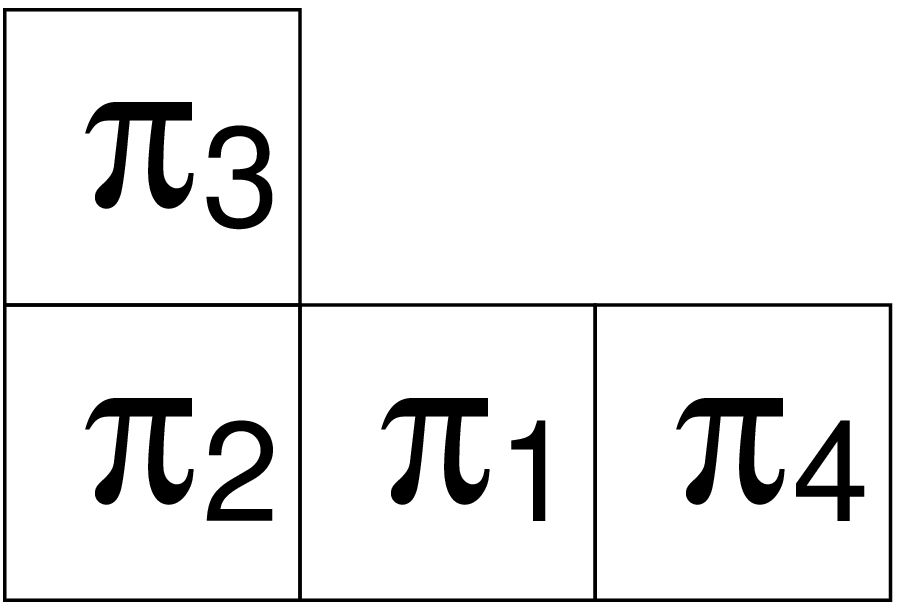}}\ ,\quad
\raisebox{-6mm}{\includegraphics[height=12mm]{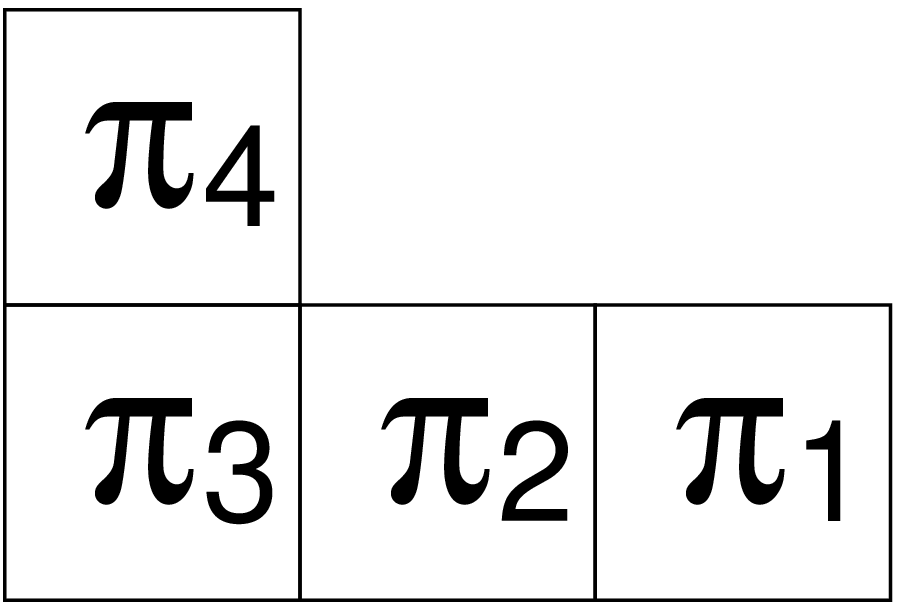}}\ ,
\end{equation}
where $\pi_i$ represents the unique path from source $i$ to sink $i$.
These tableaux $U$ of shape $31$
yield tableaux $U_1 \circ U_2$ of shape $4$,
\begin{equation*}
\raisebox{-3mm}{\includegraphics[height=6mm]{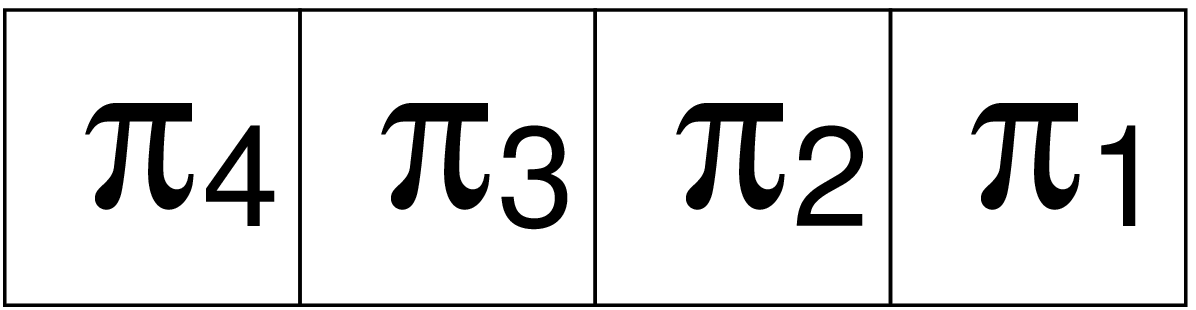}}\ ,\quad
\raisebox{-3mm}{\includegraphics[height=6mm]{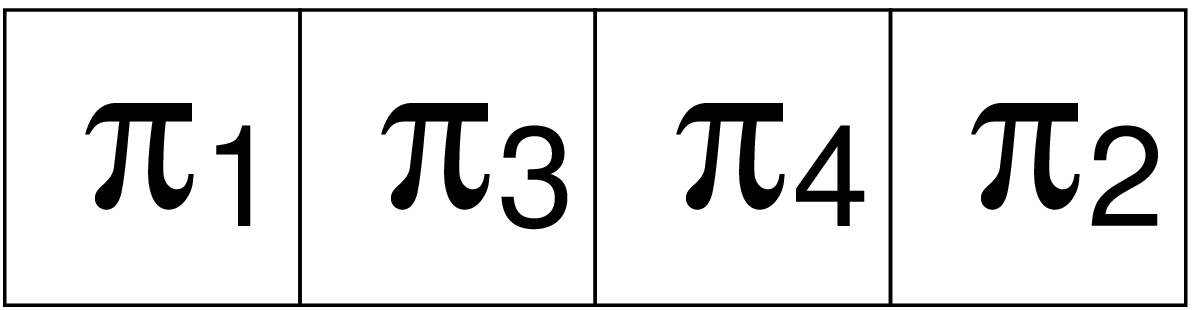}}\ ,\quad
\raisebox{-3mm}{\includegraphics[height=6mm]{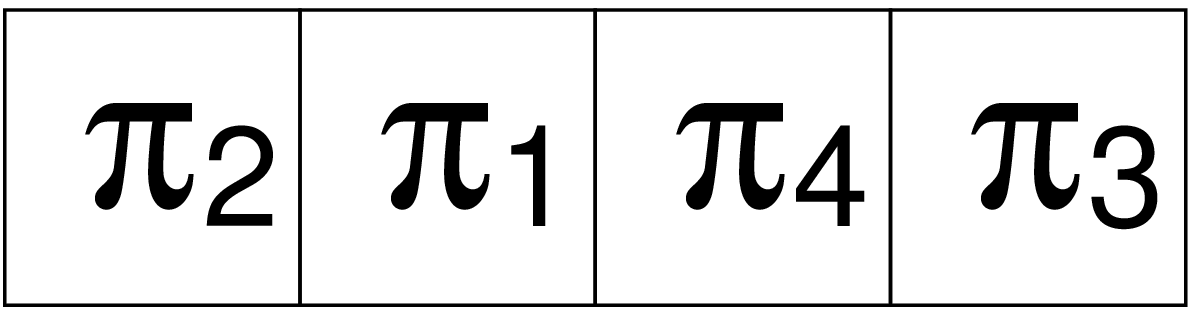}}\ ,\quad
\raisebox{-3mm}{\includegraphics[height=6mm]{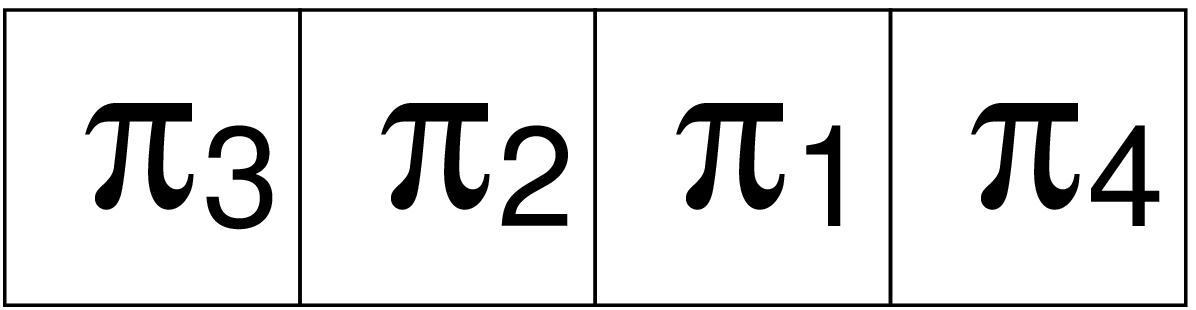}}\ ,
\end{equation*}
which have $5$, $2$, $2$, and $3$ inversions, respectively.
Together, 
they
contribute $2q^2 + q^3 + q^5$ 
to $\psi_q^{31}(q_{e,3421} C'_{3421}(q)) = 1 + 3q + 5q^2 + 5q^3 + 3q^4 + q^5$.
Now consider the sum in (\ref{eq:swqpsi2}).
It is easy to verify that there are six right-anchored 
row-semistrict $F_{3421}$-tableaux of type $e$ and shape $31$: 
the first and fourth tableaux
in (\ref{eq:rfrsstabs31}) and the four tableaux
\begin{equation*}
\raisebox{-6mm}{\includegraphics[height=12mm]{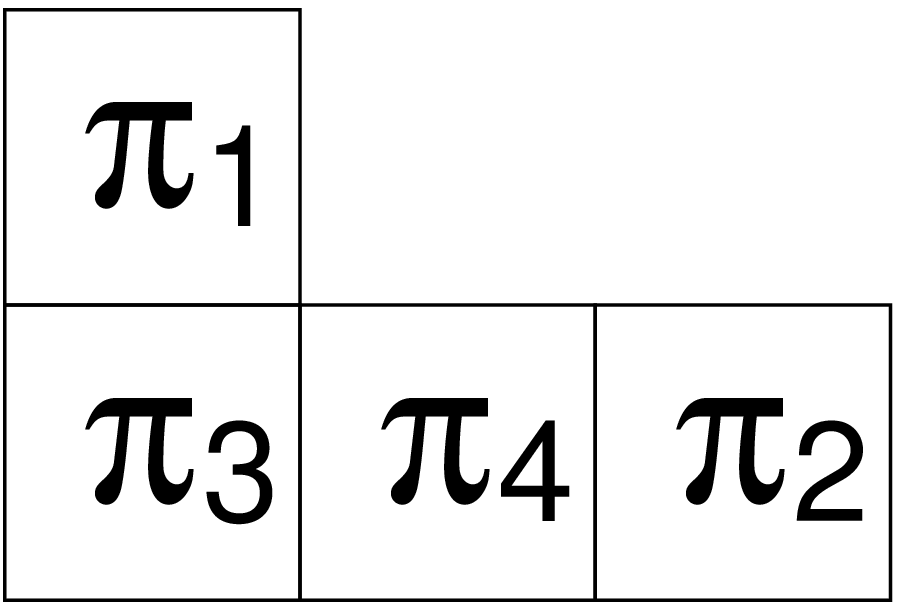}}\ ,\quad
\raisebox{-6mm}{\includegraphics[height=12mm]{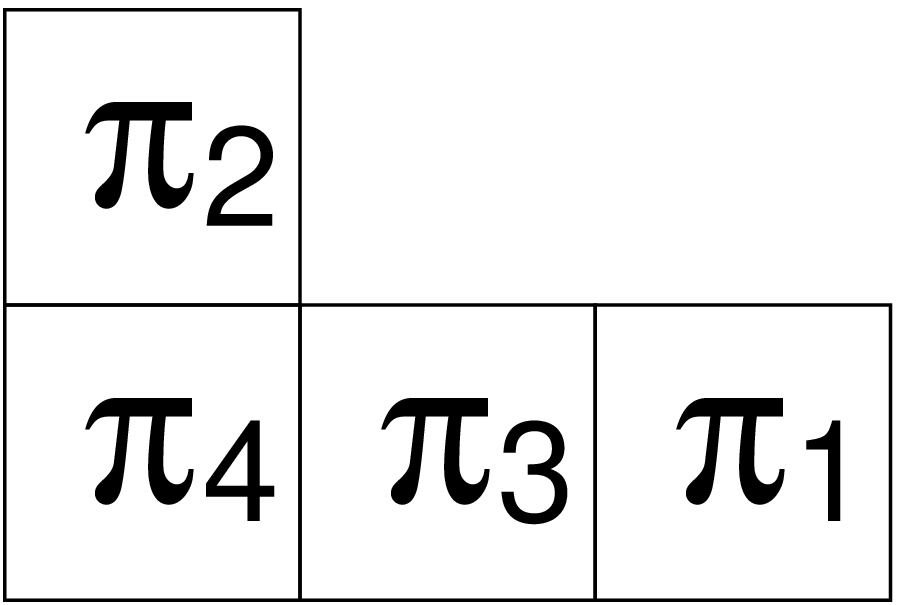}}\ ,\quad
\raisebox{-6mm}{\includegraphics[height=12mm]{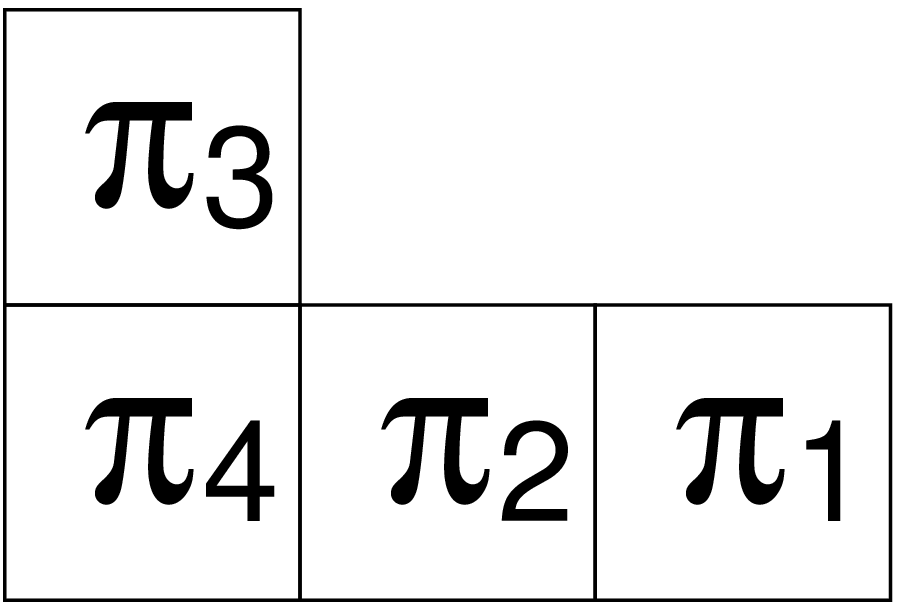}}\ ,\quad
\raisebox{-6mm}{\includegraphics[height=12mm]{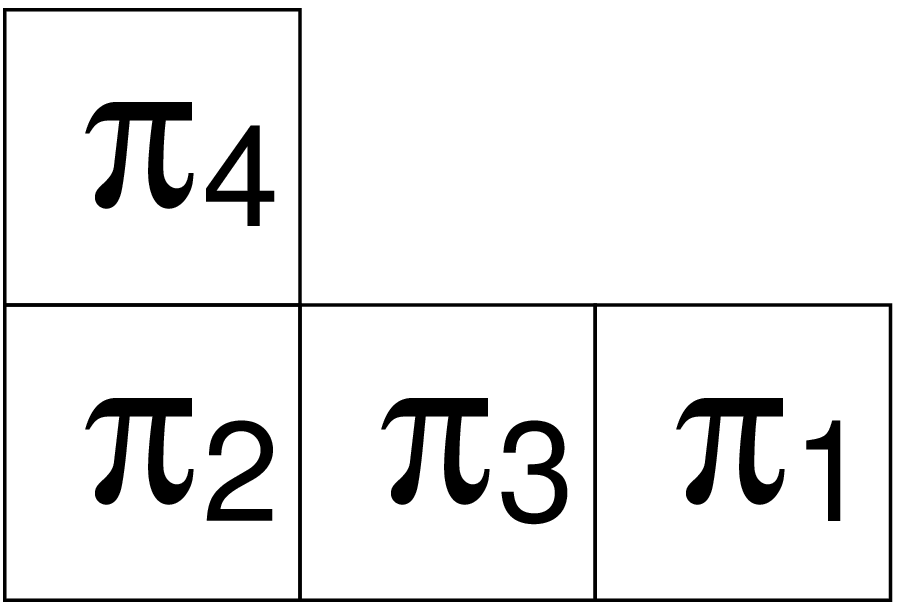}}\ .
\end{equation*}
These tableaux $U$ of shape $31$
yield six tableaux $U_1^R \circ U_2^R$ of shape $4$, 
\begin{equation*}
\begin{gathered}
\raisebox{-3mm}{\includegraphics[height=6mm]{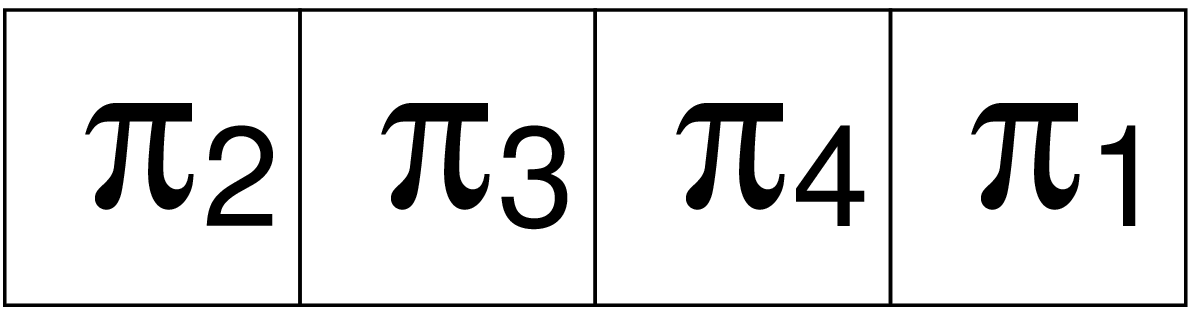}}\ ,\quad
\raisebox{-3mm}{\includegraphics[height=6mm]{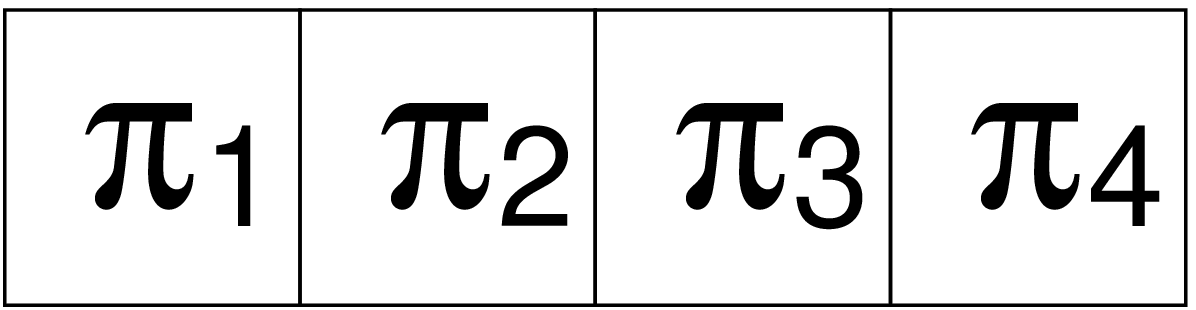}}\ ,\\
\raisebox{-3mm}{\includegraphics[height=6mm]{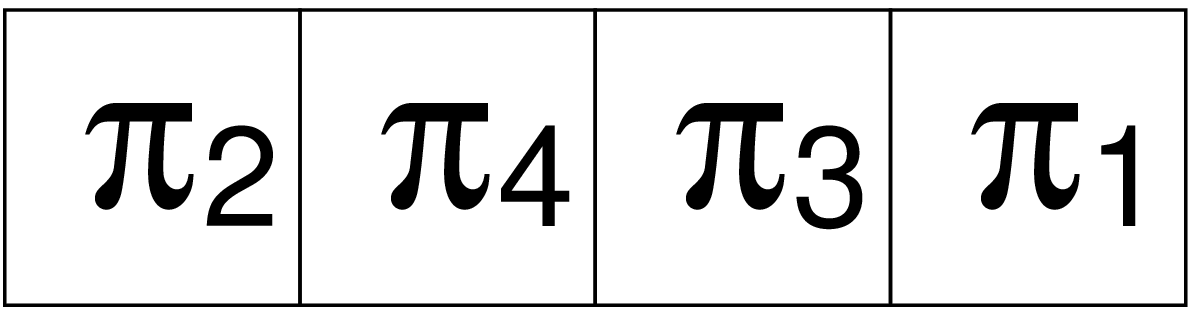}}\ ,\quad
\raisebox{-3mm}{\includegraphics[height=6mm]{tableaux/4_1342}}\ ,\quad
\raisebox{-3mm}{\includegraphics[height=6mm]{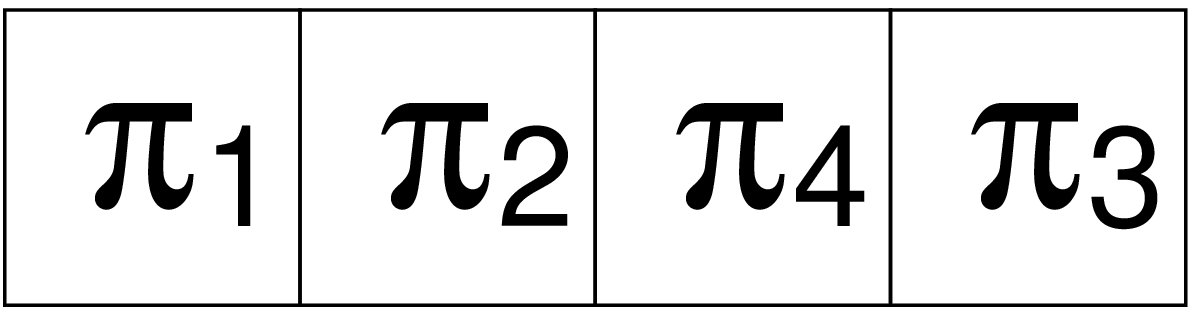}}\ ,\quad
\raisebox{-3mm}{\includegraphics[height=6mm]{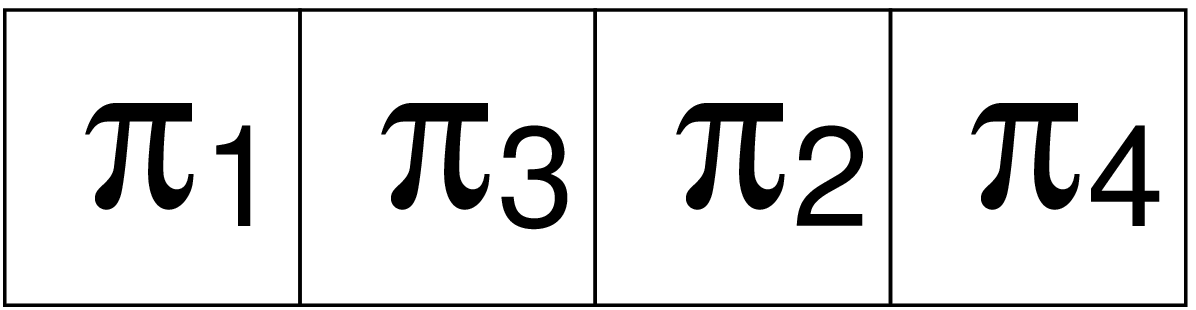}}\ ,
\end{gathered}
\end{equation*}
which have $2$, $0$, $3$, $2$, $1$, and $1$ inversions, respectively.
Together, 
the six tableaux 
contribute $1 + 2q + 2q^2 + q^3$
to 
$[3]_q [1]_q (1 + 2q + 2q^2 + q^3) 
= \psi_q^{31}(q_{e,3421} C'_{3421}(q))$. 

We will state three more combinatorial formulas for 
$\psi_q^\lambda(\qew C'_w(q))$
in Theorems~\ref{t:qpsiO} and \ref{t:qpsi}.
To justify these,
we 
associate a 
polynomial 
$O(F) \in \mathbb N[q]$ 
to a descending star network $F$,
and a 
path tableau $V(F,I)$ 
to the pair $(F,I)$, where $I$ is an ordered set partition 
of $[n]$ of type $\lambda$.

\begin{defn}\label{d:OFdef}
Let $F$ be a descending star network, and let
\begin{equation}\label{eq:concat2}
G_{[c_1,d_1]} \circ \cdots \circ G_{[c_t,d_t]}
\end{equation}
be the concatenation of star networks which corresponds to $F$
as in Section~\ref{s:dsn}.
Define the polynomial
$O(F) \in \mathbb N[q]$ by
\begin{equation}\label{eq:OFdef}
O(F) = \begin{cases}
\frac{\displaystyle{\prod_{i=1}^t}[d_i - c_i]_q!}
{\displaystyle
{
\prodsb{[c_i,d_i] \prec \ntnsp \cdot [c_j,d_j]\\ c_i < c_j}
\nTksp \nTksp \ntksp \ntnsp
[d_i - c_j]_q!
\prodsb{[c_i,d_i] \prec \ntnsp \cdot [c_j,d_j]\\ c_j < c_i}
\nTksp \nTksp \ntksp \ntnsp
[d_j - c_i]_q!
}}
&\text{if $F$ is connected},\\
0 &\text{if $F$ is disconnected}.
\end{cases}
\end{equation}
\end{defn}

For example, the connected descending star network $F_{3421}$ in (\ref{eq:F3421})
corresponds to the concatenation $G = G_{[2,4]} \circ G_{[1,3]}$
and two-element poset of intervals $[2,4] \precdot [1,3]$.  
Thus we have
\begin{equation*}
O(F_{3421}) 
= \frac{[4-2]_q! [3-1]_q!}{[3-2]_q!} 
= \frac{[2]_q! [2]_q!}{[1]_q!} 
= (1+q)^2.
\end{equation*}
Note that for the identity element $e \in \sn$ we have
$O(F_e) = 1$ if $n=1$ and $O(F_e) = 0$ otherwise.
Letting $\pi = (\pi_1, \dotsc, \pi_n)$ 
be the unique path family 
of type $e$ covering $F$, define
$V(F,I)$ to be the unique (row-semistrict) $\pi$-tableau of shape $\lambda$
for which $L(V(F,I))$ is a row-strict Young tableau 
containing indices $I_j$ in row $j$.
For $S$ a $k$-element subset of $[n]$, 
let $F|_S$ denote the zigzag network of order $k$ 
isomorphic to the subnetwork 
of $F$ 
covered by paths $\{\pi_i \,|\, i \in S \}$.

Also essential to our proofs of Theorems~\ref{t:qpsiO} and \ref{t:qpsi}
is a map 
\begin{multline*}
\iota: 
\{ w \in \sn \,|\, w \text{ \avoidsp, $F_w$ connected}\,\} \\
\rightarrow
\{ w \in \mfs{n-1} \,|\, w \text{ \avoidsp, $F_w$ connected}\,\}.
\end{multline*}
Let $F_w$ be a connected zig-zag network
of order $n \geq 2$
corresponding to the concatenation (\ref{eq:concat2}).
By Observation~\ref{o:firstorlast} we may assume that $t=1$,
or $d_1 = n$ and $[c_1, d_1] \precdot [c_2, d_2]$, 
or $d_t = n$ and $[c_{t-1}, d_{t-1}] \precdot [c_t, d_t]$.
We declare
$\iota(w)$ to be the permutation whose descending star network 
$F_{\iota(w)}$ of order $n-1$
is obtained from $F_w$
by deleting the path from source $n$ to sink $n$,
and, in the case that 
$d_1 = n$ and $d_2 = n-1$
($d_t = n$ and $d_{t-2} = n-1$),
by contracting one edge whose vertices correspond to the central vertices of 
$G_{[c_1,n]}$ and $G_{[c_2, n-1]}$
($G_{[c_{t-1},n-1]}$ and $G_{[c_t, n]}$).
Equivalently, $F_{\iota(w)}$ is the zig-zag network 
corresponding to the concatenation
\begin{equation}\label{eq:iotadef}
\begin{cases} 
G_{[c_1, n-1]} &\text{if $t = 1$},\\
G_{[c_2,d_2]} \circ \cdots \circ G_{[c_t,d_t]} 
&\text{if $d_1 = n$ and $d_2 = n-1$},\\
G_{[c_1,d_1]} \circ \cdots \circ G_{[c_{t-1},d_{t-1}]} 
&\text{if $d_t = n$ and $d_{t-1} = n-1$},\\
G_{[c_1,n-1]} \circ G_{[c_2,d_2]} \circ \cdots \circ G_{[c_t,d_t]} 
&\text{if $d_1 = n$ and $d_2 < n-1$},\\
G_{[c_1,d_1]} \circ \cdots \circ G_{[c_{t-1},d_{t-1}]} \circ G_{[c_t, n-1]} 
&\text{if $d_t = n$ and $d_{t-1} < n-1$}.
\end{cases}
\end{equation}


For example, let $n=6$. 
$F_{256431}$ corresponds to the concatenation
$G_{[3,6]} \circ G_{[2,5]} \circ G_{[1,2]}$, with $d_2 = 5 = n-1$.
Removing the path from source $6$ to sink $6$,
and contracting the edge whose endpoints correspond to the central vertices
of $G_{[3,6]}$ and $G_{[2,5]}$, 
we obtain $F_{\iota(256431)}$, which can be shown 
(as in the example preceding Theorem~\ref{t:lem53nndcb})
to be $F_{25431}$.  
\begin{equation}\label{eq:F256431}
F_{256431} = \raisebox{-10mm}{
\includegraphics[height=22mm]{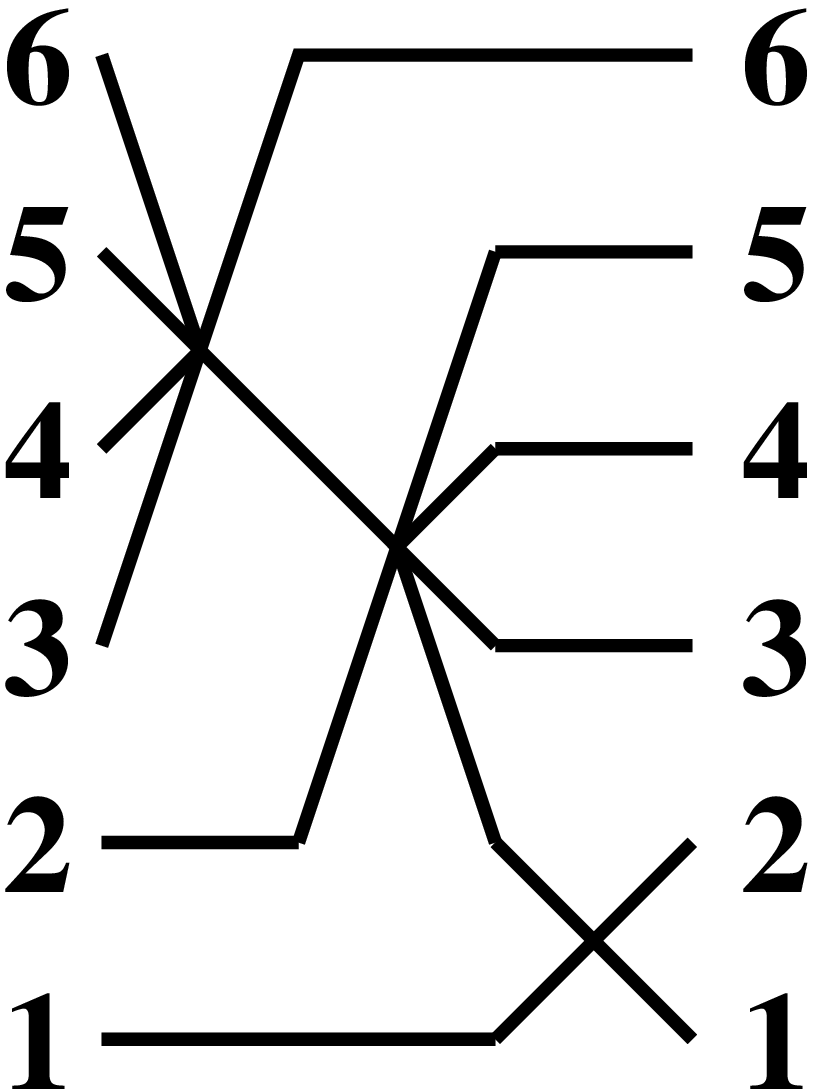}}\ \mapsto
\raisebox{-10mm}{
\includegraphics[height=18mm]{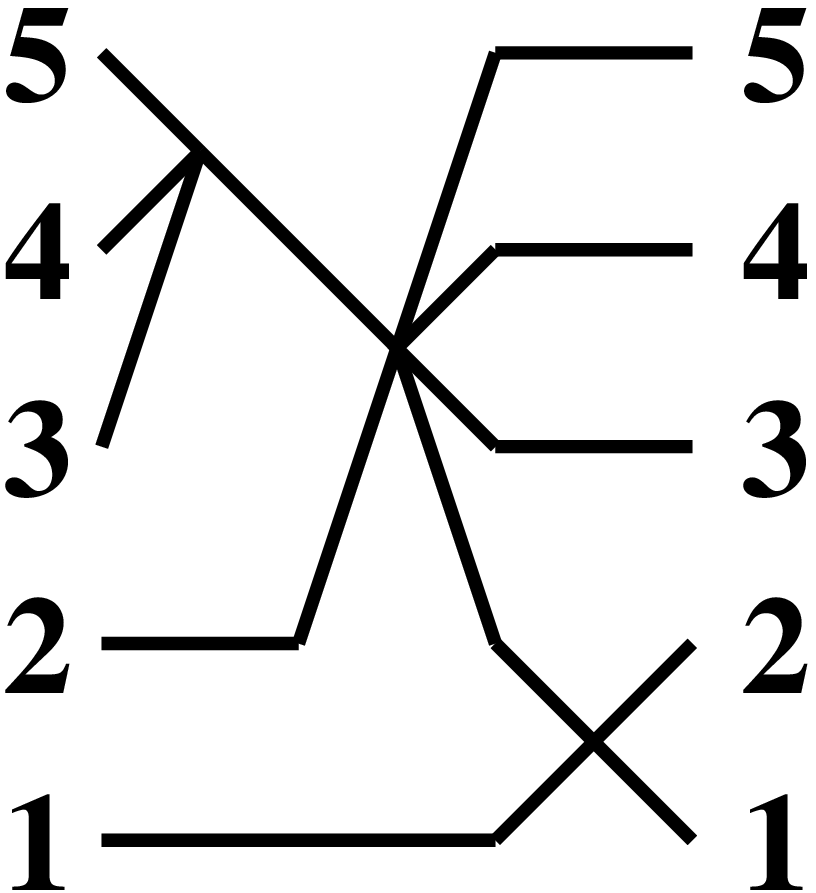}}\ \cong
\raisebox{-10mm}{
\includegraphics[height=18mm]{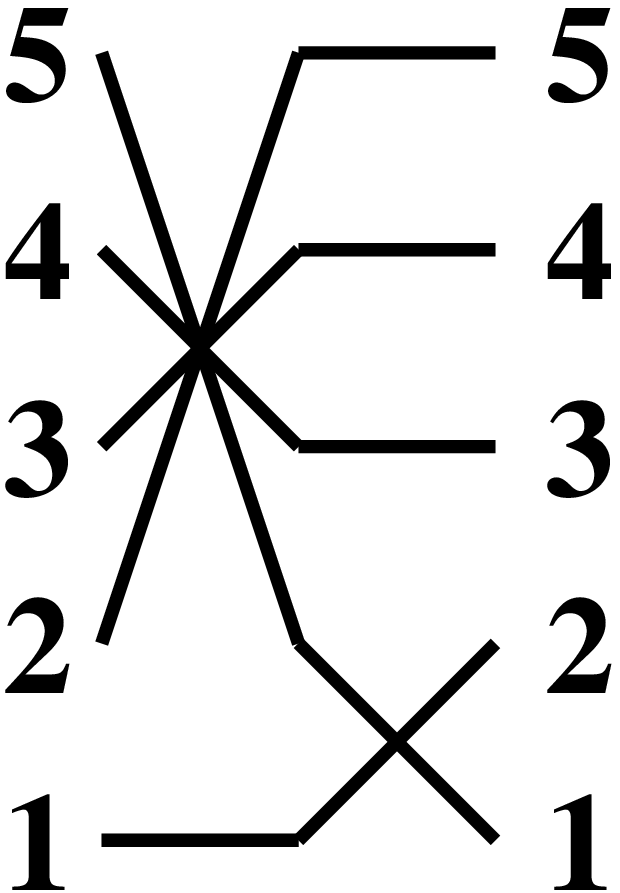}}\
= F_{25431}.
\end{equation}
Similarly,
$F_{246531}$ corresponds to the concatenation
$G_{[3,6]} \circ G_{[2,4]} \circ G_{[1,2]}$, with $d_2 = 4 < n-1$.
Removing the path from source $6$ to sink $6$,
we obtain $F_{\iota(246531)}$, which can be shown 
to be $F_{24531}$.  
\begin{equation}\label{eq:F246531}
F_{246531} = \raisebox{-10mm}{
\includegraphics[height=22mm]{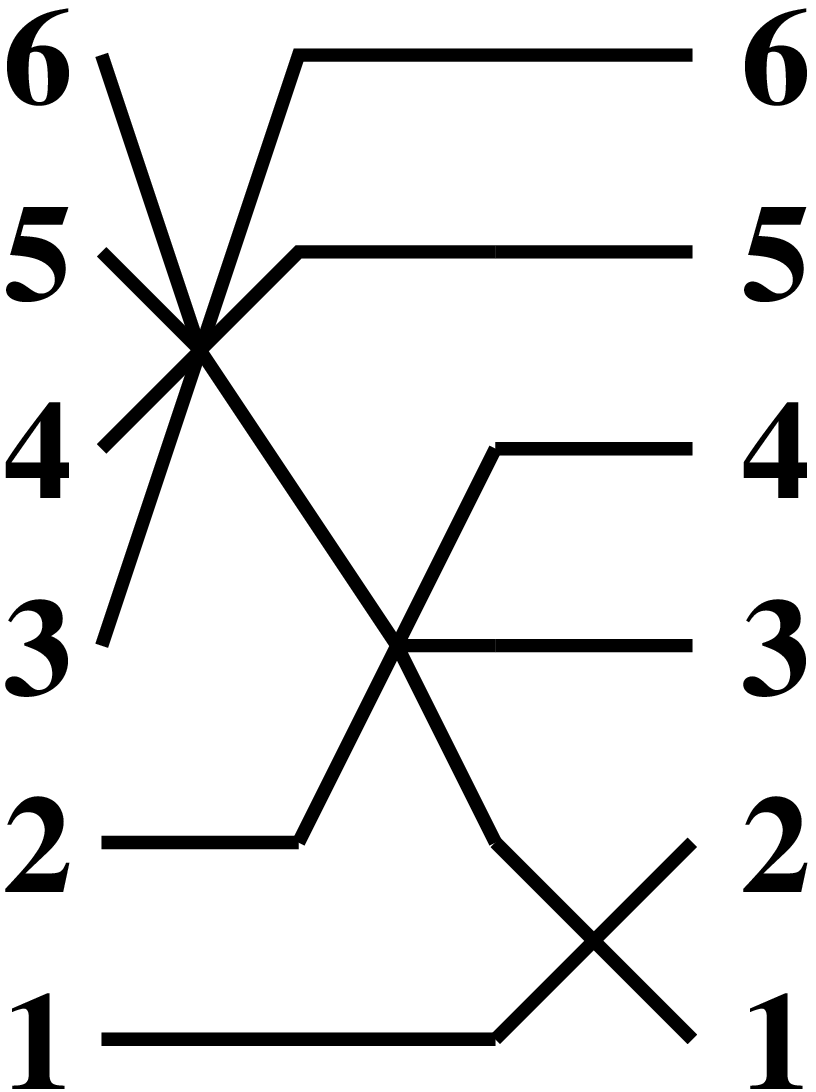}}\ \mapsto
\raisebox{-10mm}{
\includegraphics[height=18mm]{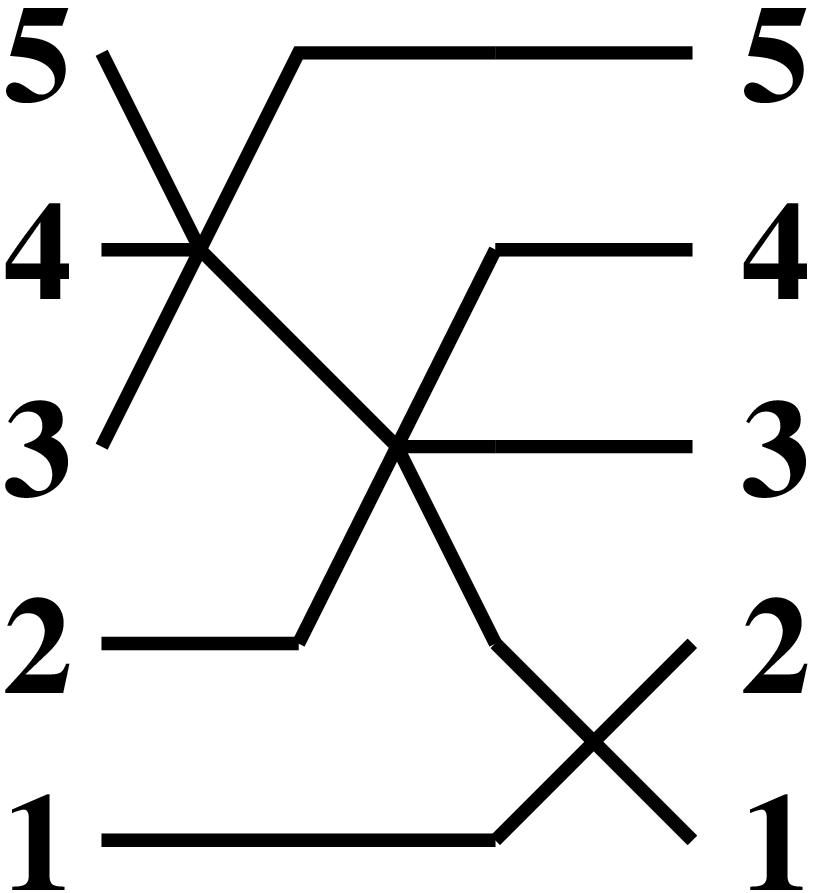}}\ 
= F_{24531}.
\end{equation}

\ssec{Right-anchored, row-semistrict $F$-tableaux and $O(F)$}

Inversions in right-anchored, row-semistrict 
path tableaux are closely related to the polynomials
defined in Definition~\ref{d:OFdef}.
In order to state this relationship precisely (Lemma~\ref{l:ranchO})
and state a third combinatorial formula for $\psi_q^\lambda(\qew C'_w(q))$
(Theorem~\ref{t:qpsiO}), 
we define a family of sets 
$\{Z(F) \,|\, F \text{ a zig-zag network } \}$ of tableaux
and maps between these.
For $F$ a zig-zag network of order $m$, let $Z(F)$ be the set of
right-anchored, row-semistrict $F$-tableaux of type $e$ and shape $(m)$.
Note that if $F$ is not connected, then $Z(F) = \emptyset$, since a 
right-anchored $F$-tableau cannot be row-semistrict when $F$ is disconnected.

Now let $F_w$ be a connected zig-zag network of order $m$ which corresponds
to the concatenation (\ref{eq:concat2}), and let $[b,m]$ be the unique interval
in (\ref{eq:concat2})
to contain $m$.
Define a map
\begin{equation*}
\begin{aligned}
\gamma: Z(F_w) &\rightarrow Z(F_{\iota(w)}) \times \{ 0, 1, \dotsc, m-b-1 \}\\
U &\mapsto (U',k),
\end{aligned}
\end{equation*}
by declaring $U'$ to be the tableau obtained from $U$ by deleting $\pi_n$,
and by declaring $k$ to be the number of indices in the interval $[b,m-1]$
appearing to the left of $m$ in $L(U)$.

For example, consider the 
network $F_{256431}$ in (\ref{eq:F256431})
and let $U$ be the tableau
\begin{equation*}
\raisebox{-3mm}{\includegraphics[height=6mm]{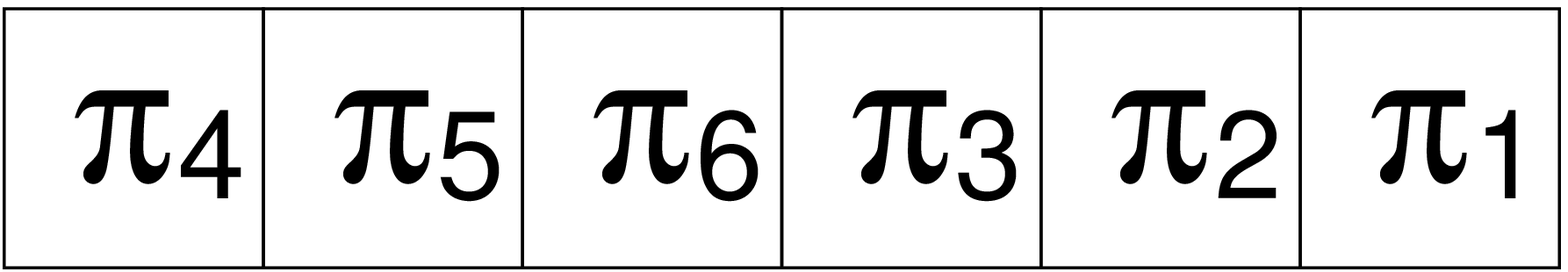}}\ .
\end{equation*}
Then the unique interval in $G$ containing $m = 6$ is $[3,6]$, and there 
are two indices in this interval 
appearing to the left of $6$ in $L(U)$.
Thus $\gamma(U) = (U', 2)$ where $U'$ is the tableau
\begin{equation*}
\raisebox{-3mm}{\includegraphics[height=6mm]{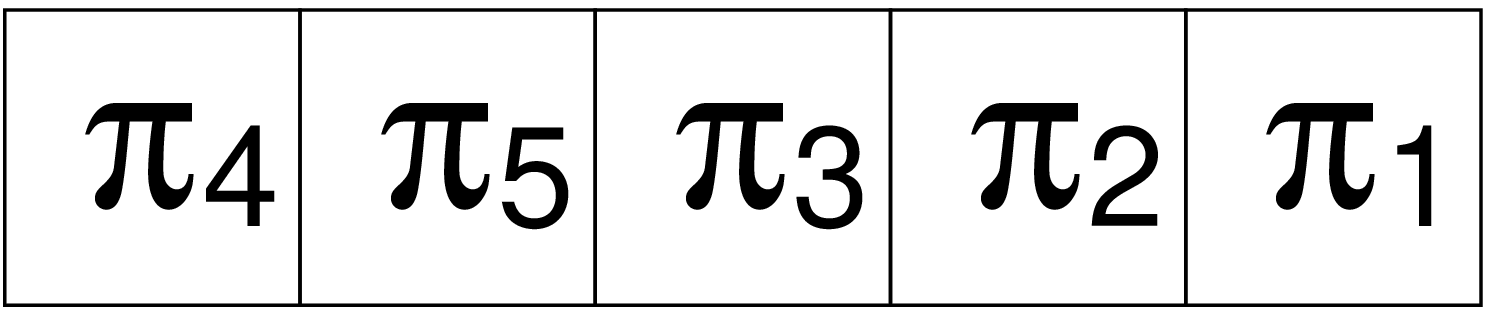}}\ .
\end{equation*}

\begin{lem}\label{l:gamma}
For each connected zig-zag network $F_w$ of order $m$, the map $\gamma$
is a bijection.  Furthermore, if $\gamma(U) = (U', k)$ then 
$\inv(U^R) = \inv(U'^R) + k$. 
\end{lem}
\begin{proof}
To see that $\gamma$ is well-defined, fix $U \in Z(F_w)$ and let
$L(U) = (i_1, \dotsc, i_m = 1)$, where $i_j = m$ and $j < m$.
Clearly $U'$ is right-anchored. If $j = 1$ then 
$U' = (\pi_{i_2}, \dotsc, \pi_{i_m})$
is row-semistrict.  If $j \geq 1$, then 
$U' = (\pi_{i_1}, \dotsc, \pi_{i_{j-1}}, \pi_{i_{j+1}}, \dotsc, \pi_{i_m})$ 
is also row-semistrict,
since $\pi_m \not >_{P(F_w)} \pi_{i_{j+1}}$ implies that 
$\pi_{i_{j-1}} \not >_{P(F_{\iota(w)})} \pi_{i_{j+1}}$.

To invert $\gamma$, find an entry $i_d$ of $L(U')$ which belongs to
$[b,m-1]$ and has exactly $k$ 
indices from the interval $[b,m-1]$ to its left.
(This is possible, since $0 \leq k \leq m - b - 1$.)
Now create a new $F_w$-tableau by inserting $\pi_m$
into $U'$
immediately before $\pi_{i_d}$.
This map is well-defined, because $\pi_{i_d}$ intersects $\pi_m$.
It is clear that the map inverts $\gamma$.  
Since $U$ has type $e$, it is clear that the number of inversions in $U^R$
involving $\pi_m$ is equal to the number of indices in the interval $[b, m-1]$
appearing to the left of $m$ in $L(U)$.  
It follows that $\inv(U^R) = \inv(U'^R) + k$.
\end{proof}

By the above lemma, we can interpret $O(F)$ as a generating function
for inversions in tableaux belonging to $Z(F)$.
\begin{lem}\label{l:ranchO}
For each $w \in \mfs m$ \avoidingp, we have
\begin{equation}\label{eq:ranchO}
\sum_{U \in Z(F_w)} q^{\inv(U^R)} = O(F_w).
\end{equation}
\end{lem}
\begin{proof}
When $m = 1$, both sides of (\ref{eq:ranchO}) are $1$.
Now assume that (\ref{eq:ranchO}) holds for all zig-zag networks
corresponding to \pavoiding permutations in $\mfs 1, \dotsc, \mfs{m-1}$,
and consider $w \in \mfs m$ \avoidingp.  Let $F_w$ correspond to the 
concatenation (\ref{eq:concat2}).
If $F_w$ is disconnected, then $Z(F_w) = \emptyset$ and both sides of 
(\ref{eq:ranchO}) are $0$.
If $F_w$ is connected, let $[b,m]$ be the unique interval in (\ref{eq:concat2})
to contain $m$.  
By induction and Lemma~\ref{l:gamma} we have
\begin{equation}\label{eq:Oinduction}
\sum_{U \in Z(F_w)} q^{\inv(U^R)} = 
\sum_{k=0}^{m-b-1} q^k \sum_{U' \in Z(F_{\iota(w)})} q^{\inv(U'^R)}
= [m-b]_q O(F_{\iota(w)}).
\end{equation}
By Observation~\ref{o:firstorlast} we may assume that we have 
$t=1$, or
$[b,m] = [c_1,d_1] \precdot [c_2,d_2]$, 
or $[c_{t-1},d_{t-1}] \precdot [c_t,d_t] = [b,m]$.
In the first case, the expression (\ref{eq:Oinduction}) is
$[m-b]_q!$.
In the second case,
by Definition~\ref{d:OFdef}
and (\ref{eq:iotadef}), it is
\begin{equation}\label{eq:OandO}
\begin{cases}
[m-c_1]_q 
\frac{
{\displaystyle[m-1-c_1]_q! [d_2-c_2]_q! \cdots [d_t-c_t]_q!}}
{\displaystyle
{\prodsb{[c_i,d_i] \prec \ntnsp \cdot [c_j,d_j]\\ c_i < c_j}
\nTksp \nTksp \ntksp \ntnsp
[d_i - c_j]_q!
\prodsb{[c_i,d_i] \prec \ntnsp \cdot [c_j,d_j]\\ c_j < c_i}
\nTksp \nTksp \ntksp \ntnsp
[d_j - c_i]_q!}}
&\text{if $d_1 = m$, $d_2 < m-1$},\\
\displaystyle{
\frac{[m-c_1]_q!}{[d_2 - c_1]_q!} 
}
\frac{
[d_2-c_2]_q! \cdots [d_t-c_t]_q!}
{\displaystyle{
\prodsb{[c_i,d_i] \prec \ntnsp \cdot [c_j,d_j]\\ c_i < c_j}
\nTksp \nTksp \ntksp \ntnsp
[d_i - c_j]_q!
\prodsb{[c_i,d_i] \prec \ntnsp \cdot [c_j,d_j]\\ c_j < c_i \\ i \geq 2}
\nTksp \nTksp \ntksp \ntnsp
[d_j - c_i]_q!}}
&\text{if $d_1 = m$, $d_2 = m-1$}.
\end{cases}
\end{equation}
In the third case, we obtain an expression similar to (\ref{eq:OandO}).
In all cases, the expression 
is equal to $O(F_w)$.
\end{proof}

Now we can state the precise relationship between the polynomials
$O(F)$ and inversions in right-anchored, row-semistrict path tableaux.

\begin{thm}\label{t:qpsiO}
Let $w \in \sn$ \avoidp.
For $\lambda = (\lambda_1, \dotsc, \lambda_r) \vdash n$, we have
\begin{equation}\label{eq:qpsi3}
\psi_q^\lambda(\qew C'_w(q)) = 
[\lambda_1]_q \cdots [\lambda_r]_q
\sumsb{I \vdash [n]\\ \type(I) = \lambda}
q^{\inv(V(F_w,I)_1 \circ \cdots \circ V(F_w,I)_r)} O(F_w|_{I_1}) \cdots O(F_w|_{I_r}).
\end{equation}
\end{thm}
\begin{proof}
By Theorem~\ref{t:swqpsi}, $\psi_q^\lambda(\qew C'_w(q))$
is equal to 
the right-hand side of (\ref{eq:swqpsi2}).
Grouping terms in the sum and using (\ref{eq:invdecomp}),
we may rewrite this expression as
\begin{equation}\label{eq:qpsi4}
[\lambda_1]_q \cdots [\lambda_r]_q
\nTksp \ntksp \sumsb{I \vdash [n]\\ \type(I) = \lambda} \nTksp
\sum_U q^{\inv(U_1^R \circ \cdots \circ U_r^R)} = 
[\lambda_1]_q \cdots [\lambda_r]_q
\nTksp \ntksp \sumsb{I \vdash [n]\\ \type(I) = \lambda} \nTksp
\sum_U q^{\inv(U_1^R) + \cdots + \inv(U_r^R) + \inv(U^\tr)},
\end{equation}
where 
$U$ now varies over the subset of right-anchored, row-semistrict
tableaux of type $e$ and shape $\lambda$ satisfying
$U_j = I_j$ for 
each component of the appropriate ordered set partition
$I = (I_1,\dotsc,I_r)$.  
For fixed $I$, this inner sum can be rewritten as
\begin{equation*}
\sum_{W^{(1)}}q^{\inv((W^{(1)})^R)} \cdots \sum_{W^{(r)}}q^{\inv((W^{(r)})^R)}
\sum_U q^{\inv(U^\tr)},
\end{equation*}
where $W^{(j)}$ varies over right-anchored, row-semistrict $F_w|_{I_j}$-tableaux
of shape $\lambda_j$ and type $e \in \mfs{\lambda_j}$,
and $U$ again varies as in (\ref{eq:qpsi4}).
By Lemma~\ref{l:ranchO}, the first $r$ sums are equal to
$O(F_w|_{I_1}), \dotsc, O(F_w|_{I_r})$, and it
is easy to see that for any tableau $U$ in the last sum, we have
$\inv(U^\tr) = \inv(V(F_w,I)_1 \circ \cdots \circ V(F_w,I)_r)$.
Thus we obtain
the right-hand side of (\ref{eq:qpsi3}).
\end{proof}


For example, consider again the descending star network 
$F_{3421}$ in (\ref{eq:F3421}) and the expression in (\ref{eq:qpsi3}).
The ordered set partitions of $[4]$
of type $31$ are $123|4$, $124|3$, $134|2$, and $234|1$.
Corresponding to the set partitions $I$ are the tableaux $V(F_{3421},I)$
of shape 31
\begin{equation*}\label{eq:lacyltabs31}
\raisebox{-6mm}{\includegraphics[height=12mm]{tableaux/31_4123}}\ ,\quad
\raisebox{-6mm}{\includegraphics[height=12mm]{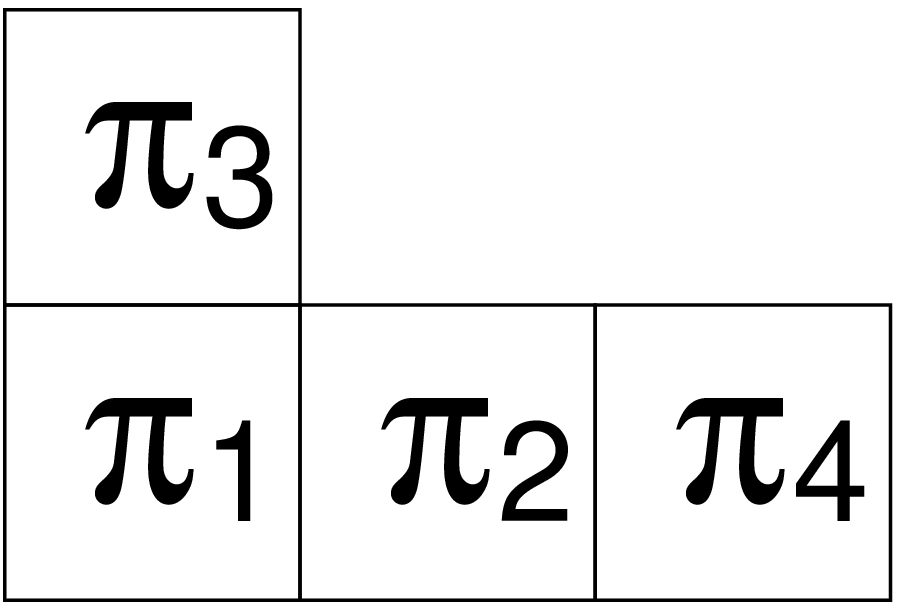}}\ ,\quad
\raisebox{-6mm}{\includegraphics[height=12mm]{tableaux/31_2134}}\ ,\quad
\raisebox{-6mm}{\includegraphics[height=12mm]{tableaux/31_1234}}\ ,
\end{equation*}
respectively,
where $(\pi_1, \pi_2, \pi_3, \pi_4)$ is the unique path family of type $e$
covering $F_{3421}$.  
These in turn yield tableaux $V(F_{3421},I)_1 \circ V(F_{3421},I)_2$ of shape 4
\begin{equation*}
\raisebox{-3mm}{\includegraphics[height=6mm]{tableaux/4_1234}}\ ,\quad
\raisebox{-3mm}{\includegraphics[height=6mm]{tableaux/4_1243}}\ ,\quad
\raisebox{-3mm}{\includegraphics[height=6mm]{tableaux/4_1342}}\ ,\quad
\raisebox{-3mm}{\includegraphics[height=6mm]{tableaux/4_2341}}\ ,
\end{equation*}
having $0$, $1$, $2$, and $2$ inversions, respectively.
The subnetworks 
$F_{3421}|_{123}$,
$F_{3421}|_{124}$,
$F_{3421}|_{134}$, 
$F_{3421}|_{234}$ and polynomials
$O(F_{3421}|_{123})$,
$O(F_{3421}|_{124})$,
$O(F_{3421}|_{134})$, 
$O(F_{3421}|_{234})$ 
are
\begin{equation*}
\begin{alignedat}{4}
\raisebox{-8mm}{\includegraphics[height=16mm]{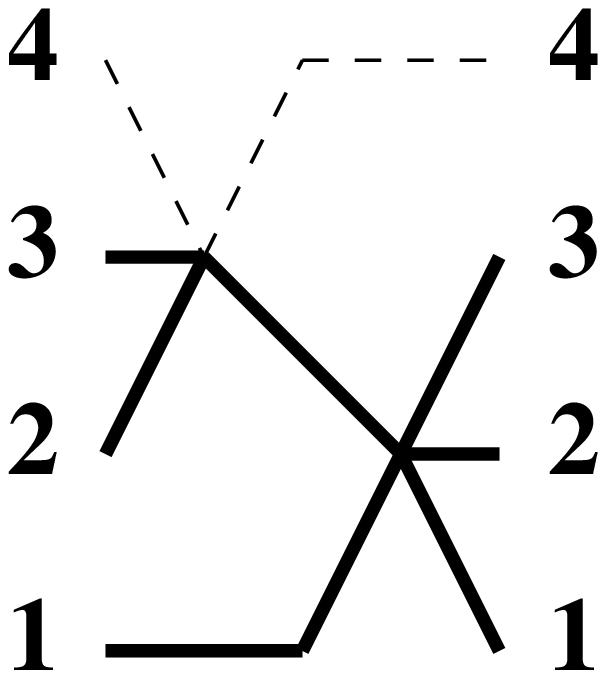}}\ 
\ntksp &\cong \ntksp
\raisebox{-6mm}{\includegraphics[height=12mm]{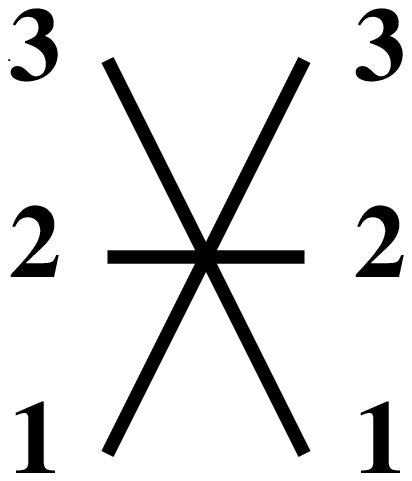}},&\qquad
\raisebox{-8mm}{\includegraphics[height=16mm]{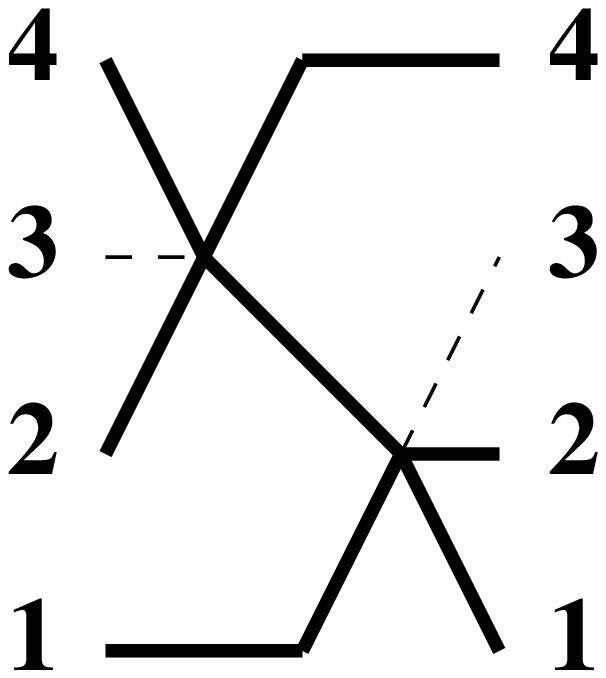}}\ 
\ntksp \cong \ntksp
\raisebox{-6mm}{\includegraphics[height=12mm]{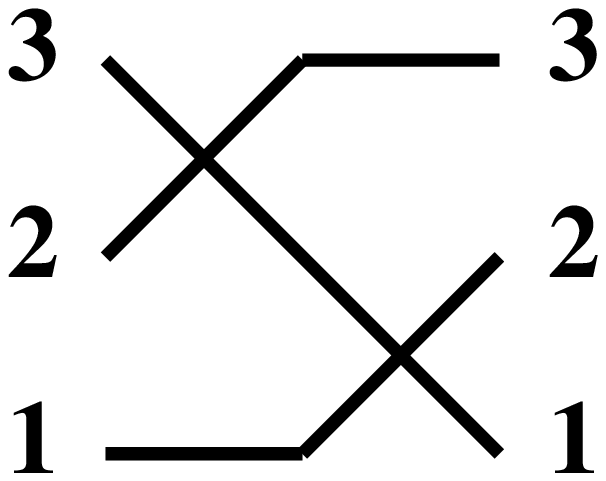}}&,&\qquad
\raisebox{-8mm}{\includegraphics[height=16mm]{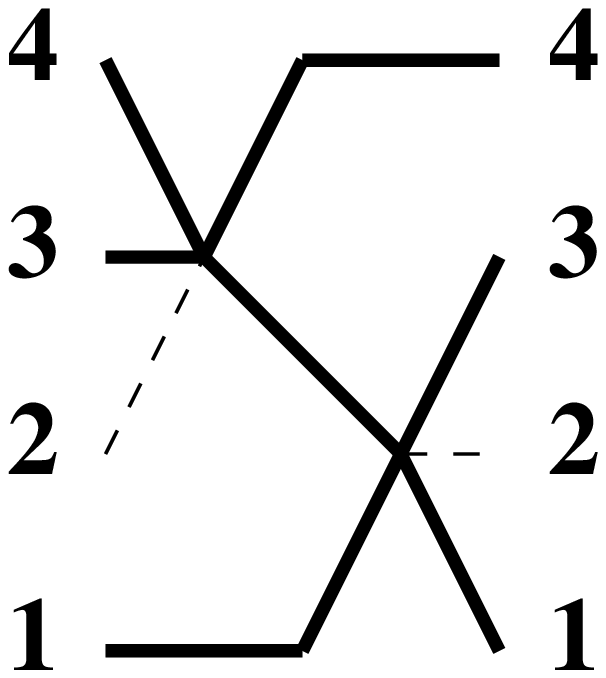}}\ 
\ntksp \cong \ntksp
\raisebox{-6mm}{\includegraphics[height=12mm]{xfigures/xstars1nwo3s}}&,&\qquad
\raisebox{-8mm}{\includegraphics[height=16mm]{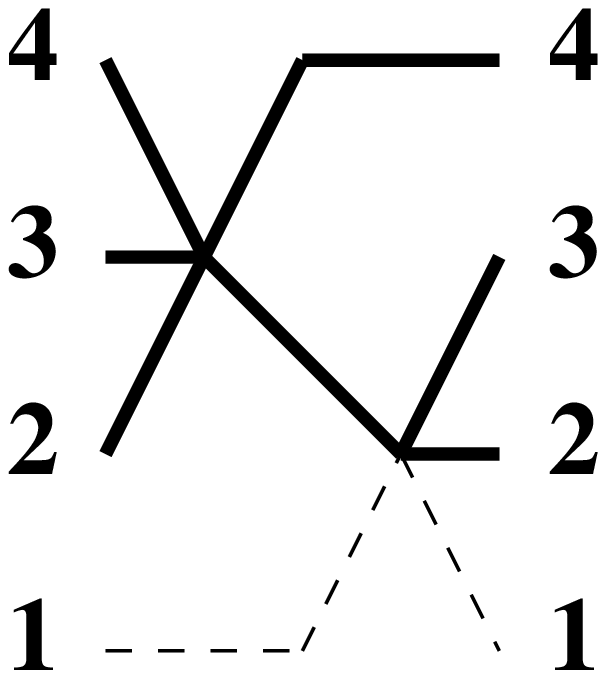}}\ 
\ntksp &\cong \ntksp
\raisebox{-6mm}{\includegraphics[height=12mm]{xfigures/xstars1nwo4s}}\,,\\
[3-1]_q! &= 1 + q, &
\frac{[3-2]_q! [2-1]_q!}{[2-2]_q!} = 1&, &
\frac{[3-2]_q! [2-1]_q!}{[2-2]_q!} = 1&, &
[3-1]_q! &= 1 + q.
\end{alignedat}
\end{equation*}
On the other hand, each subnetwork $F_{3421}|_{i}$ is simply a path and satisfies
$O(F_{3421}|_{i}) = 1$.  
Thus the four set partitions contribute
$q^0(1+q)(1)$, $q^1(1)(1)$,  $q^2(1)(1)$, $q^2(1+q)(1)$, or a total of
$1 + 2q + 2q^2 + q^3$ to
$[3]_q[1]_q ( 1 + 2q + 2q^2 + q^3 ) = \psi_q^{31}(q_{e,3421} C'_{3421}(q))$.

The special case $\lambda = (n)$ of Theorem~\ref{t:qpsiO} confirms a 
conjecture of Haiman~\cite[Conj.\,4.1]{HaimanHecke} concerning 
evaluations of the form $\phi_q^{(n)}(\qew C'_w(q)) = \psi_q^{(n)}(\qew C'_w(q))$.
\begin{prop}
Let $w = w_1 \cdots w_n \in \sn$ avoid the pattern $312$ 
and define the sequence $(f(1), \dotsc, f(n))$ 
by $f(j) = \max \{ w_1, \dotsc, w_j \}$.
Then we have
\begin{equation}\label{eq:qpsiH}
\psi_q^{(n)}(\qew C'_w(q)) = 
[n]_q [f(1) - 1]_q [f(2) - 2]_q \cdots [f(n-1) - (n-1)]_q.
\end{equation}
\end{prop}
\begin{proof}
Setting $\lambda = (n)$ in Theorem~\ref{t:qpsiO}, we have
\begin{equation}\label{eq:qpsiOn}
\psi_q^{(n)}(\qew C'_w(q)) = [n]_q q^0 O(F_w).
\end{equation}
If $F_w$ is not connected, then both sides of (\ref{eq:qpsiOn}) are $0$,
and for some index $k$ the prefix $w_1 \cdots w_k$ of $w$ belongs to $\mfs{k}$.
Thus $f(k) = k$ and the right-hand side of (\ref{eq:qpsiH}) is $0$ as well.

Assume therefore that $F_w$ is connected. 
Since $w$ avoids the pattern $312$, $F_w$ is a descending star network
and the intervals in the corresponding concatenation (\ref{eq:concat2})
form the chain
$[c_1, d_1] \precdot \cdots \precdot [c_t, d_t]$
with $d_1 = n$, $c_t = 1$.
Thus the formula (\ref{eq:OFdef}) for $O(F_w)$ becomes
\begin{multline*}\label{eq:DSNOFw}
[d_1-c_1]_q!
\frac{[d_2-c_2]_q!}{[d_2-c_1]_q!} \cdots \frac{[d_t-c_t]_q!}{[d_t-c_{t-1}]_q!}
= 
\Big([d_1-c_1]_q[d_1-(c_1+1)]_q \cdots [d_1-(n-1)]_q \Big) \cdot \\
\Big([d_2-c_2]_q[d_2-(c_2+1)]_q \cdots [d_2-(c_1-1)]_q \Big) 
\cdots
\Big([d_t-1]_q[d_t-2]_q \cdots [d_t-(c_{t-1}-1)]_q \Big).
\end{multline*}
Defining $g(j) = \min \{ i \,|\, j \in [c_i,d_i] \}$ for $j = 1,\dotsc,n$,
we may now rewrite (\ref{eq:qpsiOn}) as
\begin{equation}\label{eq:almostHaiman}
\psi_q^{(n)}(\qew C'_w(q)) = [n]_q 
[d_{g(1)} - 1]_q [d_{g(2)} - 2]_q \cdots [d_{g(n-1)} - (n-1)]_q.
\end{equation}

Finally we claim that $d_{g(j)} = f(j)$ for $j = 1,\dotsc,n-1$.
We have $f(j) \leq d_{g(j)}$ 
because there are no paths
in $F_w$ from source $j$ to sinks $d_{g(j)}+1, \dotsc, n$
and therefore by Observation~\ref{o:sourcetosink}, 
no paths from sources $1, \dotsc, j-1$ to these sinks either.
Similarly, we have $f(j) \geq d_{g(j)}$ because $j$ belongs to the interval
$[c_{g(j)}, d_{g(j)}]$ and $w_{c_{g(j)}} = d_{g(j)}$.
\end{proof}
It is straightforward to show that the right-hand side of 
(\ref{eq:almostHaiman}) coincides with the expression in 
\cite[Thm.\,7.1]{SWachsChromQ} for $(-1)^{n-1}n = (-1)^{n-\ell(n)} z_{(n)}$ 
times the coefficient of $p_n$ in $X_{P,q}$.

\ssec{Cylindrical $F$-tableaux and $O(F)$}

In Theorem~\ref{t:qpsi},
we will prove an analog of Theorem~\ref{t:swqpsi}
in which sums are taken over (left-anchored) cylindrical $F$-tableaux.
To do so, we partition the set of cylindrical $F$-tableaux into
equivalence classes as follows.
Fix a permutation $w \in \sn$ \avoidingp, 
an integer partition $\lambda \vdash n$, 
and an ordered set partition
$I = (I_1, \dotsc, I_r)$ of $[n]$ of type $\lambda$.
Let $\mathcal C(I,F_w)$ be the set of cylindrical
$F_w$-tableaux $U$ such that for $j = 1,\dotsc,r$,
the set of entries of $L(U_j)$ is equal to $I_j$.
Let $\mathcal C_L(I,F_w)$ be the subset of these tableaux which
are left-anchored.
Now the cylindrical analogs of the sums in (\ref{eq:swqpsi1}) 
and (\ref{eq:swqpsi2}) are
\begin{equation}\label{eq:sumxu}
\begin{gathered}
\sum_U q^{\inv(U_1 \circ \cdots \circ U_r)} = 
\sumsb{I \vdash [n]\\ \type(I) = \lambda} \sum_{U \in \mathcal C(I,F_w)}
 q^{\inv(U_1 \circ \cdots \circ U_r)},\\
\sum_U q^{\inv(U_1 \circ \cdots \circ U_r)} = 
\sumsb{I \vdash [n]\\ \type(I) = \lambda} \sum_{U \in \mathcal C_L(I,F_w)}
 q^{\inv(U_1 \circ \cdots \circ U_r)},
\end{gathered}
\end{equation}
where the left-hand sums are over cylindrical $F_w$-tableaux of shape $\lambda$
and left-anchored cylindrical $F_w$-tableaux of shape $\lambda$,
respectively.
In both cases,
it is easy to show that the 
inner right-hand sum factors as in Theorem~\ref{t:qpsiO}.
To state this factorization explicitly, 
we relate $\inv(U_1 \circ \cdots \circ U_r)$ to intervals
in the concatenation (\ref{eq:concat2}).
\begin{lem}\label{l:indexpairs}
Let $w \in \sn$ \avoidp, and let $[c_1,d_1], \dotsc, [c_t,d_t]$ be
the intervals appearing in the concatenation (\ref{eq:concat2}) 
of star networks that corresponds to $F_w$.  Let $U$ be a cylindrical 
$F_w$-tableau having $r$ rows, and fix indices $p_1 < p_2$ in $[r]$.
Then we have
\begin{multline}\label{eq:indexpairs}
\#\{ (\pi_a,\pi_b) \in U_{p_2} \times U_{p_1} \,|\, 
(\pi_b, \pi_a) \text{ an inversion in } U_1 \circ \cdots \circ U_r \}\\
=
\#\{ (a,b) \in L(U)_{p_2} \times L(U)_{p_1} \,|\, 
c_j \leq a < b \leq d_j \text{ for some } j\}.
\end{multline}
\end{lem}
\begin{proof}
Let $A$ and $B$ denote the sets on the left- and right-hand sides of
(\ref{eq:indexpairs}), respectively.  Define a map $\varphi: A \rightarrow B$
as follows, assuming $(\pi_a, \pi_b) \in A$.
If $a$ and $b$ belong to a common interval $[c_i,d_i]$, then set
$\varphi((\pi_a, \pi_b)) = (a,b)$.  Otherwise, read $U_{p_1}$ cyclically from
left to right starting at $\pi_b$, and let $\pi_c$ be the first path which
lies entirely above $\pi_a$.  Then set $\varphi((\pi_a,\pi_b)) = (a,c)$.
We claim that $\varphi$ is a bijection.

To see that $\varphi$ is well defined, suppose that $a$ and $b$ belong to
no common interval $[c_i, d_i]$,
and let $\pi_f$ be the path in 
$U_{p_1}$ terminating at sink $b$.  If $\pi_f$ intersects $\pi_a$, then
there exists a path in $F_w$ from source $a$ to sink $b$.
Since $a < b$, Observation~\ref{o:sourcetosink} and the comment following it
imply that $a$ and $b$ belong to a common interval $[c_i, d_i]$,
a contradiction.  Thus the set of paths lying strictly above $\pi_a$ in $U_{p_1}$
is nonempty, and the path $\pi_c$ is well defined.
Suppose $a$ and $c$ belong to no common interval $[c_i,d_i]$,
and let $\pi_g$ be the path in $U_{p_1}$ cyclically preceding $\pi_c$.
Then $\pi_g$ terminates at sink $c > a$.
By 
our choice of $\pi_c$, the path $\pi_g$ must intersect $\pi_a$.
Thus there is a path in $F_w$ from source $a$ to sink $c$.  
But this contradicts Observation~\ref{o:sourcetosink}.

The inverse $\xi$ of $\varphi$ may be described as follows, 
assuming $(a,b) \in B$.
If $\pi_a$ intersects $\pi_b$,
then set 
$\xi((a,b)) = (\pi_a, \pi_b)$.  Otherwise, read $U_{p_1}$ cyclically
from right to left starting at $\pi_b$, and let $\pi_c$ be the first path
such that $a$ and $c$ belong to no common interval $[c_i, d_i]$, and $c > a$.
Then set $\xi((a,b)) = (\pi_a, \pi_c)$.

To see that $\xi$ is well defined, suppose that $\pi_a$ does not intersect
$\pi_b$ and that $\xi((a,b)) = (\pi_a, \pi_c)$.  
Let $d$ be the index of the sink of $\pi_c$, so that $\pi_d$ 
immediately follows $\pi_c$.  By our choice of $\pi_c$, 
we have that $a$ and $d$ belong to a common interval $[c_i, d_i]$ 
or that $d < a$.  Thus by Observation~\ref{o:sourcetosink},
there is a path in $F_w$ from source $a$ to sink $d$,
and by Lemma~\ref{l:dsnintersect} paths $\pi_a$ and $\pi_c$ intersect.

Now we claim that $\varphi$ and $\xi$ are in fact inverse to one another.
This is clear when we restrict to pairs $(a,b)$ 
belonging to a common interval in (\ref{eq:concat2})
such that the paths $\pi_a$, $\pi_b$ intersect.
Suppose therefore that $\pi_a$ intersects $\pi_b$, 
and that $a$, $b$ belong to no common interval in (\ref{eq:concat2}). 
Let $\varphi((\pi_a, \pi_b)) = (a,c)$
and let $\xi((a,c)) = (\pi_a, \pi_{b'})$.
Since $\pi_c$ lies entirely above $\pi_a$, we have $b' \neq c$.
Suppose $b' \neq b$ and let $\pi_f$ be the path in $U_{p_1}$ terminating at
sink $b'$.  
By the definitions of $\varphi$ and $\xi$, and since $a$, $b$ belong to
no common interval in (\ref{eq:concat2}), the row
$U_{p_1}$ (up to cyclic rotation) has the form
\begin{equation*}
\cdots\
\raisebox{-3mm}{\includegraphics[height=6mm]{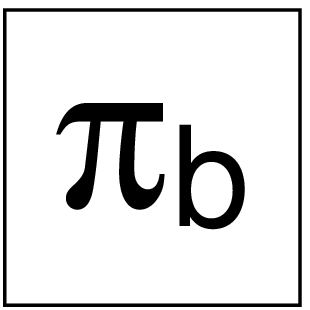}}\ 
\cdots\
\raisebox{-3mm}{\includegraphics[height=6mm]{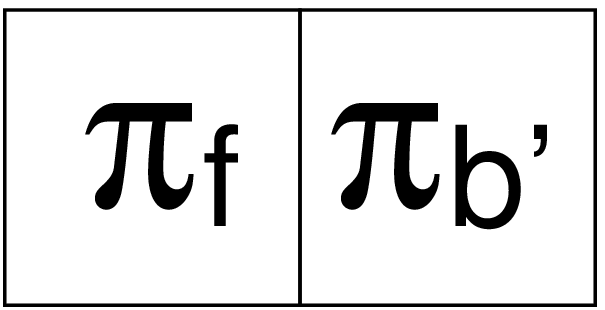}}\ 
\cdots\
\raisebox{-3mm}{\includegraphics[height=6mm]{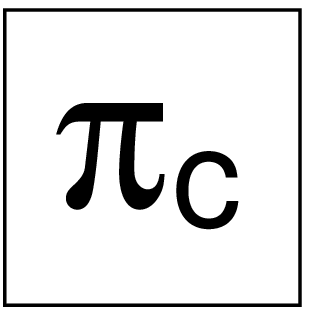}}\ . 
\end{equation*}
Since $a$, $b'$ belong to no common interval and $a < b'$ by the definition
of $\xi$, we have by Observation~\ref{o:sourcetosink} that there is
no path in $F_w$ from source $a$ to sink $b'$.  Thus $\pi_a$ and $\pi_f$
do not intersect, and  
$\pi_f$ must lie entirely above $\pi_a$. 
It follows that $\varphi((\pi_a,\pi_b))=(a,f) \neq (a,c)$, contradiction.
Now suppose that $\pi_a$ does not intersect $\pi_b$, 
and that $a$, $b$ belong to some common interval $[c_i, d_i]$ 
in (\ref{eq:concat2}). 
Let 
$\xi((a,b)) = (\pi_a, \pi_c)$
and let 
$\varphi((\pi_a, \pi_c)) = (a,b')$.
Since $a$ and $c$ belong to no common interval in (\ref{eq:concat2}),
we have $c \neq b'$. 
Suppose $b' \neq b$ and let $g$ be the sink index of $\pi_{b'}$.
By the definitions of $\varphi$ and $\xi$, and since 
$\pi_b$ lies entirely above $\pi_a$,
the row $U_{p_1}$ (up to cyclic rotation) has the form
\begin{equation*}
\cdots\
\raisebox{-3mm}{\includegraphics[height=6mm]{tableaux/1_c}}\ 
\cdots\
\raisebox{-3mm}{\includegraphics[height=6mm]{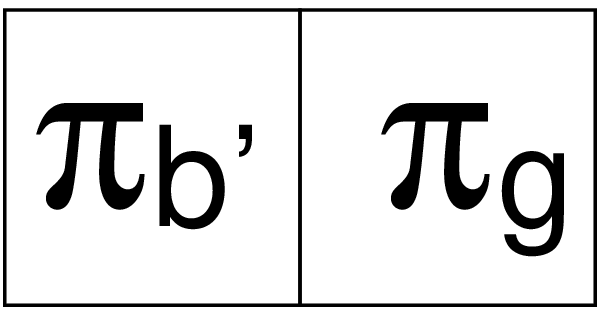}}\ 
\cdots\
\raisebox{-3mm}{\includegraphics[height=6mm]{tableaux/1_b}}\ . 
\end{equation*}
By our choice of $\pi_{b'}$ we have $g > a$, and there is no path in $F_w$
from source $a$ to sink $g$. Thus by Observation~\ref{o:sourcetosink} 
we have that $a$ and $g$ belong to no common interval in (\ref{eq:concat2}).
But this implies that $\xi((a,b)) = (\pi_a, \pi_g) \neq (\pi_a, \pi_c)$, 
a contradiction.
\end{proof}
Now we can factor the expressions in (\ref{eq:sumxu}) as follows.

\begin{prop}\label{p:multisum}
Fix $w \in \sn$ \avoidingp, 
$\lambda  = (\lambda_1,\dotsc,\lambda_r) \vdash n$,
and a set partition $I = (I_1,\dotsc,I_r) \vdash [n]$ of type $\lambda$.
Let $V = V(F_w,I)$.
Then we have
\begin{equation*}
\sum_U 
q^{\inv(U_1 \circ \cdots \circ U_r)} = 
q^{\inv(V_1 \circ \cdots \circ V_r))}
\Big( \sum_{W^{(1)}} q^{\inv(W^{(1)})} \Big) \cdots
\Big( \sum_{W^{(r)}} q^{\inv(W^{(r)})} \Big),
\end{equation*}
where the sums are over
$U \in \mathcal C(I,F_w)$, 
$W^{(1)} \in \mathcal C(I_1, F_w|_{I_1}), \dotsc,
W^{(r)} \in \mathcal C(I_r, F_w|_{I_r})$,
or over
$U \in \mathcal C_L(I,F_w)$, 
$W^{(1)} \in \mathcal C_L(I_1, F_w|_{I_1}), \dotsc,
W^{(r)} \in \mathcal C_L(I_r, F_w|_{I_r})$.
\end{prop}
\begin{proof}
First observe that we have a bijection
$\mathcal C(I_1, F_w|_{I_1}) \times \cdots \times \mathcal C(I_r, F_w|_{I_r})
\rightarrow \mathcal C(I,F_w)$ defined by joining $W^{(1)}, \dotsc, W^{(r)}$
into a single tableau $U$ with $U_j = W^{(j)}$.
Clearly, the bijection restricts to the 
corresponding subsets of left-anchored tableaux, and satisfies
\begin{equation}\label{eq:multisum}
\inv(U_1 \circ \cdots \circ U_r) = 
\inv(W^{(1)}) + \cdots + \inv(W^{(r)})
+ \inv(U^\tr).
\end{equation}
Now observe that the number $\inv(U^\tr)$
is equal to the left-hand side of (\ref{eq:indexpairs}), 
summed over pairs $(p_1, p_2)$ with $p_1 < p_2$.
Furthermore, by Observation~\ref{o:sourcetosink}, Lemma~\ref{l:dsnintersect}, 
and the definition of $V(F_w,I)$,
we have that $\inv(V_1 \circ \cdots \circ V_r)$ 
is equal to the right-hand side of (\ref{eq:indexpairs}),
summed over pairs $(p_1, p_2)$ with $p_1 < p_2$.
Thus we may rewrite (\ref{eq:multisum}) as
\begin{equation*}
\inv(U_1 \circ \cdots \circ U_r) = 
\inv(W^{(1)}) + \cdots + \inv(W^{(r)})
+ \inv(V_1 \circ \cdots \circ V_r),
\end{equation*}
as desired.
\end{proof}



It is clear that a cylindrical tableau $U$ is completely determined
by the Young tableau $L(U)$.  Under some conditions the insertion
or deletion of a greatest letter in the left tableau of a cylindrical
tableau yields a valid left tableau of another cylindrical tableau.
In these cases, intersecting paths in the two cylindrical tableaux
are closely related.


\begin{lem}\label{l:ndelins}
Let $F_w$ be a connected 
zig-zag
network of order $n$
corresponding to the concatenation (\ref{eq:concat2}) of star networks
(\ref{eq:concat2}),
and let $[b,n]$ be the unique interval containing $n$.
\begin{enumerate}
\item For any cylindrical $F_w$-tableau $U$ of shape $(n)$,
if $T'$ is the Young tableau obtained from $L(U)$ by deleting
the entry $n$, then there exists a unique 
cylindrical $F_{\iota(w)}$-tableau $U'$
such that $T' = L(U')$.
\item For any cylindrical $F_{\iota(w)}$-tableau $V'$ of shape $(n-1)$,
let $T$ be the Young tableau obtained from $L(V')$ 
by inserting the entry $n$ cyclically before any element in $[b, n-1]$ 
if the interval $[b,n]$ is maximal in $\preceq$, 
and cyclically after any element in $[b,n-1]$ otherwise.
Then there exists a unique
cylindrical $F_w$-tableau $V$ such that $L(V) = T$.
\item Let tableaux $U$ and $U'$ in (1) contain the path families
$(\pi_1, \dotsc, \pi_n)$ and $(\pi'_1, \dotsc, \pi'_{n-1})$, respectively.
For all pairs $(i,j)$, if $\pi_i$, $\pi'_i$ have the same sink
index, and $\pi_j$, $\pi'_j$ have the same sink index,
then $\pi_i$ and $\pi_j$ intersect 
if and only if $\pi'_i$ and $\pi'_j$ intersect.
\end{enumerate}
\end{lem}
\begin{proof}
(1) Let $(i_1, i_2)$ be a pair of cyclically consecutive entries in $T'$.
If these entries are also cyclically consecutive in $L(U)$, then there exists
a (unique) 
path $\pi_{i_1}$ in $F_w$ from source $i_1$ to sink $i_2$.  
Since $F_{\iota(w)}$ differs from $F_w$ by the removal of the 
unique path from source $n$ to sink $n$
and possibly the contraction of an edge to a single vertex, 
the image of $\pi_{i_1}$ is again the unique path in $F_{\iota(w)}$ from
source $i_1$ to sink $i_2$.
If $(i_1, i_2)$ are not cyclically consecutive in $L(U)$, then $i_1$
cyclically precedes $n$ 
and $i_2$ cyclically follows $n$ 
in $L(U)$.
Thus either $i_1$ or $i_2$ belongs to $[b,n-1]$.
Thus in $G$ there is a path from source $i_1$ to the central vertex
of $G_{[b,n]}$ and a path from this vertex to sink $i_2$.
It follows that there is a path in $F_w$ from source $i_1$ to sink $i_2$.
Uniqueness of $U'$ follows from uniqueness of source-to-sink paths
in zig-zag networks.  (See comment following Theorem~\ref{t:lem53nndcb}.)

(2)
Let $(i_1, i_2)$ be a pair of cyclically consecutive entries in $T$.
If these entries are also cyclically consecutive in $L(V')$,
then there exists a (unique) 
path from source $i_1$ to sink $i_2$ 
in $F_{\iota(w)}$.  Since each interval in the concatenation 
corresponding to $F_{\iota(w)}$ is equal to or is contained in an interval
in the concatenation corresponding to $F_w$, we use
Observation~\ref{o:sourcetosink} and the fact that both $F_w$ and $F_{\iota(w)}$
are connected to infer that there is a path in $F_w$ from source $i_1$ to 
sink $i_2$. 
Again, uniqueness of $V$ follows from uniqueness of 
source-to-sink paths in zig-zag networks.  

(3)
Let $\pi_i$, $\pi'_i$, $\pi_j$, $\pi'_j$ satisfy the stated conditions.
Then the source and sink indices of these paths are not equal to $n$.
Suppose first that in the concatenation (\ref{eq:concat2})
corresponding to $F_w$ we have $t=1$, or $d_1 = n$ and $d_2 < n-1$, 
or $d_t = n$ and $d_{t-1} < n-1$.
Then $F_{\iota(w)}$ is the subgraph of $F_w$ obtained by 
deleting the unique path from source $n$ to sink $n$.
By the uniqueness of paths in descending star networks, we have
$\pi_i = \pi'_i$ and $\pi_j = \pi'_j$.
Now suppose that in (\ref{eq:concat2}) we have 
$d_1 = n$ and $d_2 = n-1$, 
or $d_t = n$ and $d_{t-1} = n-1$, 
and let $x$, $y$ be the vertices in $F_w$ corresponding to
the central vertices of the star networks 
$G_{[b,n]}$ and $G_{[c_2,d_2]}$, respectively,
or $G_{[c_{t-1},d_{t-1}]}$ and $G_{[b,n]}$, respectively.
Then $F_{\iota(w)}$ is obtained from $F_w$ by
deleting the unique path from source $n$ to sink $n$,
and by contracting the edge $(x,y)$ to a single vertex $z$.
Thus if $\pi_i \cap \pi_j$ contains the edge $(x,y)$
then $\pi'_i \cap \pi'_j$ contains the vertex $z$;
if $\pi_i \cap \pi_j$ does not contain the edge $(x,y)$,
then $\pi'_i \cap \pi'_j = \pi_i \cap \pi_j$. 
\end{proof}



Given a zig-zag network $F$ of order $n$, define $Y_n(F,i)$ 
to be the set of 
cylindrical $F$-tableaux $U$ of shape $(n)$ 
in which the first entry of $L(U)$ is $i$.
In terms of our earlier notation, we have
\begin{equation}\label{eq:CY}
\mathcal C([n], F_w) = \bigcup_{1 \leq i \leq n} Y_n(F_w, i),
\qquad
\mathcal C_L([n], F_w) = Y_n(F_w, 1),
\end{equation}
where we interpret $[n]$ as the ordered set partition having one block.

Let us examine the map 
$U \mapsto U'$ and the numbers $\inv(U)$, $\inv(U')$
defined by Lemma~\ref{l:ndelins} (1) in the case that
$U \in Y(F_w, i)$ and $i \leq n-1$.
To be precise,
for each pair $(F_w, i)$ where $F_w$ 
is a connected zig-zag network of order at least 2
with corresponding concatenation (\ref{eq:concat2}),
poset $\preceq$ of intervals,
and $[b,n]$ the unique interval containing $n$,
and $1 \leq i \leq n-1$, we define a map 
\begin{equation*}
\begin{aligned}
\delta_1: Y_n(F_w, i) &\rightarrow
Y_{n-1}(F_{\iota(w)}, i) \times \{ 0, 1, \dotsc, n - b - 1 \}\\
U &\mapsto (U',k)
\end{aligned}
\end{equation*}
by declaring $U'$ to be the cylindrical tableau 
whose left tableau is obtained from $L(U)$
by deleting $n$,
and by declaring 
$k$ to be the number of paths following $\pi_n$ in $U$ whose
\begin{equation*}
\begin{cases}
\text{sink index belongs to $[b,n-1]$} 
&\text{if $[b,n]$ maximal in $\preceq$},\\
\text{source index belongs to $[b,n-1]$} 
&\text{otherwise}.
\end{cases}
\end{equation*}

For example, let $n = 6$, recall 
(\ref{eq:F256431}), and consider the $F_{256431}$-tableau and its left 
Young tableau
\begin{equation*}
U = 
\raisebox{-1.6mm}{\includegraphics[height=6mm]{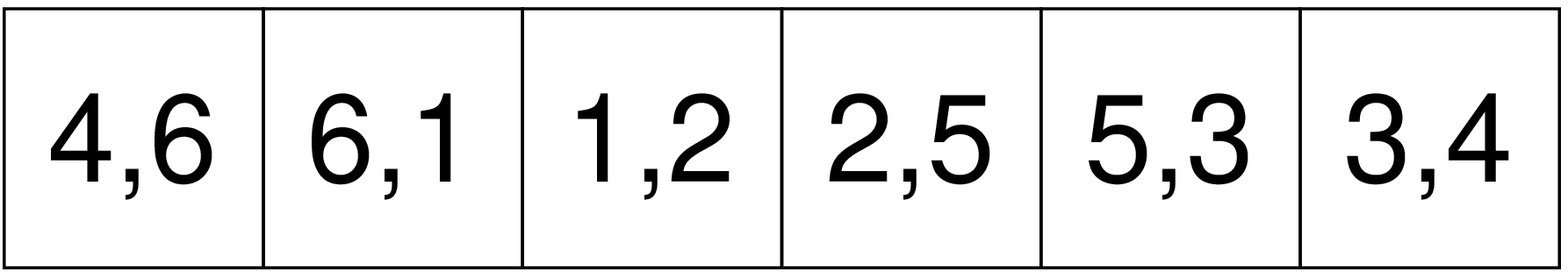}}\
, \quad
T =
\raisebox{-1.6mm}{\includegraphics[height=6mm]{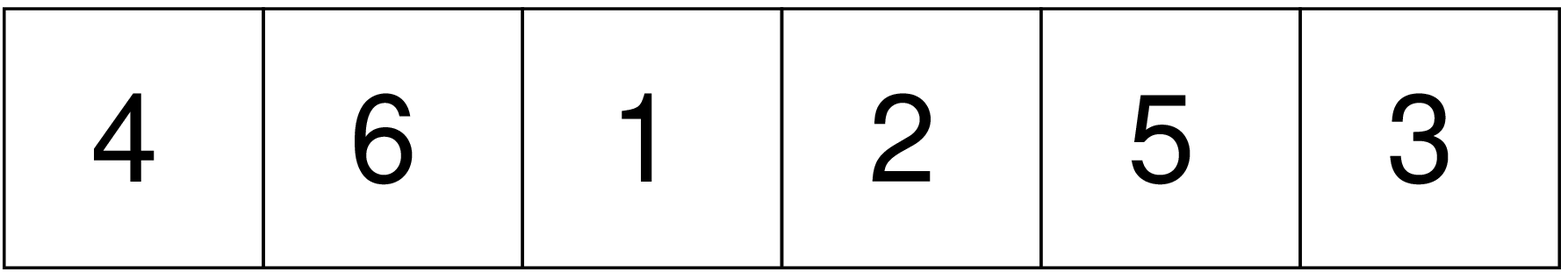}}\ ,
\end{equation*}
with $U \in Y_6(256431,4)$, and
$[3,6]$ not maximal in the poset 
$[3,6] \prec [2,4] \prec [1,2]$.
Removing $6$ from $T$ we have the tableaux
\begin{equation*}
T' =
\raisebox{-1.6mm}{\includegraphics[height=6mm]{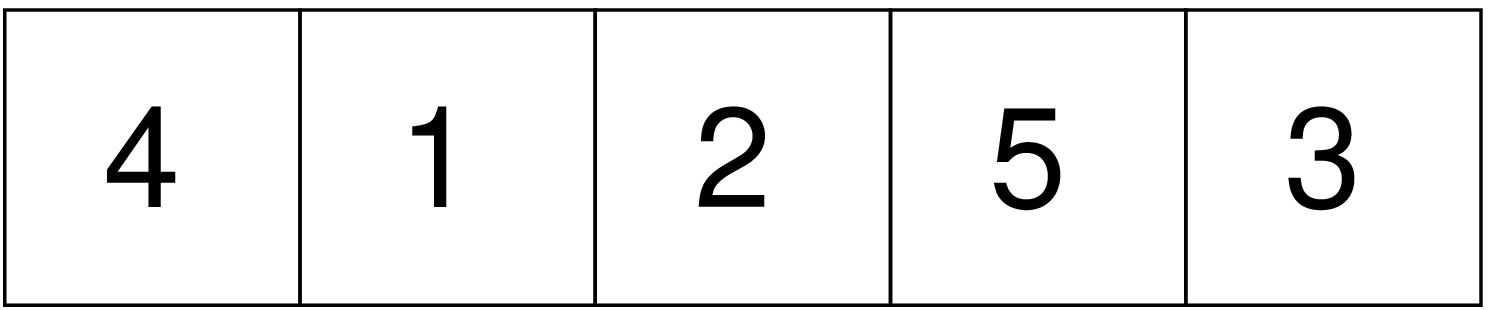}}\ , \quad
U' = 
\raisebox{-1.6mm}{\includegraphics[height=6mm]{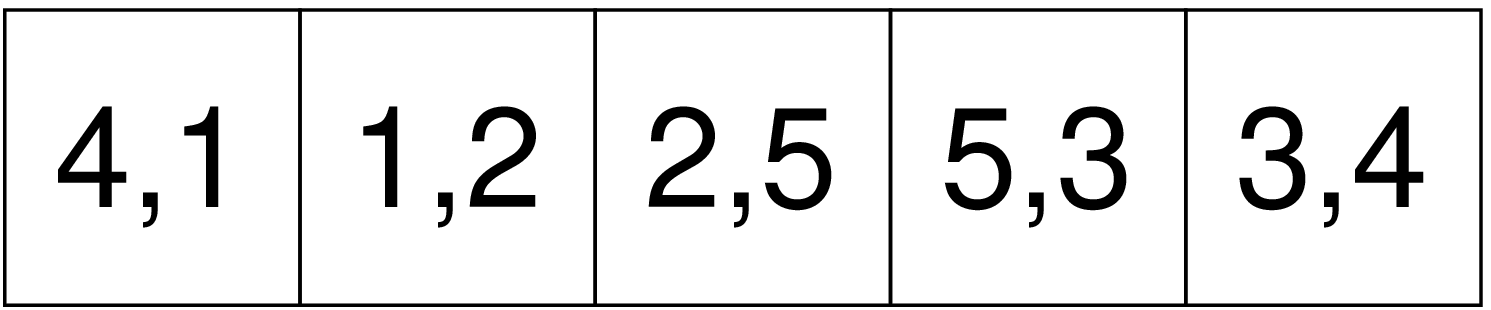}}\ ,
\end{equation*}
with
$T' = L(U')$ and
$U' \in Y_5(F_{25431},4)$.
Since 
the only paths in $U$ which follow $\pi_6$ and have source indices
in $[3,5]$ are $\pi_5$, $\pi_3$, 
we have $\delta_1(U) = (U', 2)$.

\begin{lem}\label{l:delta1}
For each connected zig-zag 
network $F_w$ of order $n$, 
the map $\delta_1$ 
is a bijection.  
Furthermore, if $\delta_1(U) = (U',k)$ then we have $\inv(U) = \inv(U') + k$.
\end{lem}
\begin{proof}
Assume that $F_w$ 
corresponds to the concatenation (\ref{eq:concat2})
in which the unique interval containing $n$ is $[b,n]$,  
and let $(U',k)$ be a pair in 
$Y_{n-1}(F_{\iota(w)}, i) \times \{ 0, 1, \dotsc, n - b - 1 \}$.
To invert $\delta_1$, find an entry $j \in [b,n-1]$ in the tableau $L(U')$
with exactly $k-1$ entries in $[b,n-1]$ to its right.  
(This is possible, since $k \leq n-1-b$.)
Now create a new Young tableau $T$ by inserting the letter $n$ into $L(U')$ 
immediately to the left of $j$ if $[b,n]$ is maximal in $\preceq$, 
or immediately to the right of $j$ otherwise.
By Lemma~\ref{l:ndelins} (2)
there is a unique $F_w$-tableau $U$ with $L(U) = T$.


To compare inversions in $U$ 
and $U'$, write
\begin{equation*}
\begin{aligned}
U &= (\pi_{i_1}, \dotsc, \pi_{i_n}),\\
U' &= (\pi'_{i_1}, \dotsc, \pi'_{i_{\ell-1}}, \pi'_{i_{\ell+1}}, \dotsc, \pi'_{i_n}),
\end{aligned}
\end{equation*}
where $\pi_{i_\ell} = \pi_n$.
Observe that all paths except $\pi_{i_{\ell-1}}$, $\pi_{i_\ell}$
in $U$ have the same sources and sinks as the corresponding paths in $U'$.
Thus by Lemma~\ref{l:ndelins} (3), two such paths form an inversion
in $U$ if and only if the corresponding paths form an inversion in $U'$.
Consider therefore inversions in $U$ which involve one of the paths
$\pi_{i_{\ell-1}}$, $\pi_{i_\ell}$,
and inversions in $U'$ which involve the path $\pi'_{i_{\ell-1}}$.

Suppose first that $[b,n]$ is maximal in $\preceq$.
By definition, there are $k$ inversions in $U$ of the form 
$(\pi_{i_\ell}, \pi_c)$, where $\pi_c$ terminates at a sink 
having index in $[b,n-1]$.  These in fact are the only inversions in $U$
involving $\pi_{i_\ell}$: since there are no paths in $F_w$ from source $n$
to sinks $1, \dotsc, b-1$, Observation~\ref{o:pnetintersect}
implies that $\pi_{i_\ell}$ cannot intersect any path terminating at one
of these sinks.
Now observe that path $\pi_{i_{\ell-1}}$ terminates at sink $n$,
while paths $\pi'_{i_{\ell-1}}$ and $\pi_n$ terminate at sink $i_{\ell+1}$.
By the maximality of $[b,n]$, we have that $i_{\ell+1}$ 
belongs to the interval $[b,n]$. 
Thus the paths $\pi_{i_{\ell-1}}$ and $\pi'_{i_{\ell-1}}$ are identical
from their sources 
up to the vertex of $F_w$ ($F_{\iota(w)}$) 
corresponding to the central vertex of $G_{[b,n]}$.
It follows that any path in $U$ intersects $\pi_{i_{\ell-1}}$ if and only if the
corresponding path in $U$ intersects $\pi'_{i_{\ell-1}}$.
Therefore we have $\inv(U) = \inv(U') + k$.

Now suppose that $[b,n]$ is not maximal in $\preceq$. 
Then it must be minimal, 
and since $\pi_{i_{\ell-1}}$ terminates at sink $n$,
we have $i_{\ell-1} \geq b$.
Consider paths $\pi_c$ with $c \geq b$.
By the minimality of $[b,n]$ in $\preceq$,
we have that $\pi_c$ intersects $\pi_{i_{\ell-1}}$ and $\pi_{i_\ell}$
at the vertex of $F_w$ ($F_{\iota(w)}$) corresponding to the
central vertex of $G_{[b,n]}$.
Thus
$(\pi_c, \pi_{i_{\ell-1}})$ 
or
$(\pi_{i_{\ell-1}}, \pi_c)$ 
is an inversion in $U$ if and only if the corresponding pair
is an inversion in $U'$.
By definition, there are $k$ inversions in $U$ 
of the form $(\pi_{i_\ell}, \pi_c)$ with $c \geq b$.
Now consider paths $\pi_c$ with $c < b$.
By Observation~\ref{o:pnetintersect},
no pair $(\pi_{i_{\ell-1}}, \pi_c)$ 
is an inversion in $U$,
since there is no path from source $c$ to sink $n$ in $F_w$
On the other hand, the paths $\pi_{i_\ell}$ and $\pi'_{i_{\ell-1}}$
are identical from the vertex of $F_w$ (or $F_{\iota(w)}$) corresonding
to the central vertex of $G_{[b,n]}$ to sink $i_{\ell+1}$.
Thus each pair $(\pi_{i_\ell}, \pi_c)$ is an inversion in $U$
if and only if $(\pi'_{i_{\ell-1}}, \pi'_c)$ is an inversion in $U'$.
It follows again that $\inv(U) = \inv(U') + k$.
\end{proof}

Note that in the example preceding Lemma~\ref{l:delta1}
we have $\inv(U) = 6 = \inv(U') + 2$, and
$\delta_1(U) = (U',2)$.

Now let us examine a map 
$U \mapsto U'$ and the numbers $\inv(U)$, $\inv(U')$
closely related to those defined by Lemma~\ref{l:ndelins} (1) 
in the case that
$U \in Y(F_w, n)$.
To be precise,
for each connected zig-zag
network $F_w$ of order at least $2$
with corresponding concatenation (\ref{eq:concat2}), poset $\preceq$
of intervals, 
and $[b,n]$ the unique interval containing $n$,
we define a map 
\begin{equation}\label{eq:delta2}
\begin{aligned}
\delta_2: Y_n(F_w, n) &\rightarrow
\bigcup_{j=b}^{n-1} Y_{n-1}(F_{\iota(w)}, j),\\
U &\mapsto U',
\end{aligned}
\end{equation}
by declaring $U'$ to be the cylindrical tableau 
whose left tableau 
is obtained from $L(U) = (i_1, \dotsc, i_n)$
by 
\begin{equation*}
L(U') = 
\begin{cases}
(i_2, \dotsc, i_n) &\text{if $[b,n]$ maximal in $\preceq$},\\
(i_n, i_2, \dotsc, i_{n-1}) &\text{otherwise}.
\end{cases}
\end{equation*}
The tableau $U'$ exists and is unique by Lemma~\ref{l:ndelins} (1).
If $[b,n]$ is maximal in $\preceq$ then any path beginning at source $n$
must terminate at a sink in this interval.
Thus we have $i_2 \in [b,n-1]$.
If $[b,n]$ is not maximal in $\preceq$ it must be minimal, 
and any path terminating at sink $n$
must begin at a source in this interval.
Thus we have $i_n \in [b,n-1]$.
It follows that $U'$ belongs to the union in (\ref{eq:delta2}).

For example, let $n = 6$, recall 
(\ref{eq:F246531}), and consider the $F_{246531}$-tableau and its left 
Young tableau
\begin{equation*}
U = 
\raisebox{-1.6mm}{\includegraphics[height=6mm]{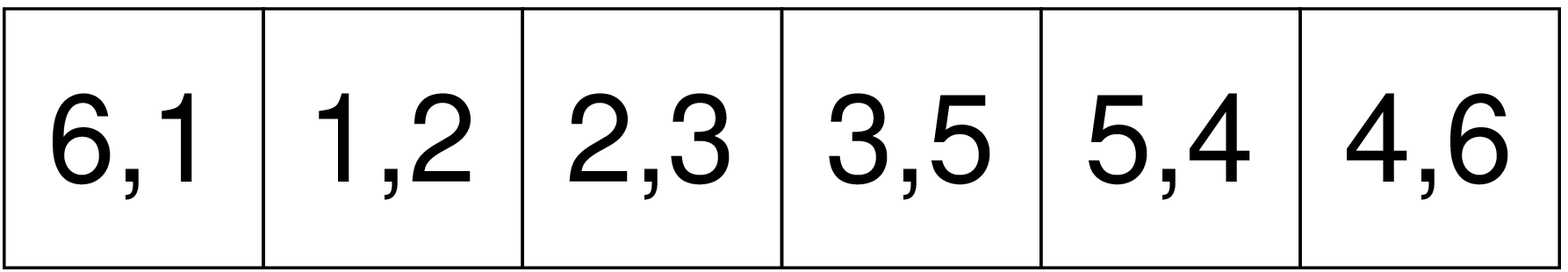}}\
, \quad
T =
\raisebox{-1.6mm}{\includegraphics[height=6mm]{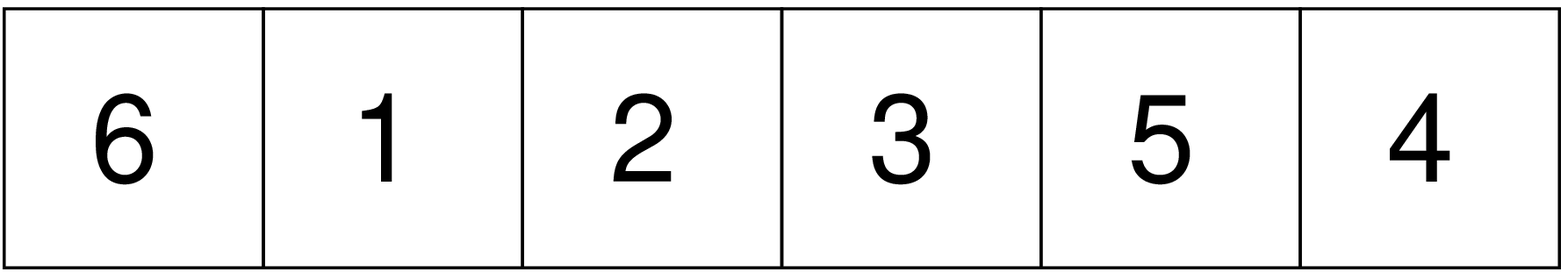}}\ ,
\end{equation*}
with $U \in Y_6(246531,6)$, and $T$ having rightmost entry $4$.
Removing $6$ from $T$ and moving $4$ to the leftmost position,
we have the tableaux
\begin{equation*}
T' =
\raisebox{-1.6mm}{\includegraphics[height=6mm]{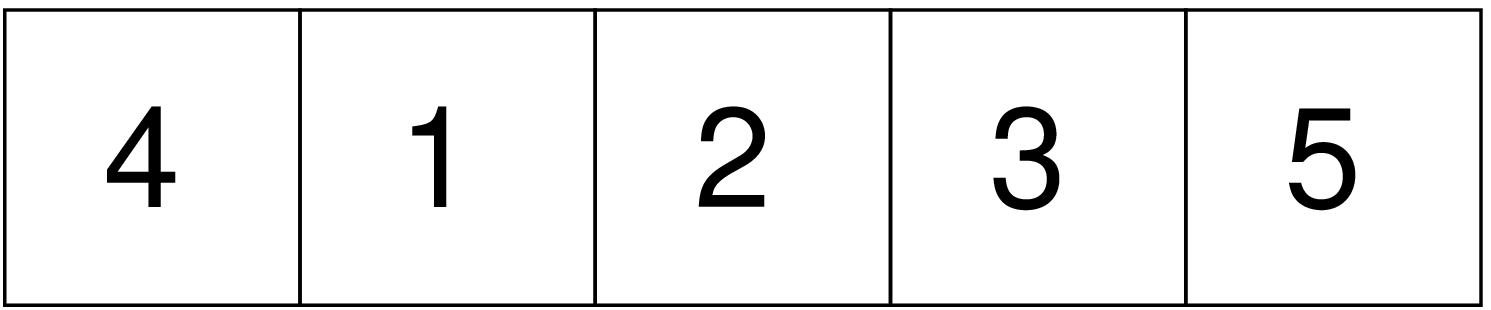}}\ , \quad
U' = 
\raisebox{-1.6mm}{\includegraphics[height=6mm]{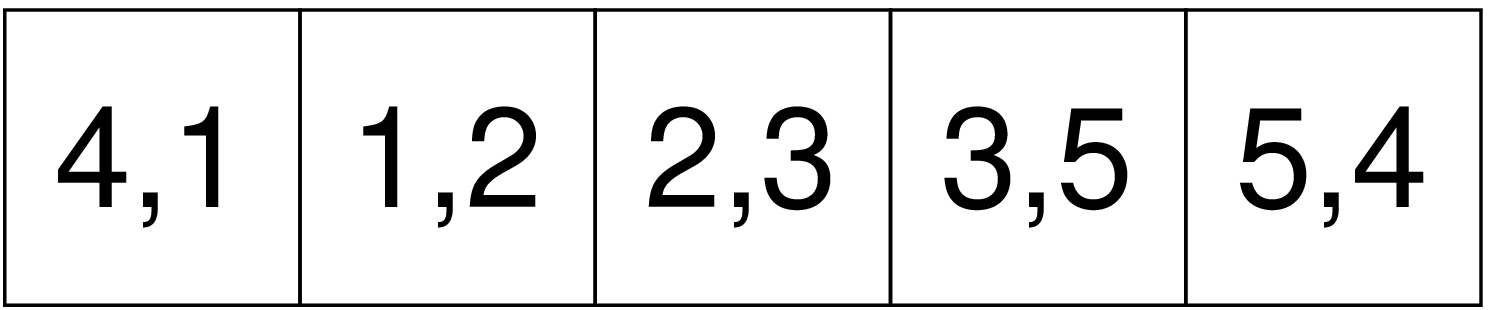}}\ ,
\end{equation*}
with
$T' = L(U')$ and
$U' \in Y_5(F_{24531},4)$.

\begin{lem}\label{l:delta2}
For each connected zig-zag
network $F_w$ of order $n$, 
the map $\delta_2$ 
is a bijection.  Furthermore, if $\delta_2(U) = U' \in Y_{n-1}(F_{\iota(w)},j)$ 
and $b$ is as in (\ref{eq:delta2})
then we have 
\begin{equation*}
\inv(U) = 
\begin{cases}
\inv(U') + n - b &\text{if $[b,n]$ maximal in $\preceq$},\\
\inv(U') + 2n - 2j - 1 &\text{otherwise}.
\end{cases}
\end{equation*}
\end{lem}
\begin{proof}
To invert $\delta_2$, let $U'$ be an element of 
$Y(F_{\iota(w)},j)$ for some $1 \leq j \leq n-1$.
Create a Young tableau $T$
from $L(U')$ by inserting $n$ into the leftmost position
and if $[b,n]$ is not maximal in $\preceq$ 
by moving $j$ to the rightmost position.
By Lemma~\ref{l:ndelins} (2) there is a unique $F_w$-tableau 
$U = (\pi_{i_1}, \dotsc, \pi_{i_n})$ satisfying $L(U) = T$.

To compare inversions in $U$ and $U'$, 
observe first that the paths
$\pi_{i_2}, \dotsc, \pi_{i_{n-1}}$ in $U$ have the same 
sources and sinks as the corresponding paths in 
\begin{equation*}
U' = 
\begin{cases}
(\pi'_{i_2}, \dotsc, \pi'_{i_n}) &\text{if $[b,n]$ is maximal in $\preceq$},\\
(\pi'_{i_n} = \pi'_j, \pi'_{i_2}, \dotsc, \pi'_{i_{n-1}}) &\text{otherwise}.
\end{cases}
\end{equation*}
Thus by Lemma~\ref{l:ndelins} (3), two such paths form an inversion
in $U$ if and only if the corresponding paths form an inversion in $U'$.
Consider therefore inversions in $U$ which involve the paths $\pi_{i_1} = \pi_n$
or $\pi_{i_n}$, and inversions in $U'$ which involve the path $\pi'_{i_n}$.


If $[b,n]$ is maximal in $\preceq$, the path $\pi_{i_1} = \pi_n$
in $U$ terminates at sink $i_2 \in [b,n-1]$.
Thus the path, which precedes all others in $U$, intersects only those
$n - b$ other paths in $U$ which terminate at sinks $[b,n] \ssm \{ i_2 \}$,
and thus also pass through the vertex of $F_w$
corresponding to the central vertex of $G_{[b,n]}$.
Now observe that $\pi_{i_n}$ terminates at sink $n$ of $F_w$, while
$\pi'_{i_n}$ terminates at sink $i_2$ of $F_{\iota(w)}$.  Since 
$i_2 \geq b$,
the paths $\pi_{i_n}$ and $\pi'_{i_n}$ are identical
up to the vertex of $F_w$ ($F_{\iota(w)}$) which corresponds to the central
vertex of $G_{[b,n]}$.  Thus any path $\pi_{i_k}$ in $U$ intersects $\pi_{i_n}$
if and only if the corresponding path $\pi'_{i_k}$ in $U'$ intersects
$\pi'_{i_n}$.  It follows that $\inv(U) = \inv(U') + b - n$ in this case.

If $[b,n]$ is not maximal in $\preceq$, then it is minimal.
Since $\pi_{i_n}$ terminates at sink $n$,
we have that $i_n \geq b$.  Thus $\pi_{i_n}$ intersects
and follows the $n - i_n + 1$ paths $\pi_{i_n + 1}, \dotsc, \pi_n$ which all 
intersect at least at the vertex of $F_w$ (or $F_{\iota(w)}$) corresponding 
to the central vertex of $G_{[b,n]}$. Now observe that the paths $\pi_n$
and $\pi'_{i_n}$ both terminate at sink $i_2$.  Thus, since 
$i_n \geq b$,
the paths are identical from the vertex of $F_w$ ($F_{\iota(w)}$) which 
corresponds to the central vertex of $G_{[b,n]}$ until sink $i_2$.  Thus
any path $\pi_k$ in $U$ intersects $\pi_n$ if and only if the corresponding
path $\pi'_k$ in $U'$ intersects $\pi'_{i_n}$.
It follows that $\inv(U) = \inv(U') + n - j + 1$ in this case.
\end{proof}


Note that in the example preceding Lemma~\ref{l:delta2},
the network $F_{246531}$ (\ref{eq:F246531}) of order $m = 6$ 
corresponds to the concatenation $G_{[3,6]} \circ G_{[2,4]} \circ G_{[1,2]}$,
and the interval $[3,6]$ is not maximal in the poset 
$[3,6] \prec [2,4] \prec [1,2]$.
The tableaux $U$, $U'$ satisfy
$j = 4$, $\inv(U) = 6$, $\inv(U') = 3$, and
$\inv(U) = \inv(U') + 2m - 2j - 1$.

Now we return to the problem of factoring the inner sums in (\ref{eq:sumxu}).

\begin{prop}
Fix $w \in \mfs m$ \avoidingp{} 
and an index $j \in [m]$.
Then we have
\begin{equation}\label{eq:qjm1}
\sum_{U \in Y_m(F_w,j)}q^{\inv(U)} = q^{j-1}O(F_w).
\end{equation}
\end{prop}
\begin{proof}
We prove (\ref{eq:qjm1}) by induction on $m$.
The only zig-zag 
network of order $m = 1$ is $F_e$, $e \in \mfs 1$.
Thus the set $Y_1(F_e, 1)$ 
consists of one tableau of shape $(1)$ having no inversions.
The left- and right-hand sides of (\ref{eq:qjm1}) are therefore 
$q^0 = 1$
and
$q^0O(F_e) = [0]_q! = 1$, respectively.

Now assume (\ref{eq:qjm1}) to hold for zig-zag
networks corresponding to \pavoiding
permutations in $\mfs {m-1}$
and consider $F_w$ with $w \in \mfs m$ \avoidingp.
Let $F_w$ correspond to the concatenation (\ref{eq:concat2}) 
of star networks with $[b,m]$ the unique interval containing $m$,
let $u = \iota(w)$, and fix an integer $j \in [m]$.
If $F_w$ is disconnected, then
the set $Y_m(F_w,j)$ is empty and both sides of (\ref{eq:qjm1}) are $0$.
Assume therefore that $F_w$ is connected.
If $j < m$, then 
by Lemma {\ref{l:delta1}} 
and induction 
we have
\begin{equation}\label{eq:mminusc11}
\sum_{U \in Y_m(F_w,j)}q^{\inv(U)} 
= \sum_{k=0}^{m-b-1} q^k \sum_{U' \in Y_{m-1}(F_u,j)}q^{\inv(U')}\\
= [m-b]_q q^{j-1}O(F_u).
\end{equation}
Similarly, if $j = m$, then 
by Lemma {\ref{l:delta2}} 
and induction 
we have
\begin{equation}\label{eq:mminusc12}
\begin{aligned}
\sum_{U \in Y_m(F_w,m)}q^{\inv(U)} 
&= \begin{cases}
\displaystyle{\sum_{k=b}^{m-1}} q^{m-b} 
q^{k-1} O(F_u) 
&\text{if $[b,m]$ maximal in $\preceq$},\\
\displaystyle{\sum_{k=b}^{m-1}} q^{2m-2k-1} 
q^{k-1} O(F_u) 
&\text{otherwise},
\end{cases}\\
&= [m-b]_q q^{m-1} O(F_u).
\end{aligned}
\end{equation}
Now 
applying 
(\ref{eq:OandO})  
to (\ref{eq:mminusc11}), (\ref{eq:mminusc12}),
we have for $j = 1, \dotsc, m$ that 
\begin{equation*}
\sum_{U \in Y_m(F_w,j)}q^{\inv(U)} = 
q^{j-1}[m-b]_q O(F_u) = q^{j-1} O(F_w).
\end{equation*}
\end{proof}


Applying (\ref{eq:CY}) 
to the previous result, we have the following.
\begin{cor}\label{c:SOF}
Let $w \in \mfs m$ \avoidp.
Then we have
\begin{equation*}
\sum_{U \in \mathcal C([m],F_w)} q^{\inv(U)} = [m]_q O(F_w),
\qquad
\sum_{U \in \mathcal C_L([m],F_w)} q^{\inv(U)} = O(F_w).
\end{equation*}
\end{cor}



Now we have the following $q$-analogs of 
Theorem~\ref{t:sncfinterp} (iv-a).
\begin{thm}\label{t:qpsi}
Let $w \in \sn$ \avoidp.
For $\lambda = (\lambda_1, \dotsc, \lambda_r) \vdash n$, we have
\begin{equation}\label{eq:qpsi1}
\psi_q^\lambda(\qew C'_w(q)) 
= \sum_U q^{\inv(U_1 \circ \cdots \circ U_r)},
\end{equation}
where the sum is over all cylindrical $F_w$-tableaux of shape $\lambda$, and
\begin{equation}\label{eq:qpsi2}
\psi_q^\lambda(\qew C'_w(q)) 
= [\lambda_1]_q \cdots [\lambda_r]_q \sum_U q^{\inv(U_1 \circ \cdots \circ U_r)},
\end{equation}
where the sum is over all left-anchored cylindrical $F_w$-tableaux 
of shape $\lambda$.
\end{thm}
\begin{proof}
Rewrite the 
sums 
above 
as in 
(\ref{eq:sumxu}) and factor the resulting inner sums
as in
Proposition~\ref{p:multisum}.
By Corollary~\ref{c:SOF}, 
the right-hand sides of 
(\ref{eq:qpsi1}), (\ref{eq:qpsi2}) are both equal to
\begin{equation*}
[\lambda_1]_q \cdots [\lambda_r]_q
\sumsb{I \vdash [n]\\ \type(I) = \lambda} 
q^{\inv(V(F_w,I)_1 \circ \cdots \circ V(F_w,I)_r)}
O(F_w|_{I_1}) \cdots O(F_w|_{I_r}).
\end{equation*}
By Theorem~\ref{t:qpsiO}, this is $\psi_q^\lambda(\qew C'_w(q))$.
\end{proof}

For example, consider the descending star network 
$F_{3421}$ in (\ref{eq:F3421}) 
and the sum in (\ref{eq:qpsi1}).
It is easy to verify that there are eighteen 
cylindrical $F_{3421}$-tableaux of shape $31$.
Four of these are
\begin{equation}\label{eq:cyltabs31}
\raisebox{-6mm}{\includegraphics[height=12mm]{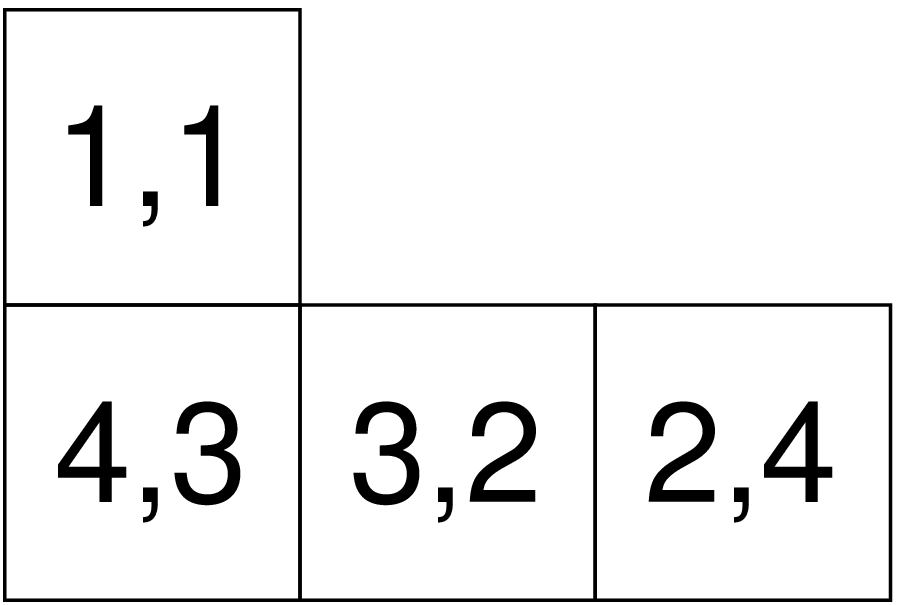}}\ ,\quad
\raisebox{-6mm}{\includegraphics[height=12mm]{tableaux/31_22133441}}\ ,\quad
\raisebox{-6mm}{\includegraphics[height=12mm]{tableaux/31_33122441}}\ ,\quad
\raisebox{-6mm}{\includegraphics[height=12mm]{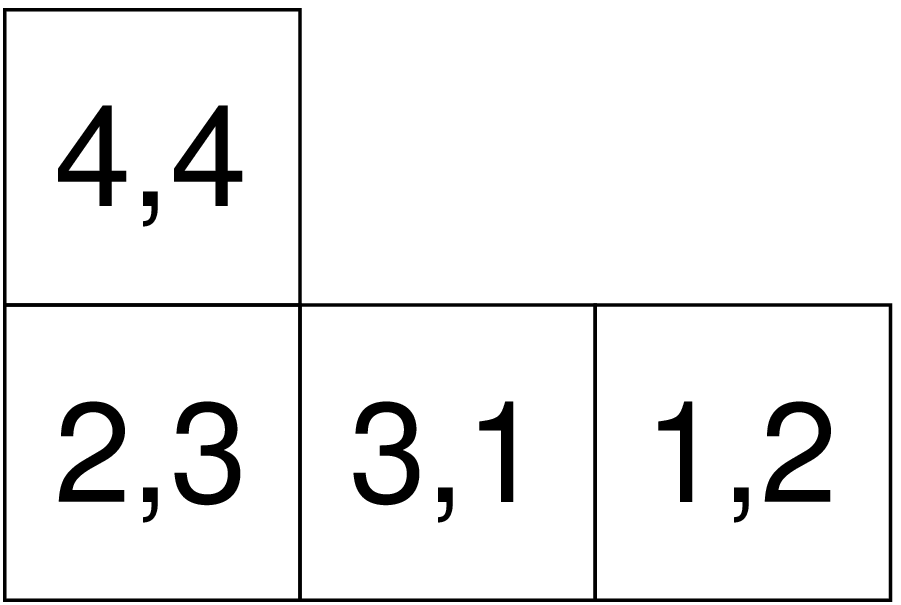}}\ ,
\end{equation}
where $i,j$ represents the unique path from source $i$ to sink $j$.
These tableaux $U$ of shape $31$
yield tableaux $U_1 \circ U_2$ of shape $4$,
\begin{equation}\label{eq:cyltabs4}
\raisebox{-3mm}{\includegraphics[height=6mm]{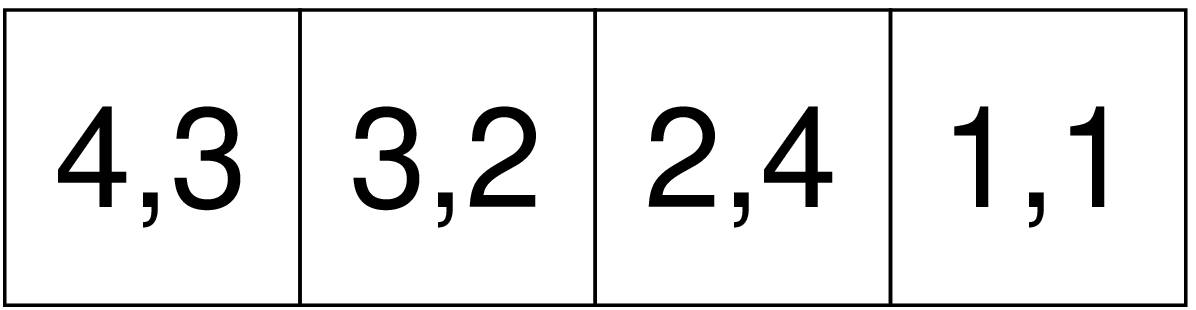}}\ ,\quad
\raisebox{-3mm}{\includegraphics[height=6mm]{tableaux/4_13344122}}\ ,\quad
\raisebox{-3mm}{\includegraphics[height=6mm]{tableaux/4_12244133}}\ ,\quad
\raisebox{-3mm}{\includegraphics[height=6mm]{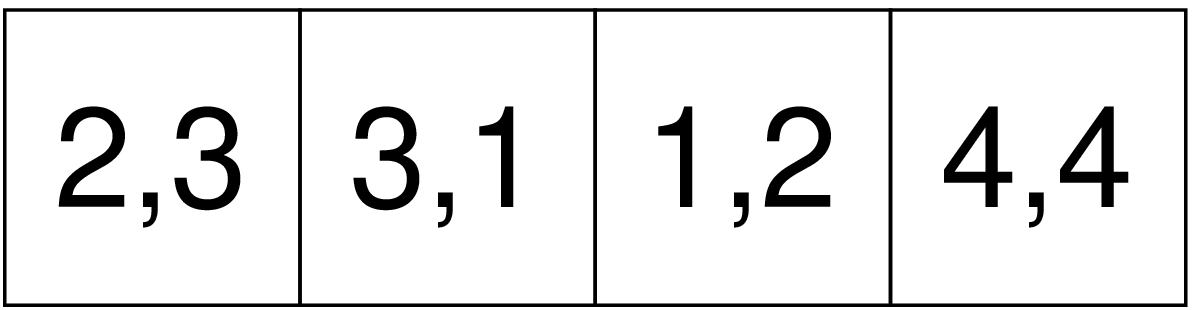}}\ ,
\end{equation}
which have $5$, $2$, $1$, and $2$ inversions, respectively.
Together, 
they
contribute $q + 2q^2 + q^5$ 
to $\psi_q^{31}(q_{e,3421} C'_{3421}(q)) = 1 + 3q + 5q^2 + 5q^3 + 3q^4 + q^5$.
Now consider the sum in (\ref{eq:qpsi2}).
It is easy to verify that there are six left-anchored 
cylindrical $F_{3421}$-tableaux of shape $31$: the second and third tableaux
in (\ref{eq:cyltabs31}) and the four tableaux
\begin{equation*}
\raisebox{-6mm}{\includegraphics[height=12mm]{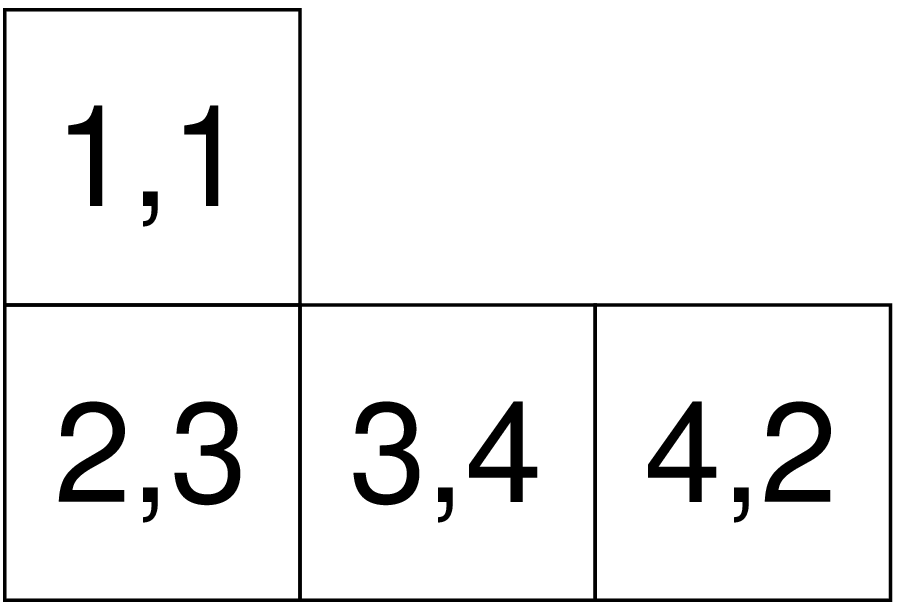}}\ ,\quad
\raisebox{-6mm}{\includegraphics[height=12mm]{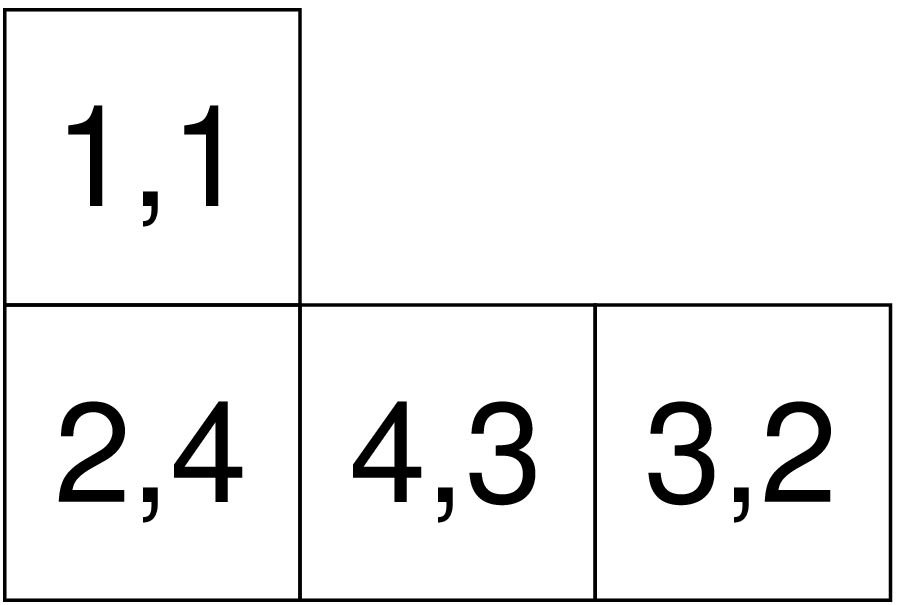}}\ ,\quad
\raisebox{-6mm}{\includegraphics[height=12mm]{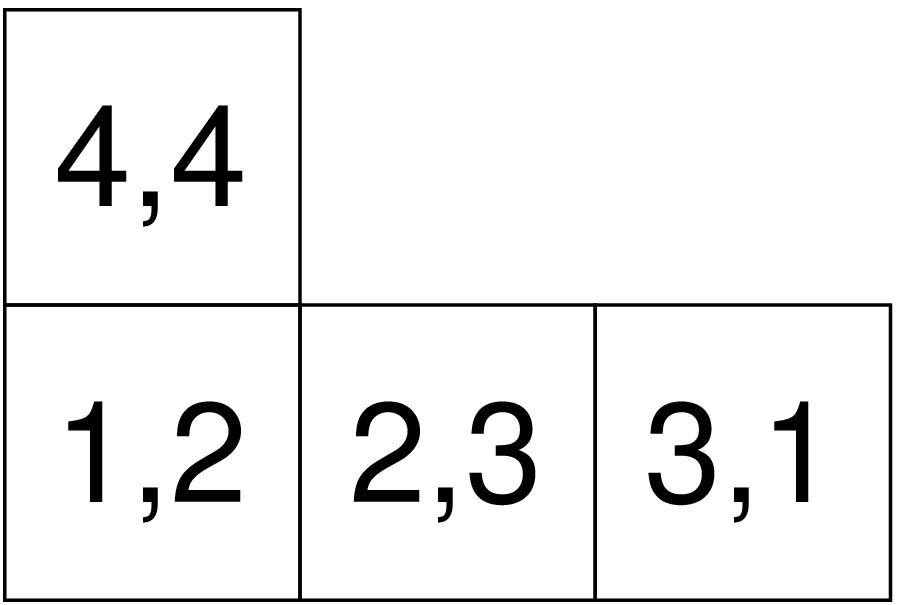}}\ ,\quad
\raisebox{-6mm}{\includegraphics[height=12mm]{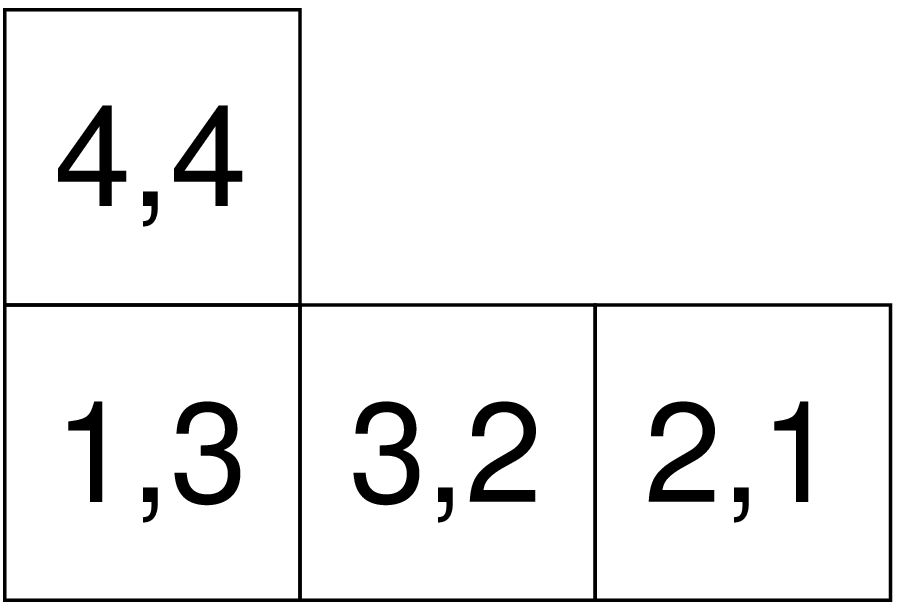}}\ .
\end{equation*}
These tableaux $U$ of shape $31$
yield six tableaux $U_1 \circ U_2$ of shape $4$: 
the second and third tableaux in (\ref{eq:cyltabs4}) and the four tableaux
\begin{equation*}
\raisebox{-3mm}{\includegraphics[height=6mm]{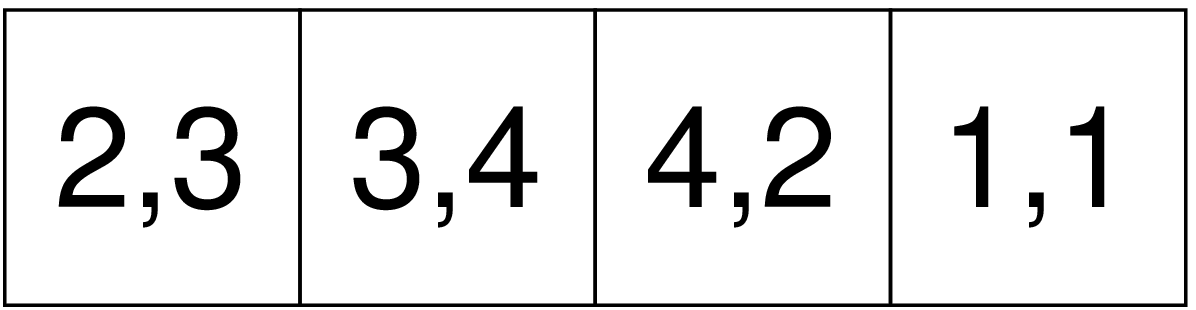}}\ ,\quad
\raisebox{-3mm}{\includegraphics[height=6mm]{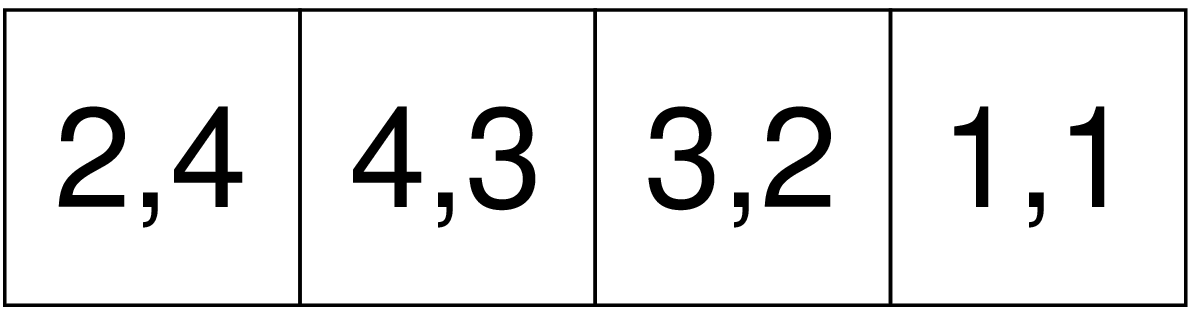}}\ ,\quad
\raisebox{-3mm}{\includegraphics[height=6mm]{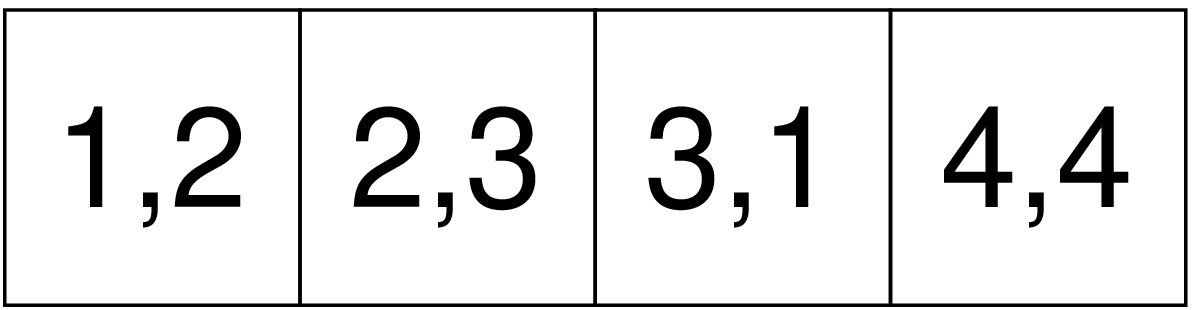}}\ ,\quad
\raisebox{-3mm}{\includegraphics[height=6mm]{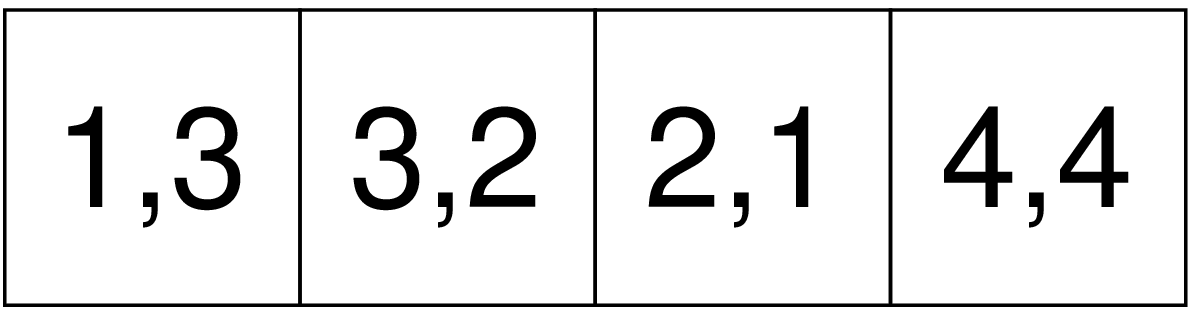}}\ ,
\end{equation*}
which have $2$, $1$, $2$, $3$, $0$, and $1$ inversions, respectively.
Together, 
the six tableaux 
contribute $1 + 2q + 2q^2 + q^3$
to 
$[3]_q [1]_q (1 + 2q + 2q^2 + q^3) 
= \psi_q^{31}(q_{e,3421} C'_{3421}(q))$. 

As a consequence of Theorem~\ref{t:qpsi}, we now have the following analog
of Corollary~\ref{c:edeck}.
Say that a partition $\lambda$ of $n$ is a {\em refinement}
of another partition (or composition) $\mu$ of $n$ if 
$\lambda$ can be obtained from $\mu$ by replacing each part $\mu_i$ 
by an integer partition of $\mu_i$
and sorting the results into weakly decreasing order.

\begin{cor}\label{c:prefine}
Let $w \in \sn$ \avoidp and let the component sizes of $F_w$ 
be $\mu = (\mu_1, \dotsc, \mu_s)$. 
Then we have
\begin{equation*}
\psi_q^\lambda (\qew C'_w(q)) = 0.
\end{equation*}
if and only if $\lambda$ is not a refinement of $\mu$.
\end{cor}
\begin{proof}
It is clear that if $\lambda$ is not a refinement of $\mu$,
then there is no cylindrical $F_w$-tableau of shape $\lambda$.
Therefore by Theorem~\ref{t:qpsi} we have that $\psi_q^\lambda(\qew C'_w(q)) = 0$.
Suppose on the other hand that $\lambda = (\lambda_1, \dotsc, \lambda_r)$ 
is a refinement of $\mu$, and let $J_1, \dotsc, J_s$ be the
subintervals of $[n]$ corresponding to the connected components of $F_w$.
Then there exists an ordered set partition
$I = (I_1, \dotsc, I_r) \vdash [n]$,
whose type is a rearrangement of $\lambda$,
such that each block of $J$ is a union of several consecutive blocks of $I$.
Let $\pi = (\pi_1, \dotsc, \pi_n)$ be the unique path family of type $e$
covering $F_w$.
It is clear now that we can construct at least one record-free, row-semistrict
$F_w$ tableau of type $e$ and shape $\lambda$, by creating a row
containing the paths $\pi_a, \dotsc, \pi_b$ (in order)
for each block $[a,b]$ of $I$.
By Theorem~\ref{t:swqpsi}, we therefore have
$\psi_q^\lambda (\qew C'_w(q)) \neq 0$.
\end{proof}

It would be interesting to extend Theorem~\ref{t:qpsi}
to include a $q$-analog of Theorem~\ref{t:sncfinterp} (iv-b).
In particular the fourth identity in Equation~(\ref{eq:moncfdef})
suggests that an answer to Problem~\ref{p:monevalinterp} and its $q$-analog
are related to a set partition of tableaux counted by $\psi^\lambda$.
It is not clear whether such a partition is more easily expressed
in terms of 
record-free, row-semistrict $F_w$-tableaux of type $e$,
right-anchored, row-semistrict $F_w$-tableaux of type $e$,
cylindrical $F_w$-tableaux, 
left-anchored cylindrical $F_w$-tableaux, or
cyclically row-semistrict $F_w$-tableaux of type $e$.

\begin{prob}
Find a statistic $\STAT$ on $F$-tableaux such that we have
\begin{equation*}
\psi_q^\lambda(\qew C'_w(q)) = \sum_U q^{\STAT(U)},
\end{equation*}
where the sum is over all cyclically row-semistrict $F_w$-tableaux 
of type $e$ and shape $\lambda$.
\end{prob}

It would also be interesting to show that all of the stated
interpretations
of $\psi_q^\lambda(\qew C'_w(q))$ remain valid if we reverse the order of 
concatenating rows of tableaux: this may turn out to be the easiest way
to link the $\hnq$-traces $\psi_q^\lambda$ and $\phi_q^\mu$.

\begin{prob}
Show that Theorems~\ref{t:swqpsi}, \ref{t:qpsiO}, \ref{t:qpsi}
remain valid if one replaces the numbers
$\inv(U_1 \circ \cdots \circ U_r)$,
$\inv(U_1^R \circ \cdots \circ U_r^R)$,
$\inv(V(F_w,I)_1 \circ \cdots \circ V(F_w,I)_r)$,
with
$\inv(U_r \circ \cdots \circ U_1)$, 
$\inv(U_r^R \circ \cdots \circ U_1^R)$,
$\inv(V(F_w,I)_r \circ \cdots \circ V(F_w,I)_1)$,
respectively.
\end{prob}

\section{Results concerning $\phi_q^\lambda(\qew C'_w(q))$}\label{s:hnqinterpm}


Recall that the component statements of Theorem~\ref{t:sncfinterp} pertaining
to monomial traces are weaker 
than those pertaining to other traces.
To state a $q$-analog of Theorem~\ref{t:sncfinterp} (v-a), we will use
several partial orders, including majorization and refinement of integer
partitions.  We will use the symbol $\unlhd$ to denote majorization and 
$\leq_R$ to denote refinement, as defined before Corollary~\ref{c:prefine}.
We begin by stating an analog of Corollaries \ref{c:edeck}, \ref{c:sdeck}.


\begin{prop}\label{p:mnodec}
Fix $\lambda = (\lambda_1, \dotsc, \lambda_r) \vdash n$
and $w \in \sn$ \avoidingp.  If $w_1 \cdots w_n$ has 
a decreasing subsequence of length greater than $\lambda_1$, then 
we have $\phi_q^\lambda(\qew C'_w(q)) = 0$.
\end{prop}
\begin{proof}
Let $w \in \sn$ have a decreasing subsequence of 
length greater than $\lambda_1$ and recall that there exist integers 
$\{a_{\lambda,\mu} \,|\, \lambda,\mu \vdash n\}$ such that
\begin{equation*}
\phi_q^\lambda = \sum_{\mu \unrhd \lambda^\tr} a_{\lambda,\mu} \epsilon_q^\mu.
\end{equation*}
If $\phi_q^\lambda(\qew C'_w(q)) \neq 0$ then some partition $\mu$
in the above sum satisfies $\epsilon_q^\mu(\qew C'_w(q)) \neq 0$.
But the number of parts of $\mu$ is necessarily less than or equal to 
$\lambda_1$.  This contradicts
Corollary~\ref{c:edeck}.
\end{proof}






Similarly, we have a partial analog of Corollary~\ref{c:prefine}.
\begin{prop}\label{p:mdeck}
Let $w \in \sn$ \avoidp{} and let the component sizes of $F_w$ 
(in weakly decreasing order) be
$\mu = (\mu_1, \dotsc, \mu_r)$. 
Then for each partition $\lambda \vdash n$ not refining $\mu$ we have
\begin{equation*}
\phi_q^\lambda (\qew C'_w(q)) = 0.
\end{equation*}
\end{prop}
\begin{proof}
Observe that we may rewrite the last equation in (\ref{eq:moncfdef}) as
\begin{equation*}
\psi_q^\lambda = \sum_{\nu \geq_R \lambda} L_{\lambda,\nu} \phi_q^\nu,
\end{equation*}
since no row-constant Young tableau of shape $\lambda$ 
has content $\nu$ unless $\lambda$ refines $\nu$.  
Inverting the matrix $(L_{\lambda,\nu})_{\lambda,\nu \vdash n}$
and evaluating traces at $\qew C'_w(q)$ we have
\begin{equation}\label{eq:moncfinv}
\phi_q^\lambda(\qew C'_w(q)) = 
\sum_{\nu \geq_R \lambda} L_{\lambda,\nu}^{-1} \psi_q^\lambda(\qew C'_w(q)).
\end{equation}
Now suppose that we have $\lambda \not \leq_R \mu$.
Then each partition $\nu$ in (\ref{eq:moncfinv}) 
satisfies $\nu \not \leq_R \mu$.
By Corollary~\ref{c:prefine}, each term on the 
right-hand side of (\ref{eq:moncfinv}) is zero.
\end{proof}

We remark that Propositions~\ref{p:mnodec}, \ref{p:mdeck} are not new;
they follow from 
\cite[Prop.\,4.1\,(3),\,(2)]{HaimanHecke}, respectively.
For more facts about these evaluations,
see
\cite[Sec\,4]{HaimanHecke}.

Now we complete the proof of
Theorem~\ref{t:sncfinterp} (v-a)
and provide a $q$-analog of this result.
Let $\mathcal T_C(F_w,\mu)$ denote the set of column-strict $F_w$-tableaux of
shape $\mu$ and type $e$.
\begin{thm}\label{t:qphipartial}
Let $w \in \sn$ \avoidp.  For $\lambda_1 \leq 2$ we have
\begin{equation}\label{eq:qphipartial}
\phi_q^\lambda(\qew C'_w(q)) = 
\begin{cases}
\displaystyle{\sum_{U \in \mathcal T_C(F_w,\lambda)}} q^{\inv(U)} 
&\text{if for all $\mu \triangleleft \lambda$ we have
$\mathcal T_C(F_w,\mu) = \emptyset$},\\  
0 &\text{otherwise}.
\end{cases}
\end{equation}
\end{thm}
\begin{proof}
Clearly the claim is true for $\lambda = 1^n$, since 
$\phi_q^{1^n} = \epsilon_q^{(n)}$
and the claimed formula coincides with that in Theorem~\ref{t:qepsilon}.
Suppose 
that the claim holds for 
$\lambda = 21^{n-2}, \dotsc, 2^{k-1}1^{n-2k+2}$
and consider the case $\lambda = 2^k 1^{n-2k}$ 
($k \leq \lfloor \tfrac n2 \rfloor$).
Then we have
\begin{equation}\label{eq:em12}
\epsilon_q^{(n-k,k)}(\qew C'_w(q)) 
= \sum_{i = 0}^{k} M_{2^k 1^{n-2k}, 2^i 1^{n-2i}} \phi_q^{2^i 1^{n-2i}} (\qew C'_w(q)),
\end{equation}
where $M_{\lambda,\mu}$ is the number of column-strict Young tableaux
of shape $\lambda$ and content $\mu$.
It is easy to see that
$M_{2^k 1^{n-2k}, 2^i 1^{n-2i}}$ is equal to 
$\tbinom{n-2i}{k-i}$.
By Theorem~\ref{t:qepsilon}, the left-hand side of (\ref{eq:em12}) is
the sum of $q^{\inv(U)}$ over $U \in \mathcal T_C(F_w,2^k 1^{n-2k})$.

If $w$ has a decreasing subsequence of length three,
then by Proposition~\ref{p:mnodec}, 
the left-hand side of (\ref{eq:qphipartial}) is $0$.
By 
the proof of Corollary~\ref{c:edeck},
we have $\mathcal T_C(F_w, \mu) = \emptyset$ for all 
$\mu \trianglelefteq \lambda$, and 
the right-hand side of (\ref{eq:qphipartial}) is $0$ as well.

Assume therefore that $w$ avoids the pattern $321$.
By Theorem~\ref{t:kstar}, every connected component of $F_w$ induces
a subposet of $P(F_w)$ which is isomorphic to $P(H_k)$ where 
\begin{equation*}
H_k = \raisebox{-12mm}{
\includegraphics[width=40mm]{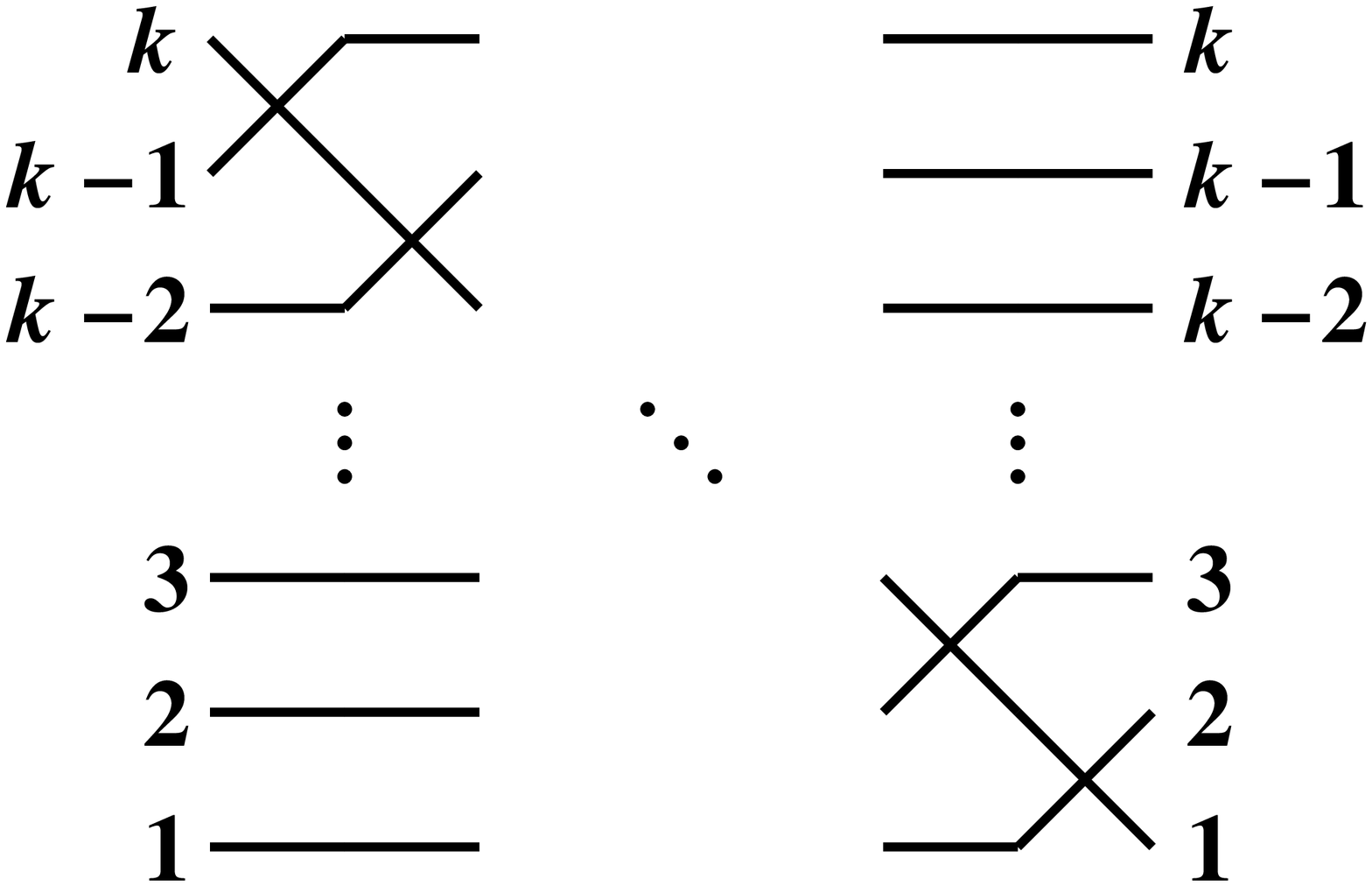}}.
\end{equation*}
Let $b = b(w)$ be the number of odd components of $F_w$.  
Then it is possible to construct an $F_w$-tableau which has
$b$ more paths in column $1$ than it has in column $2$, 
but it is not possible to construct an $F_w$-tableau 
for which this difference is greater than $b$.
That is, for $j = \frac{n-b}2$
we have $\mathcal T_C(F_w, 2^j 1^{n-2j}) \neq \emptyset$
while
\begin{equation}\label{eq:emptysets}
\mathcal T_C(F_w, 1^{n}) = 
\mathcal T_C(F_w, 2 1^{n-2}) = \cdots =
\mathcal T_C(F_w, 2^{j-1} 1^{n-2j+2}) = \emptyset.
\end{equation} 
Fix $U \in \mathcal T_C(F_w, 2^j 1^{n-2j})$,
and let the interval $[p_1, p_{2m+1}]$
of sources and sinks define an odd component of $F_w$.
Then paths indexed by $p_1, p_3, \dotsc, p_{2m+1}$ belong
to the first column of $U$ while those indexed by $p_2,\dotsc,p_{2m}$
belong to the second column.  
Note that swapping the columns of the two sets of paths creates a 
valid column-strict $F_w$-tableau of shape $2^{j+1}1^{n-2j-2}$ 
which has the same number of inversions as $U$.
Thus for each such tableau $U$, we may create a column-strict 
$F_w$-tableau $U'$ of shape $2^k 1^{n-2k}$ 
by choosing $k - j$ of the odd components and 
swapping the columns of the even and odd indexed paths within these components.
There are $\tbinom{n-2j}{k-j}$ ways to do this.
Conversely, every tableau $U' \in \mathcal T_C(F_w, 2^k 1^{n-2k})$ 
arises in this way.  
Thus we have
\begin{equation}\label{eq:emconnect}
\epsilon_q^{(n-k,k)}(\qew C'_w(q)) = 
\tbinom{n-2j}{k-j}
\sum_{U \in \mathcal T_C(F_w,2^j1^{n-2j})} q^{\inv(U)}.
\end{equation}

If $j < k$, then we have by induction and (\ref{eq:emptysets})
that
$\phi_q^{2^{i} 1^{n-2i}}(\qew C'_w(q)) = 0$ for $i < k$, $i \neq j$.
Now 
(\ref{eq:emconnect}) implies that
$\phi_q^{2^k 1^{n-2k}}(\qew C'_w(q)) = 0$, and the claim is true.
If $j = k$, then we have by (\ref{eq:em12}) -- (\ref{eq:emptysets}) 
that
$\epsilon_q^{(n-k,k)}(\qew C'_w(q)) = \phi_q^{2^k 1^{n-2k}}(\qew C'_w(q))$
and by (\ref{eq:emconnect}) the claim again is true.
\end{proof}

We remark that the obvious $q$-analogs of Theorem~\ref{t:sncfinterp} (v-b) are
false.  
Consider the permutation $w = 3142$
and the evaluation $\phi_q^{22}(\qew C'_w(q)) = q + q^2$.
$F_w$ is the penultimate zig-zag network 
in (\ref{eq:xfigureszz}), and 
there are two column-strict cylindrical $F_w$-tableaux of shape $22$,
\begin{equation*}\label{eq:cscyltabs22}
\raisebox{-6mm}{\includegraphics[height=12mm]{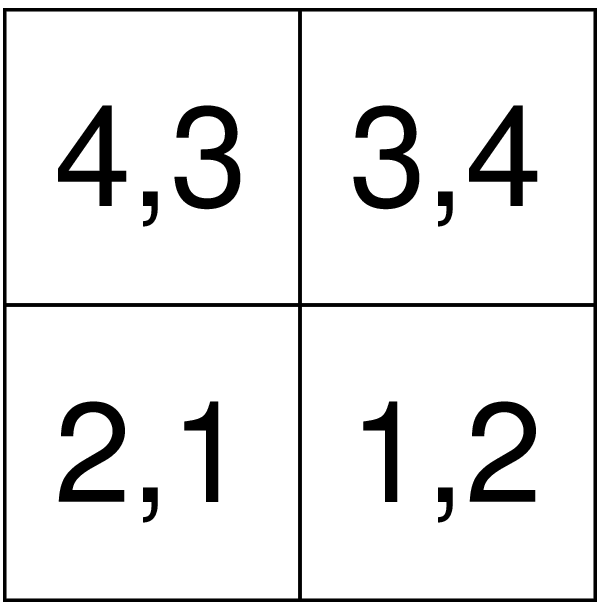}}\ ,\quad
\raisebox{-6mm}{\includegraphics[height=12mm]{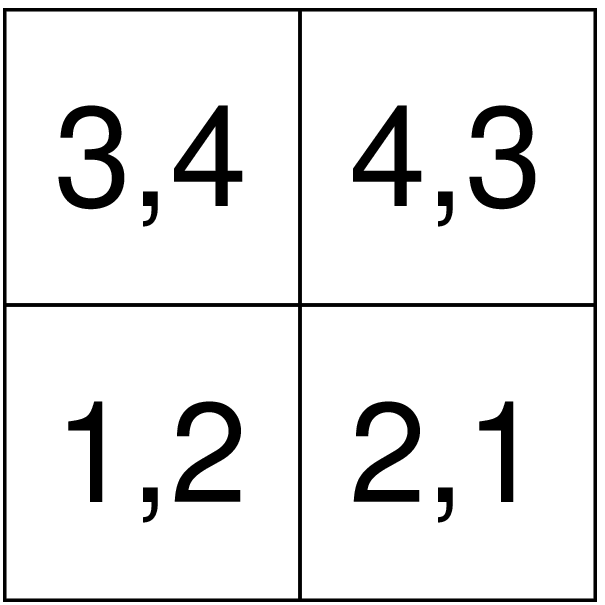}}\ .
\end{equation*}
Unfortunately, as $U$ varies over these tableaux we have
\begin{equation*}
\sum_U q^{\inv(U)} = 1 + q^3,
\qquad
\sum_U q^{\inv(U_1 \circ U_2)} = 1 + q^2,
\qquad
\sum_U q^{\inv(U_2 \circ U_1)} = q + q^3.
\end{equation*}
Perhaps a correct $q$-analog of Theorem~\ref{t:sncfinterp} (v-b)
would help with the formulation of an interpretation of
$\phi_q^\lambda(\qew C'_w(q))$ when $w$ \avoidsp.
Given Theorems \ref{t:qeta}, \ref{t:qepsilon}, \ref{t:qchi}, \ref{t:swqpsi}, 
\ref{t:qpsiO}, and \ref{t:qpsi},
it seems reasonable to hope that $F_w$-tableaux can 
play an important role in such an interpretation.
\begin{prob}\label{p:qphi}
Find a property $X$ of $F_w$-tableaux and a statistic $\STAT$ 
such that for $\lambda \vdash n$
and $w$ \avoidingp\ we have
\begin{equation*}
\phi_q^\lambda(\qew C'_w(q)) = \sum_U q^{\STAT(U)},
\end{equation*}
where the sum is over all $F_w$-tableaux $U$ of shape $\lambda$
having property $X$.
\end{prob}

\section{Acknowledgements}

The authors are grateful to Kaitlyn Peterson, Daniel Studenmund, and Michelle
Wachs for helpful conversations, to Lehigh University, 
the University of Miami, and Universidad de los Andes for financial support
and hospitality.  
The authors are also grateful to an anonymous referee for helpful suggestions.


\end{document}